%% file: main.tex
\documentclass[12pt,a4paper]{amsart}
\usepackage[dvipsnames]{xcolor}
\usepackage{amsfonts}
\usepackage{amscd}
\usepackage[utf8]{inputenc}
\usepackage[T1]{fontenc}
\usepackage{placeins}

\usepackage{t1enc}
\usepackage[mathscr]{eucal}
\usepackage{indentfirst}
\usepackage{graphicx}
\usepackage{graphics}
\usepackage{pict2e}
\usepackage{epic}
\numberwithin{equation}{section}
\usepackage[margin=2.6cm]{geometry}
\usepackage{epstopdf} 
\usepackage{amsmath}
\allowdisplaybreaks
\usepackage{amssymb}
\usepackage{amsthm}
\usepackage{mdframed}
\usepackage{bbm}
\usepackage{tikz}
\usepackage{paralist}
\usepackage[onehalfspacing]{setspace}
\usepackage{amsfonts}
\usepackage{color}
\usepackage{nicefrac}
\usepackage{mathrsfs}
\usepackage{multirow}
\usepackage{mathtools}
\usepackage{tikz-cd}

\newcounter{dummy}
\usepackage{hyperref}
\usepackage{enumitem}
\makeatletter
\newcommand\myitem[1][]{\item[#1]\refstepcounter{dummy}\def\@currentlabel{#1}}
\makeatother
\newtheorem{thm}{Theorem}
\numberwithin{thm}{section}
\newtheorem{lemma}[thm]{Lemma}

\newtheorem{definition}[thm]{Definition}
\newtheorem{coro}[thm]{Corollary}

\newtheorem*{thm*}{Theorem}
\newtheorem*{prop*}{Proposition}
\numberwithin{equation}{section}
\theoremstyle{remark}
\newtheorem{remark}[thm]{Remark}

\newtheoremstyle{algo}%
{-6pt}{-6pt}%
{\upshape}{}%
{\scshape}{}
{\newline}{}
\theoremstyle{algo}
\newmdtheoremenv{problem}[thm]{Problem}
\newmdtheoremenv{algo}[thm]{Algortihm}

\input{abbrev}

\begin{document}
\title[On Shnirelman's inequality]{On incompressible flows in discrete networks and Shnirelman's inequality}
\author[Schiffer \& Zizza]{Stefan Schiffer and Martina Zizza}
\address{Max-Planck Institute for Mathematics in the Sciences} 
\email{stefan.schiffer@mis.mpg.de}
\email{martina.zizza@mis.mpg.de}
\begin{abstract}
Let $f$ and $g$ be two volume-preserving diffeomorphisms on the cube $Q=[0,1]^{\nu}$, $\nu \geq 3$. We show that there is a divergence-free vector field $v \in L^1((0,1);L^p(Q))$ such that $v$ connects $f$ and $g$ through the corresponding flow and 
$\Vert v \Vert_{L^1_t L^p_x} \leq  C_{p,\nu} \Vert f- g \Vert_{L^p_x}$. In particular we show Shnirelman's inequality, cf. [Shnirelman, Generalized fluid flows, their approximation and applications (1994)], for the optimal H\"older exponent $\alpha =1$,  thus proving that the metric on the group of volume-preserving diffeomorphisms of $Q$ is equivalent to the $L^2$-distance.
To achieve this, we discretise our problem, use some results on flows in discrete networks and then construct a flow in non-discrete space-time out of the discrete solution.
\end{abstract}
\maketitle
% \section*{TODO List}
% \begin{itemize}
%     \item British English
%     \item Function spaces $L^2(\Omega,\R^{\nu})$;
%     \item Order: $(e,\gamma)$
%     \item How to do cases: Commas at the end of each line
%     \item What to write large and what to write small {\color{red} just equation small}
%     \item good notation for all constants that are not $2$ (or another small number)
% \end{itemize}
\input{sec1}

\input{sec2}
\input{sec3}
\input{sec4}

\input{sec5}

\input{sec6}
% \pagebreak

\bibliography{biblio.bib}
\bibliographystyle{abbrv}

\end{document}

%% file: abbrev.tex
\newcommand{\R}{\mathbb{R}}
\newcommand{\N}{\mathbb{N}}

\newcommand{\Q}{\mathscr{Q}}

\newcommand{\dx}{\,\textup{d}x}

\newcommand{\dH}{\,\textup{d}\mathcal{H}}
\newcommand{\Haus}{\mathcal{H}}

\everymath{\displaystyle}

\DeclareMathOperator{\dist}{dist}

\DeclareMathOperator{\divergence}{div}

\DeclareMathOperator{\spt}{spt}

\renewcommand{\phi}{\varphi}

\DeclareMathOperator{\id}{id}
\DeclareMathOperator{\conv}{conv}

\DeclareMathOperator{\BV}{BV}

\newcommand{\diverg}{\divergence}
\newcommand{\q}{\mathrm{Q}}

\newcommand{\Dcal}{\mathcal{D}}
\DeclareMathOperator{\Id}{Id}
\newcommand{\un}{\mathbf{e}}
\newcommand{\e}{{e}}
\newcommand{\Rcal}{\mathcal{R}}
\newcommand{\Qcal}{\mathcal{Q}}
\definecolor{Gump}{rgb}{0,0.6,0.4}
\definecolor{Hanks}{rgb}{0.7,0.3,0.1}

%% file: sec1.tex
\section{Introduction}
We assume that the motion of an ideal incompressible fluid takes place in a vessel $M$, which is a open, bounded and connected subset of $\R^\nu$ (with Euclidean volume element $dx$).  A configuration of the fluid at time $t\in\R$ is a volume-preserving diffeomorphism $\phi_t:M\rightarrow M$. In more detail, for each particle $x\in M$, $\phi_t(x)\in M$ denotes the position reached by the particle  $x\in M$ at time $t$ and $t \mapsto \phi_t(x)$ is its \emph{trajectory}. Together with the operation of composition, the set of smooth volume-preserving diffeomorphism forms a group that we call $\text{SDiff}(M)$, or simply $\mathcal{D}(M)$. The group $\mathcal{D}(M)$ is a metric space, with metric given by
\begin{equation}\label{eq:distance}
    \text{dist}_{\mathcal{D}(
    M)}(f,g)=\underset{ \phi_0=f,\phi_1=g}{\inf_{\lbrace \phi_t\rbrace\subset\mathcal{D}(M), }} \ell\lbrace \phi_t\rbrace_0^1,
\end{equation}
where $\ell$ is the \emph{length functional}
\begin{equation}\label{def:length:functional:cont}
    \mathcal{\ell}\lbrace\phi_t\rbrace_{0}^{1}= \begin{cases}\int_{0}^{1} \|\dot\phi_t\|_{L^2(M)}dt & \text{if } \phi \in W^{1,1}((0,1);L^2(M)), \\
    \infty & \text{else.}
    \end{cases}
\end{equation}
Meanwhile, we also may define the \emph{action functional} for differentiable paths as 
%While, the \emph{action functional} is defined as
\begin{equation}
    \label{eq:action}
    \mathcal{A}\lbrace\varphi_t\rbrace_{0}^{1}=\int_{0}^{1}\|\dot\varphi_t\|_2^2dt.
\end{equation}
The metric space $(\mathcal{D}(M),\text{dist}_{\mathcal{D}(M)})$ has been studied deeply, with a focus on questions regarding the finiteness of its diameter \cite{Shnirelman,Shnirelman2}, or the existence of the shortest path, namely a path $t\rightarrow\phi_t$ in $\Dcal(M)$ connecting two volume-preserving diffeomorphisms minimizing the length functional \cite{Shnirelman}. We remark that, in order to solve the problem of the existence of the shortest path, the notion of Generalised Flows  as minimisers for the action functional has been introduced in \cite{Brenier1} and further studied in \cite{AF}. For a detailed survey of these results, we refer to  \cite[Chapter IV]{Arnold:khesin}. 

 \medskip
 
 In what follows, we will study the case $M=[0,1]^\nu$,  where we set the space dimension $\nu$ to be larger than $2$ to avoid topological issues. We further comment on $\nu=2$ in Section \ref{sec:lowD}.
 %with $\nu\geq 3$ avoiding topological issues on the set $Q$ (the case $\nu=2$ is different as we will explain in the subsequent pages, see in particular Section \ref{sec:1.5}). 
 The space $\mathcal{D}(M)$ is embedded in the space $L^2(M,\R^\nu)$ of vector functions on $M$. Therefore, for every two configurations $f,g\in\mathcal{D}(M)$ we can study the comparison of the two metrics $\text{dist}_{\mathcal{D}(M)}$ and $\|f-g\|_{L^2(M)}.$ It is straightforward to see that
\begin{equation*}
    \|f-g\|_{L^2(M)}\leq \text{dist}_{\mathcal{D}(M)}(f,g),
\end{equation*}
thus, the natural question  is if it is possible to estimate $\text{dist}_{\mathcal{D}(M)}$ with the $L^2$-distance.
In \cite{Shnirelman2} the following theorem is proved.
\begin{thm}[Shnirelman's inequality \cite{Shnirelman2}]\label{thm:shnirelman}
    Let $\nu\geq 3$ and $M=[0,1]^{\nu}$. There exists $C>0$ such that for $\alpha =\frac{2}{\nu+4}$  and for every $f,g\in\mathcal{D}(M)$,
    \begin{equation}\label{eq:ineq:shnirelman}
        \text{dist}_{\mathcal{D}(M)}(f,g)\leq C\|f-g\|_{L^2(M)}^\alpha.
    \end{equation}
\end{thm}
We point out that this result has been obtained via the concept of Generalised Flows (see also \cite{Brenier1}), while a discrete type result has been presented in \cite{Shnirelman} and \cite{Zizza24}, with a smaller value of $\alpha$. 

\medskip
In particular, Theorem \ref{thm:shnirelman} shows that the topological spaces induced by the metric $\dist_{\Dcal(M)}$ and the $L^2$ distance are equivalent on $\Dcal(M)$. 
The aim of this paper is to prove that the two metrics are \emph{equivalent}. 

\medskip 

Our result is built using a discrete-type approach, where the discretisation of the elements of $\text{SDiff}(M)$ is given by the permutations of subcubes of $M$ (cf. \cite{Lax}).
More precisely, consider $\mathcal{R}_N$ a tiling of the unit cube $M$ into identical subcubes of sidelength $N^{-1}$. To fix notation, we will write the small cubes as \[
 \q_v:= N^{-1}v + N^{-1} [0,1)^{\nu},
 \]
 where $v \in V = \{0,1,\ldots,N-1\}^{\nu}$ is a multi-index. Then, for $\sigma \colon V \to V$ bijective, a permutation of cubes of $\mathcal{R}_N$ is a  map $\psi_\sigma:M\rightarrow M$ with the following property:  
\begin{equation} \label{def:psisigma:intro}
    \psi_{\sigma} (x)= N^{-1} \sigma(v) + (x- N^{-1} v) \quad \text{if } x \in \q_v,~v \in V. 
\end{equation}
We denote by $\mathcal{D}_N$ the set of permutations of the tiling $\Rcal_N$.
Observe that the induced map $\psi_{\sigma}$ \emph{does not} have the property of mapping connected sets into connected sets. Therefore, the permutation can be realised as a flow map of a divergence-free velcocity field that has \emph{only} $\BV$ regularity in space.
 We remark that if $u\in L^1([0,1],\BV(M))$ is a divergence-free vector field, then there exists unique flow $\psi\in C([0,1], L^1(M,M))$ such that, for a.e. $x\in M$, the function $\psi$ is an integral solution of
 \begin{equation*}
     \begin{cases}
         \dot\psi(t,x)=u(t,\psi(t,x)), \\
         \psi(0,\cdot)=\Id_M.
     \end{cases}
 \end{equation*}
 This unique flow is also called the \emph{Regular Lagrangian flow}.  Existence, uniqueness and stability properties of Regular Lagrangian Flows for $\BV$ vector fields have been established in \cite{Ambrosio:BV}. 

 Moreover, observe that  $\psi(t,x)_\sharp\mathcal{L}^\nu=\mathcal{L}^\nu$, for every $t\in[0,1]$, where we denote by $(\cdot)_\sharp\mathcal{L}^\nu$ the push-forward of the Lebesgue measure on $Q$ (i.e. $\psi(t,\cdot)$ a \emph{measure-preserving} map for every instant of time). With abuse of notation we call $\psi(\cdot)=\psi(1,\cdot)$ the \emph{time-$1$ map} of the vector field $u$. 
Then our main result reads as follows.

\begin{thm}\label{thm:main}
    Let $\nu\geq 3$, $p\in[1,+\infty]$, $q\in[1,+\infty]$, then there exists $C_{p,q,\nu}>0$ such that, for every $\psi_\sigma\in\mathcal{D}_N$  there exists a divergence-free vector field $u\in L^q([0,1],L^p(M))\cap L^1([0,1],\BV(M))$ such that $\psi_\sigma$ is the time-$1$ map of the vector field $u$ and
    \begin{equation}\label{eq:ineq:shnirelman:sharp}
        \|u\|_{L^q_tL^p_x}\leq C_{p,q,\nu}\|\psi_\sigma-\Id_M\|_{L^p_x}.
    \end{equation}
\end{thm}

Using the previous result of Shnirelman (Theorem \ref{thm:shnirelman}), we can extend Theorem \ref{thm:main} to the continuous setting in Section \ref{sect:from:discrete:to:continuous}, thus recovering the original formulation of \cite{Shnirelman2}. Indeed,

\begin{coro}
    [Sharp Shnirelman's Inequality]\label{coro:main}
    Let $\nu\geq 3$, then there exists $C>0$ such that, for every $f,g\in\mathcal{D}(M)$ it holds
    \begin{equation}\label{eq:ineq:shnirelman:sharp:smooth}
        \text{dist}_{\mathcal{D}(M)}(f,g)\leq C\|f-g\|_{L^2(M)}.
    \end{equation}

\end{coro}
We remark that the same inequality holds true for the minimisation of the action functional \eqref{eq:action}.
During the proofs of the paper, we mainly focus on establishing the main results in dimension $\nu=3$. Higher dimensions can be dealt with in very similar fashion. In contrast, as we explain in Subsection \ref{sec:lowD} below, dimension $\nu=2$ is particularly different and our method of construction does not work in this dimension.
 
\subsection*{Outline}
We shortly outline the remainder of the paper. Section \ref{sec:1.5} is devoted to a thorough explanation of the key ideas of the construction. We furthermore comment on some open questions related to the problem at hand.

Section \ref{sec:3} is an essential building block of our construction, as we deal with a discretised setup and solve a flow-problem on the graph (see also Section \ref{sec:1.5}). Sections \ref{sec:4} \& \ref{sec:construction} then in detail deal with the adaptation of the discrete results into a continuous setting for fluid configurations that are independent of one coordinate (i.e. $f(x_1,x_2,x_3) = (\tilde{f}(x_1,x_2),x_3)$). The necessary adaptations for the general case are then discussed in Section \ref{sec:higherD}.

Finally, Section \ref{sect:from:discrete:to:continuous} is then devoted to Corollary \ref{coro:main}, i.e. how Corollary \ref{coro:main} follows from the results of Theorem \ref{thm:main}.

\subsection*{Acknowledgements} The authors thank Alexander Shnirelman and L\'aszl\'o Székelyhidi for proposing the problem. We are also grateful towards Laura Vargas Koch for her insights and suggestions.

\section{Key ideas \& comparison to other approaches} \label{sec:1.5}

\subsection{A discretised setting}

To prove the continuous inequality, we first discretise our setting. 

Let $N\in \N$, then denote by $\mathfrak S_{N,\nu}$ the symmetric group over the set $V=\lbrace 0,1,2,\dots, (N-1)\rbrace^\nu$, and let $\mathcal{R}_N$ be a tiling of $M=[0,1]^\nu$ into $N^\nu$ cubes of side $N^{-1}$. To fix the notation, for $v \in V$ we will write the small cubes as 
$$\q_v:= N^{-1}v + N^{-1} [0,1)^{\nu}.$$ 

Above equation gives rise to a bijective map from $V$ to $\Rcal_N$.

\begin{definition} \label{def:permutation cubes}
 A permutation of cubes of $\Rcal_N$ is a map $\psi \colon M \to M$ defined as follows: There exists $\sigma \in \mathfrak{S}_{N,\nu}$ such that for all $v \in V$
 \[
 \psi(x) = N^{-1} \sigma(v) + (x-N^{-1}v), \quad \text{if } x \in Q_v.
 \]
\end{definition}
We tacitly denote such a map $\psi$ associated to a permutation $\sigma$ by $\psi_{\sigma}$ and write $\Dcal_N$ for the set of all such $\psi_{\sigma}$. When there is no fear of ambiguity we both call $\sigma$ and $\psi_{\sigma}$ \emph{permutations}. Observe that the set $\mathcal{D}_N$ has $(N^{\nu})!$ elements and every element $\psi_{\sigma} \in \Dcal_N$ is measure-preserving, namely $(\psi_\sigma)_\sharp\mathcal{L}^\nu=\mathcal{L}^\nu$, where $\mathcal{L}^\nu$ denotes the Lebesgue measure.

  \subsection{A construction via pipes/tubes}
We start by considering a slightly simpler setting.
Fix a cube $\q_v$ of the tiling $\mathcal{R}_N$ and assume that it is mapped into the cube $\q_{\tilde v}=\psi_\sigma(\q_v)$. Assume then that $\psi_\sigma(\q_w)=\q_w$ for all $w\not=v,\tilde v$, then it is possible to find a volume-preserving flow that sends the cube $\q_v$ into $\q_{\tilde v}$ (and viceversa) without moving the intermediate cubes. 

This is possible making the mass flow through thin pipes (Figure \ref{fig:flow:cubes}). Heuristically, the cost for performing this flow, in terms of the $L^1_tL^2_x$-norm of the vector field, is given by $\sim \left({N^{-\nu}\text{dist}(\q_v,\q_{\tilde v})^2}\right)^{\frac{1}{2}}$, namely the (square-root of the) total volume of the cubes $\sim N^{-\frac{\nu}{2}}$ times the distance of the two cubes. This is precisely the quantity $\|\psi_\sigma-\Id_M\|_{L^2_x}$, up to some positive constant depending only on the dimension $\nu$. Therefore, at least in this simplified setting, the exponent $\alpha \in (0,1]$ of Shnirelman's inequality is exactly $1$. In order to keep this idea to a general permutation $\psi_\sigma\in D_N$, we need to guarantee the two following constraints:
\begin{enumerate}[label=(C\arabic*)]
    \item \label{const:1} there is enough space for all the tubes (see also volume constraint and capacity constraint in \eqref{eq:volume},\eqref{eq:capacity});
    \item \label{const:2} the flow we construct is incompressible. 
\end{enumerate}
Point \ref{const:1} justifies the introduction of a discrete network (Section \ref{sec:3}). The incompressibility constraint, satisfied by the solution in the discrete network, must then be translated back in the continuous context. This is the most technical part of the paper, and it will be developed in Sections \ref{sec:4} and \ref{sec:construction}. Before entering the details of the construction, we point out some differences with previous approaches, in order to show the novelties of this method.
\begin{figure}
    \centering
    \includegraphics[width=0.6\linewidth]{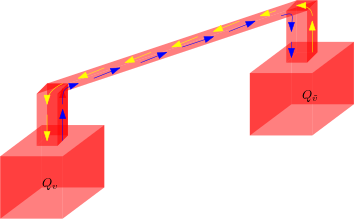}
    \caption{The cube $\q_v$ is swapped with the cube $\q_{\tilde v}$ by the permutation $\psi_\sigma$. Its flow does not influence the intermediate cubes, since the mass travels along thin pipes.}
    \label{fig:flow:cubes}
\end{figure}

\medskip

\subsection{Differences to previous approaches}\label{ss:difference:approaches}
Before coming to a more detailed description of our construction, let us outline previous approaches to the problem. 
\subsubsection*{Generalised Flows} Let us denote by $\Omega=C([0,1];M)$ the set of paths. A Generalised Flow is any probability measure on $\Omega$, that is $q\in\text{Prob}(\Omega)$ \cite{Brenier1}. A Generalised Flow is \emph{incompressible} if 
\begin{equation}\label{def:generalized:flows:incomp}
    \int_\Omega h(\gamma(t))q(d\gamma)=\int_M h(x)dx,\quad\forall h\in C(M), t\in[0,1].
\end{equation}
We denote by $\eta(dxdy)=\delta(y-f(x))dx$, a doubly stochastic probability measure on $M\times M$ where $f\in\Dcal(M)$. We say that an incompressible Generalised Flow $q$ reaches $\eta$ at time $t=1$ if
\begin{equation}\label{eq:gif:connecting:id:f}
\int_\Omega h(\gamma(0),\gamma(1))q(d\gamma)=\int_{M\times M} h(x,y)\eta(dxdy),\quad\forall h\in C(M\times M).
\end{equation}
Informally speaking, this condition replaces the connection of $f\in\Dcal(M)$ with $\Id_M$ via a classic flow.
Incompressible Generalised Flows have been introduced in \cite{Brenier1} in order to prove the existence of a minimiser for the action functional \ref{eq:action} (shortest path problem).

\smallskip

In \cite{Shnirelman2} the concept of Generalised Flows is used in a somehow similar idea to the one of Figure \ref{fig:flow:cubes}, in order achieve the exponent $\alpha = \tfrac{2}{4+\nu}$ (Theorem \ref{thm:shnirelman}). We outline the idea of the construction, without entering into details. 
Fix $f\in\Dcal(M)$, we aim to construct a generalised incompressible flow connecting $\Id_M$ to $f$ (in the sense of equation \eqref{eq:gif:connecting:id:f}). The flow is performed as follows: each particle $x\in M$ is inflated to a continuum of particles $B_{\varepsilon}(x)$ for some $\varepsilon>0$. After this operation, the flow moves the ball to $B_{\varepsilon}(f(x))$ and the ball is inflated back to the particle $f(x)$. The disadvantage here is that this Generalised Flow is not incompressible. In order to respect the incompressibility constraint, a corrective flow must be introduced, via Moser's Theorem \cite{Moser}. The cost of the correction flow affects the value of $\alpha$ that could not be increased up to $\tfrac{2}{4+\nu}$. 

\medskip

In our approach, we transfer the mass of a cube $Q$ into $\psi_\sigma(Q)$, but we do not need to introduce Generalised Flows, as the flow we construct is 'classic' in the sense of well-defined trajectories. Moreover, we can guarantee the incompressibility condition, making the mass travel along pipes.
%\begin{figure}
 %   \centering
    %\includegraphics[width=0.6\linewidth]{shn1994.png}
   % \caption{In \cite{Shnirelman} each particle $x\in M$ is inflated in a continuum of particles $B_{\epsilon}(x)$, then the ball is translated to the new position $B_{\epsilon}(f(x))$ and inflated back to the particle $f(x)$. This does not give an incompressible Generalised Flow.}
  %  \label{fig:shn1994}
%\end{figure}
\subsubsection*{Swap Distance and discretised context} In \cite{Shnirelman} another proof of Inequality \ref{eq:ineq:shnirelman} was given, in the framework of the discrete setting we introduced at the beginning of this section.
In \cite{Shnirelman} 
a discrete flow $\lbrace \psi_{\sigma_j}\rbrace_{j=1}^{\mathcal T}$ is a sequence of \emph{elementary movements} in \emph{duration} $\mathcal{T} \in \N$. An elementary movement 
 $\psi_\sigma:M\rightarrow M$ is a measure-preserving invertible map for which there exist $\psi_{\tau_1},\dots,\psi_{\tau_{J}}$ swaps of the adjacent subcubes $(Q_{v_{1,1}},Q_{v_{1,2}}),\dots,(Q_{v_{J,1}},Q_{v_{J,2}})\subset\mathcal{R}_N\times\mathcal{R}_N$, with 
    \begin{equation*}
        \text{int}(Q_{v_{h,1}},Q_{v_{h,2}})\cap\text{int}(Q_{v_{k,1}},Q_{v_{k,2}})=\emptyset, \quad h\not =k,
    \end{equation*}
    and
    \begin{equation*}
        \psi_\sigma=\psi_{\tau_{J}}\circ \psi_{\tau_{J-1}}\circ\dots \psi_{\tau_2}\circ \psi_{\tau_1}.
    \end{equation*}
    We call $\text{swap}(\psi_\sigma)=J$.
The discrete length functional \eqref{def:length:functional:cont} is defined as
\begin{equation}\label{eq:discr:length}
    L\lbrace\psi_{\sigma_j}\rbrace_{j=1}^{\mathcal T}=\sum_{j=1}^{\mathcal{T}} N^{-1-\frac{\nu}{2}}\sqrt{\text{swap}(\psi_{\sigma_j})}.
\end{equation}
This functional is related to the swap distance of permutations \cite{DistPermutations} (see also \cite{Zizza24} as a reference). Define then
\begin{equation*}
    \text{dist}_{\mathcal{D}_N}(\psi_\sigma,\psi_\eta)=\underset{ \psi_{\sigma_j}\circ\dots\circ\psi_{\sigma_1}\circ\psi_\sigma=\psi_\eta}{\min_{\lbrace S_j\rbrace_{j=1}^{\mathcal{T}} }} L\lbrace\psi_{\sigma_j}\rbrace_{j=1}^{\mathcal T}.
\end{equation*}

It holds
\begin{thm}[Shnirelman's discrete inequality \cite{Shnirelman}]
    Let $\nu=2$, then for $\alpha= \tfrac{1}{64}$ exists a positive constant $C>0$, such that for every $N\in\N$ and every permutations of the tiling $\psi_\sigma$ and $ \psi_\eta$ it holds
    \begin{equation}
        \text{dist}_{\mathcal{D}_N}(\psi_\sigma,\psi_\eta)\leq C\|\psi_\sigma-\psi_\eta\|_2^\alpha.
    \end{equation}
\end{thm}
In \cite{Zizza24} an equivalent length functional was considered, resulting in an improvement of the previous result in the discretised context up to $\alpha=\tfrac{2}{7}$ in dimension $\nu=2$ and $\alpha=\tfrac{1}{1+\nu}$ in higher dimension. This is then extended to the continuous setting via the smoothing procedure of \cite[Lemma IV.7.19]{Arnold:khesin}.
One of the main weakness of this method is as follows (cf. \cite{Zizza24} for more details). It is based on volume estimates on the permutation $\psi_\sigma$, which translates into quantifying \emph{how many} cubes are moved by the permutation itself.  This gives supotimal estimates in terms of the H\"older exponent $\alpha$. In contrast, our method directly exploits geometrical features of the permutation $\psi_\sigma$.

\subsection{Incompressible flows on networks.}
In order to deal with \ref{const:1}, we describe our flow as a flow in a discrete network, i.e. in a graph. 

While we also consider a bijective map on the tiling $\Rcal_N$ to approximate the real flow, as done in \cite{Shnirelman},\cite{Zizza24}, the construction of the flow for the discretised problem will be vastly different and \emph{not} based on a swapping procedure. 

We translate our problem into a discrete problem first. Every cube in the tiling is represented by a \emph{vertex} $v \in V$ in an undirected graph $G=(V,E)$, where vertices represent cubes and edges indicate adjacent cubes (in particular the graph is a $\nu$-dimensional grid - also known as \emph{lattice graph}). A \emph{flow} in this graph is then represented by a path (or multiple paths) going from a vertex $v$ to $\sigma(v)$ (cf. Figure \ref{fig:ideas:1}).
\begin{figure}[!htb]
     \centering
     \includegraphics[width=.55\linewidth]{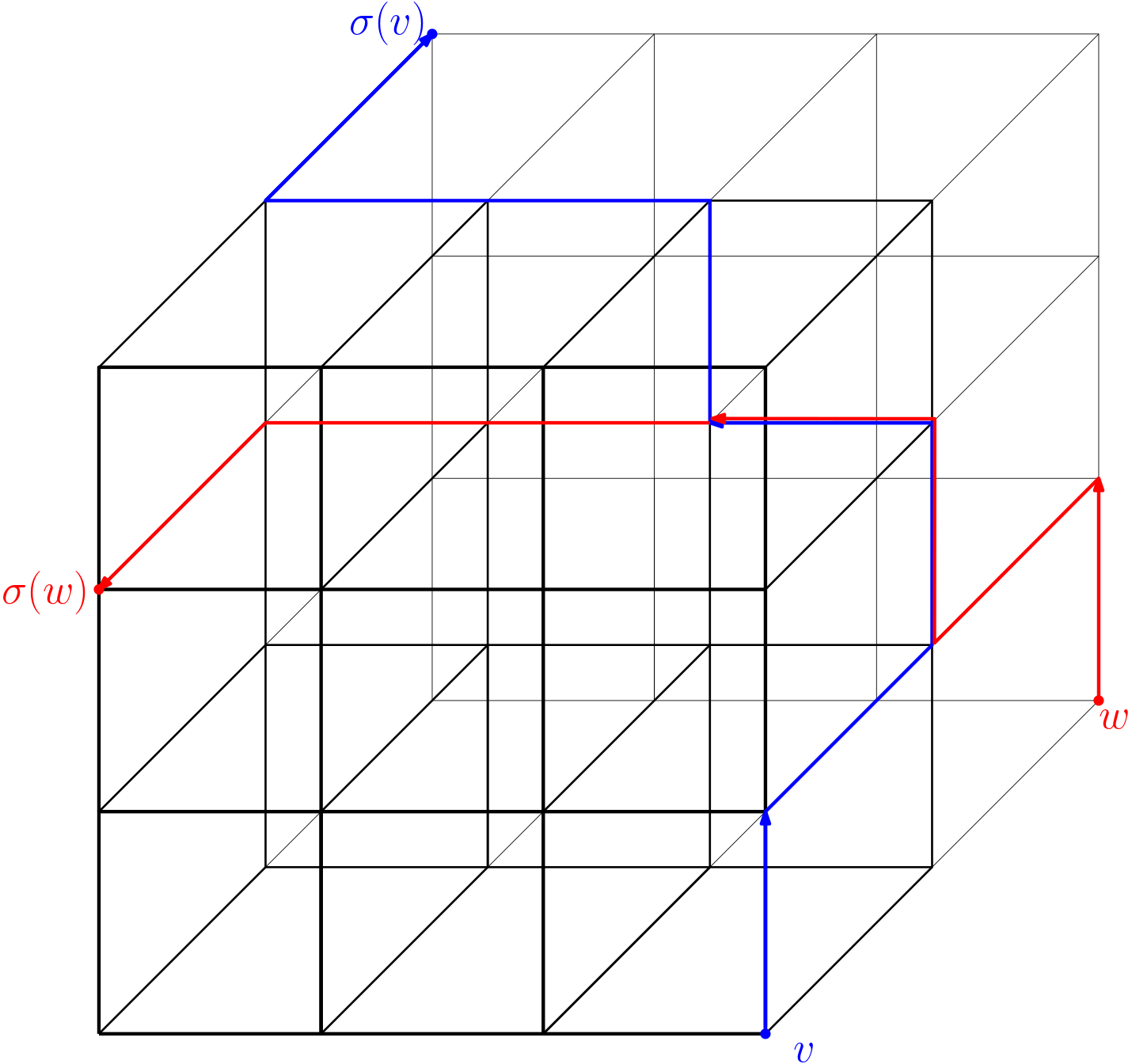}
     \caption{The discrete network, two points $v$ and $w$ and paths from $v$ and $w$ to $\sigma(v)$ and $\sigma(w)$.}
     \label{fig:ideas:1}
\end{figure}
In comparison to other discrete problems (also see Section \ref{subsec:discretecomp}), we enforce an \emph{incompressibility} condition along a path as follows: Imagine a fluid particle that flows along the discrete network with a velocity $u=u(e,\gamma)$ that may depend on the edge $e$ and on the path $\gamma$ that the particle takes. We picture this fluid as if it was flowing through a pipe of a certain \emph{thickness} $\rho(e,\gamma)$ (Figure \ref{fig:flow:cubes}). Then the incompressibility condition tells us that along edges $e$ and $e'$ of the path $\gamma$
\[
\rho(e,\gamma) \cdot u(e,\gamma) = \rho(e',\gamma) \cdot u(e',\gamma).
\]
In particular, in comparison to other discrete flows, we do not picture the particle moving from one vertex to another in a fixed time span (say $1$), but the time that the particle needs from a vertex to an adjacent is anti-proportional to the velocity $u$, i.e. proportional to the thickness $\rho$. Additionally, the $L^p$ norm of the ensuing velocity field below also is connected to the thickness $\rho$. Finally, as a drawback of this approach, our vector field will be 'almost' stationary in the sense that the velocity field will be essentially (up to some translations) constant-in-time.
\medskip

We shortly describe some cornerstones of the problem and its solution, we refer to Section \ref{sec:3} for the details. For a permutation $\sigma$ of vertices $V$ we search for paths $\gamma \in \Gamma_v \subset \Gamma$ from $v$ to $\sigma(v)$ (where $\Gamma= \bigcup_{v \in V} \Gamma_v$) together with weigths $\omega = \omega(\gamma)$ such that
\[
\sum_{\gamma \in \Gamma_v} \omega(\gamma) =1.
\]
We shall now imagine the flow as follows: A volume portion of the cube (exactly $\omega(\gamma)$) is sent through the discrete network along $\gamma$. To accommodate all flows along different paths along an edge $e$ we have the \emph{capacity constraint}
\begin{equation} \label{intro:capa}
\sum_{\gamma \colon e \in E(\gamma)} \omega(\gamma) \rho(e,\gamma) \leq 1 \quad \forall e \in E,
\end{equation}
in other words this entails that the sum of thicknesses of tubes does not exceed $1$. Furthermore, we have a constraint that we need to transport the fluid from $v$ to $\sigma(v)$ in a certain amount of time, i.e. the \emph{time constraint}
\begin{equation} \label{intro:time}
\sum_{e \colon e \in E(\gamma)}  \rho(e,\gamma) \leq 1 \quad \forall \gamma \in \Gamma.
\end{equation}
Given these constraints on the flow, we aim to minimise a functional that stems from the corresponding $L^p$ norm, which is the \emph{cost}
\begin{equation} \label{intro:cost}
\left(\sum_{\gamma \in \Gamma} \sum_{e \in E(\gamma)} \omega(\gamma) \rho(e,\gamma)^{-p+1} \right)^{1/p}.
\end{equation}

The goal Section \ref{sec:3} is to give a approximative solution to that minimising to that problem under the given constraint- i.e. an admissible instance that is only worse than the the actual minimiser by a dimensional constant. 
This approximate minimiser actually consists of a carefully chosen set of shortest paths between $v$ and $\sigma(v)$, cf. Figure \ref{fig:choose1}, together with appropriately chosen thicknesses $\rho$ and weights $\omega$.
\begin{figure}
    \centering \includegraphics[width=0.3 \textwidth]{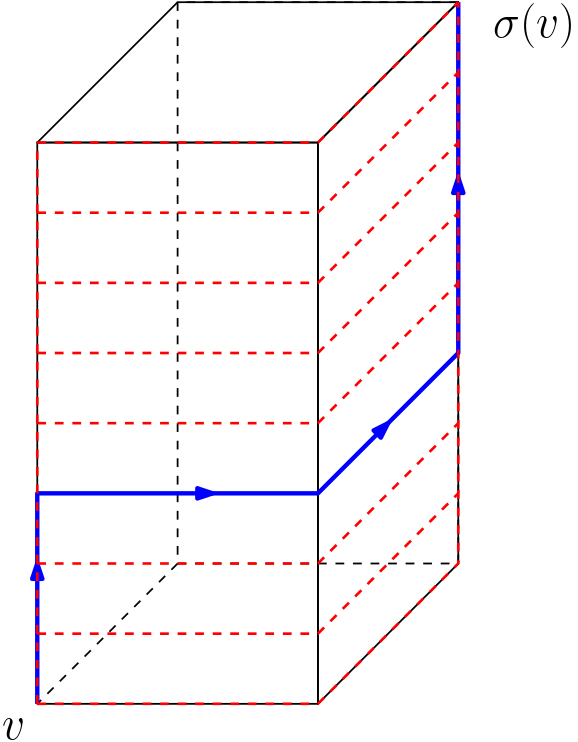}
    \caption{The set of paths from $v$ to $\sigma(v)$ in three dimensions. The paths are chosen in such a way, so that they follow the direction, where $(v-\sigma(v))_i$ is maximal. The path then follows this direction until some $v+\alpha e_i$. The path then adjusts the other two coordinates, before following the initial direction for the rest of the path. All paths are equipped with equal weight $\omega$.}
    \label{fig:choose1}
\end{figure}

\subsection{From the discrete back to the continuous setting}
\begin{figure}
      \includegraphics[width=.55\linewidth]{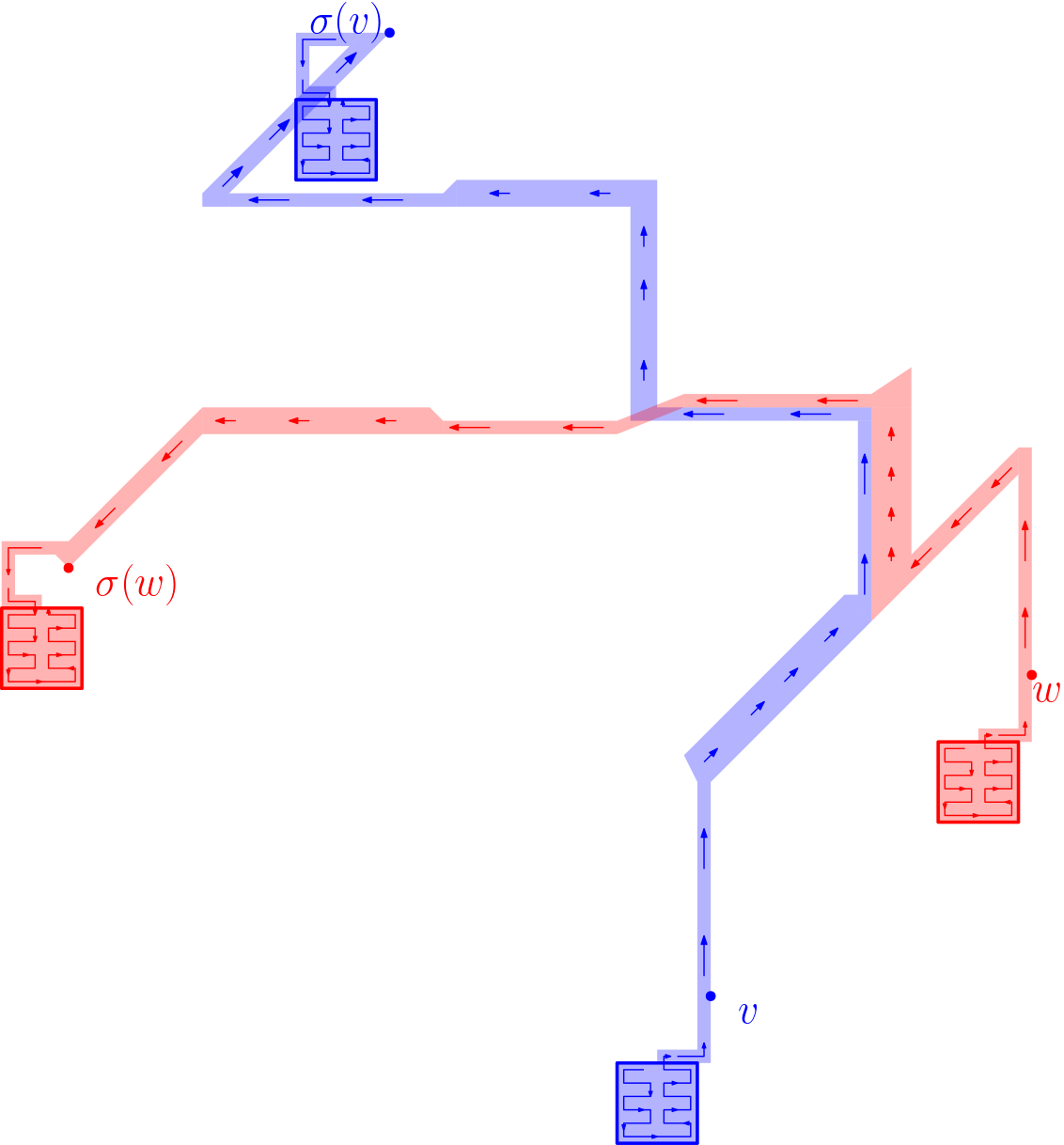}
      \caption{The realisation of the solution of the discrete problem displayed in Figure \ref{fig:ideas:1} as a flow/vector field. Observe that we roughly follow the edges of the network and use the previously attained solution. The 'pipe' that we use for every edge has some width $\rho(e,\gamma)$. The velocity of the flow then is antiproportional to this width.
      The vector field then transports the particles in a small subcube to the appropriate subcube in a fixed amount of time.
      }\label{fig:ideas:2}
\end{figure}
The solution given by the discrete framework has to be translated back into a continuous setting. For this, any small cube $\q_v$ in the tiling is subdivided into $K^{\nu}$ even smaller subcubes, where $K \in \N$ is a fixed constant we choose prior to the construction. In a single step, we only move appropriate subcubes to their correct position; we then repeat this construction $K^{\nu}$-times to move all subcubes into the right position.

The flow itself roughly looks as follows. A fluid particle starting in a small subcube will first move inside the subcube and then leave it via a \emph{pipe} that \emph{connects} the cube to a discretised network. It will then flow along pipes in the network that are constructed according to the solution of the discrete problem (Section \ref{sec:3}). Finally, it leaves the discrete network and then is connected via another pipe to its destination subcube.

The exact construction of such a flow is carefully carried out in Section \ref{sec:4} \& \ref{sec:construction}. We remark that all bounds in relevant $L^p$ spaces then are, up to dimensional constants, a direct consequence of the bounds for the discrete problem. In particular, as in the graph framework we achieve $\alpha=1$, we also achieve $\alpha=1$ for Theorem \ref{thm:main} -- only the involved constant might be quite large and depend on the choice of some construction parameters such as $K$.

\FloatBarrier
\subsection{Comparison to  classical discrete network flows} \label{subsec:discretecomp}
%%--------------------------------------
Before returning to some problems regarding continuous fluid flows, we pause and compare the discrete problem at hand to some well-known flow problems of discrete mathematics. Naturally, the list below is non-exhaustive and refer to the textbooks \cite{AMO,KV} for more examples.

One of the first deeply studied network problems is the simple $s$-$t$ flow in an undirected graph, famously studied in \cite{FF,FF2}, where the task it to find a maximal flow from a source $s$ to a sink $t$ that obeys a capacity constraint at any edge $e \in E$. A variant of this problem that is much closer to our problem is the \emph{multi-commodity flow} or edge-disjoint path problem (cf. \cite[Section 19]{KV}). In particular, in those problems there are multiple sinks and sources that need to be precisely liked, as in our problem. Adapted to our setting, that would mean that we need to find paths from any $v$ to $\sigma(v)$ such that a certain capacity constraint
\[
\# \{ \gamma \colon e \in E(\gamma)\} \leq C(\nu)
\]
for a purely dimensional constraint $C(\nu)$ is satisfied.

Our approach is more flexible than that: We additionally add the \emph{thickness}, such that we can make any set of paths admissible (in the sense that all capacity constraints are satisfied); but we have to pay a cost for that in terms of the $L^p$ norm of the fluid flow, i.e. \eqref{intro:cost}. Observe that this cost function is convex and nonlinear.

While we are not aware of previous study of our problem, let us note that \emph{convex} optimisation problems also have been studied for discrete network flows, for instance in \cite[Chapter 14]{AMO}. Examples of such problems are electrical networks \cite{CR7,GKMN} or problems in pipe networks \cite{collins,Cross}.

A difference to most convex optimisation problems is, however, the nature of our constraints and the target functional. In particular, our convex cost function is \emph{decreasing} while increasing the thickness $\rho$. On the contrary, we do not have tight constraints (any set of paths can be equipped with very small $\rho$ so that conditions \eqref{intro:capa} and \eqref{intro:time} are satisfied.

We also refer to the survey \cite{Bressan} for other connections between physical flow and discrete flows on networks.
\subsection{Issues for low dimensions} \label{sec:lowD}
While the case $\nu=1$ in the context of volume-preserving diffeomorphisms is trivial (being $\mathcal{D}((0,1))=\lbrace\Id\rbrace$), thus it is studied via Generalised Flows, the case $\nu=2$ is of independent interest. We recall that Shnirelman's inequality, in the formulation of Theorem \ref{thm:shnirelman}, Corollary \ref{coro:main}, is valid in $\nu\geq 3$, but no such inequality can be obtained in dimension $\nu=2$ because of topological obstructions. In particular, it holds that for every pair of positive constants $c, C$ there exists a diffeomorphism $g \in\mathcal D(M)$ such that
$dist_{\mathcal D(M)}(g,\Id_M) > C$, but $\|g-\Id_M\|_2\leq c$ (cf. \cite{Arnold:khesin}). Nevertheless, discrete-type versions of inequality \eqref{eq:ineq:shnirelman}, where instead of volume-preserving diffeomorphisms, permutations of the tiling $\mathcal{R}_N$ are considered (Subsection \ref{ss:difference:approaches}), are available also in dimension $\nu=2$ . 

\medskip
Even if the setting of our result (Theorem \ref{thm:main}) is purely discrete, the main constructions  can be realised only in dimension $\nu\geq 3$. A proof for the case $\nu=2$ with our techniques is still not available, the following problem is still left open.

\medskip
\textbf{Open Problem 1.} Let $\nu=2$, $1 \leq p,q \leq \infty$. Find the optimal exponent $\alpha$ for which there exists $C_{p,q}>0$ such that for every $N\in\N$, for every $\psi_\sigma\in\mathcal{D}_N$, there exists a vector field $u\in L^p([0,1],L^q(M))\cap L^1([0,1],\BV(M))$, with $\psi_\sigma$ the time-$1$ map of the vector field $u$, and
    \begin{equation}\label{eq:ineq:shnirelman:sharp:2}
        \|u\|_{L^p_tL^q_x}\leq C_{p,q}\|\psi_\sigma-\Id_M\|^\alpha_{L^q_x}.
    \end{equation}
To the best of our knowledge, so far the best bound for $\nu=2$ is $\alpha = \frac{2}{7}$, cf. \cite{Zizza24}. We refer also to that paper for open questions related to the discrete context.

\subsection{Absence of shortest path and the two point problem}
   In \cite{Shnirelman},\cite{Shnirelman2} it is proved that the minimum for the length functional cannot be achieved in general, thus $\text{dist}_{\mathcal{D}(M)}$ is an infimum. On the other side, we point out that, thanks to our construction, 
   \begin{equation}\label{eq:equivalence:metric}
       \|f-\Id_{M}\|_{L^2_x}\leq \text{dist}_{\mathcal D(M)} (f,\Id_M)\leq C \|f-\Id_{M}\|_{L^2_x},
 \end{equation}
   where the flow connecting $f$ to $\Id_M$ is constructed explicitly. Let us mention that our constant $C>0$ is not very explicit, and it might be further reduced considering the minimisation problem on the network, see Theorems \ref{thm:1D}, \ref{thm:2D} and \ref{thm:multiD}, as well as optimising some steps of the construction, cf. Section \ref{sec:construction}. In particular, our construction gives a minimiser of the length functional \emph{up to a constant}. 
   
   While the picture is clear in dimension $\nu\geq 3$, in dimension $\nu=2$ there are still many open questions regarding the existence of the shortest path even for isotopic states. Here, $f,g \in \text{Sdiff}(M)$ are isotopic if there exists a smooth path $t\mapsto\varphi_t$ connecting those two states with finite action.
   To pose the question more precisely, if $M\subset\R^2$ is a bounded simply connected open domain, and if $f,g\in\text{SDiff}(M)$ are two isotopic states, then (cf. \cite{Shnirelman2}) one may conjecture that 
   there exists $t\mapsto \varphi_t$ minimising the length functional \eqref{eq:action}, see also \cite{DrivasElgindi}. In particular, we expect that for isotopic states, inequality \eqref{eq:equivalence:metric} should be also valid in dimension $\nu=2$, together with an explicit construction of the fluid flow. We arrive at the following question.
   
   \textbf{Open Problem 2.} There exists a positive constant $C>0$ such that, for every $f,g\in\Dcal([0,1]^2)$ isotopic,
   \begin{equation*}
       \text{dist}_{\Dcal([0,1]^2)}\leq C\|f-g\|_{L^2([0,1]^2)}.
   \end{equation*}

\subsection*{General Notation}
Below we give a non-exhaustive overview over frequently used notation; some of it will be introduced in further detail below.
\subsubsection*{The cube and its discretisation}
\begin{itemize}
    \item $M=[0,1]^\nu$ is the reference domain;
    \item $\Dcal(M)=\text{SDiff}(M)$ is the space of volume-preserving diffeomorphisms on $M$;
    \item $\mathcal{R}_N$ is a tiling of $M$ into $N^\nu$ cubes of sidelength $N^{-1}$;
     \item $\Dcal_N$ denotes the set of permutations of $\mathcal{R}_N$;
     \item for a multiindex $v \in \{0,1,\ldots,N-1\}^{\nu}$ denote by $Q_v= N^{-1} v + N^{-1} [0,1)^{\nu}$ the cubes of $\Rcal_N$.

\end{itemize}
\subsubsection*{Discrete Mathematics}
    \begin{itemize}
    \item $G=(V,E)$ denotes an (undirected) graph, where $V$ is the set of vertices and $E$ the set of edges;
    \item for our purpose on the grid, $V=\{0,1,2,\ldots,N-1\}^{\nu}$ is the set of vertices;
    \item in that instance, $e = \{v,w\} \in E$ if and only if $\vert v- w \vert=1$;
    \item $\mathfrak S_{N,\nu}$ denotes the symmetric group over the set $V$, i.e. the set of bijective maps $V \to V$;
    \item $\sigma$ denotes the generic element of $\mathfrak S_{N,\nu}$;
    \item a path $\gamma$ is a sequence of disjoint vertices $v_1, \ldots,v_k$ such that $V_i$ and $v_{i+1}$ are connected by an edge;
    \item $\Gamma$ denotes a set of paths;
    \item for a path $\gamma$, $E(\gamma)$ denotes the set of edges belonging to a path $\gamma$ and $V(\gamma)$ the set of vertices in $\gamma$.
    \end{itemize}
\subsubsection*{The discrete problem and its continuous counterpart}
\begin{itemize}
    \item $\psi_{\sigma} (x)= N^{-1} \sigma(v) + (x- N^{-1} v) \quad \text{if } x \in \q_v,~v \in V$, and $\sigma \in \mathfrak S_{N,\nu}$;
    \item for an edge $e$ and a path $\gamma$ the 'thickness' of a pipe is denoted by  $\rho(e,\gamma)$;
    \item $u(e,\gamma)$ is the velocity of the fluid through a pipe;
    \item $\omega(\gamma)\in (0,1]$ is the weight of a path;
    \item $\mathcal{Q}_v$ is a subcube of $Q_v$ of sidelength $\ell$;
    \item $K= N^{-1} /\ell$ denotes the ratio of sidelength of cubes $Q_v$ and subcubes $\Qcal_v$;
    \item $\lambda < \ell$ is another length parameter that is needed during the construction;
    \item $\kappa = \ell/\lambda$ is another ratio we fix throughout the construction.
\end{itemize}
Another small caveat on notation: We often write $\spt(u_1) \cap \spt(u_2) = \emptyset$; with this we mean that the corresponding open sets are disjoint, i.e.
\[
(\spt(u_1))^{\circ} \cap (\spt(u_1))^{\circ} = \emptyset,
\]
or, equivalently (for the type of functions we consider) that $\mathcal{L}^n(\spt(u_1)\cap \spt(u_2)) =0$.

%% file: sec3.tex
\section{The discrete problem and its solution} \label{sec:3}
We recall some basic notation from discrete mathematics, see for instance \cite{KV}. An undirected graph $G$ consists of a finite number of vertices $V$ and a set of edges $E \subset V^2 = \{ e \subset V \colon \vert e \vert =2 \}$ (i.e. we do not allow edges from a vertex to itself as well as parallel edges).
For an edge $e =\{v,w\}$ we say that $v$ and $w$ are the endpoints of $e$.
A \emph{path} $\gamma$ from a vertex $v_1$ to $v _{k+1}$ is a sequence $ \gamma= \lbrace v_1,...,v_{k+1}\rbrace$ with $v_i \neq v_j$ for $i \neq j$ such that $e_i := \{ v_i, v_{i+1}\} \in E$ \footnote{So that both vertices and edges do not appear twice in a path. Everything that we consider below may also be done for \emph{walks}, we stick to paths for brevity}. We can understand a path as a subgraph of $G$ and, consequently, write $V(\gamma) = \{v_1,\ldots v_{k+1}\}$ and write $e \in E(\gamma)$ if $e= \{v,v_{i+1}\}$ and $v_i,v_{i+1} \in V(\gamma)$. Usually we will denote a set of paths by $\Gamma$.

We assume that the graph $G$ is connected, i.e. that for any $v,w \in V$ we may find a path between $v$ and $w$. For a path $\gamma$ define the \emph{length} $\ell(\gamma)$ of $\gamma$ as the number of edges contained in $E(\gamma)$ and define
\[
    \dist(v,w) =\min_{\gamma \text{ path from } v \text{ to } w } \ell(\gamma).
 \]
If $\gamma_1$ is a path from $v_1$ to $v_{k+1}$ and $\gamma_2$ is a path from $v_{k+1}$ to $v_{m}$, (such that $v_i \neq v_j$ for $1 \leq i,j \leq m$, then we may define $\gamma ={\lbrace}v_1,...,v_m{\rbrace}$ as the concatenation of the paths $\gamma_1$ and $\gamma_2$. 

\subsection{Incompressible flows in graphs} 
In contrast to classical problems in network flow, where the in-flow at a point equals its outflow, in analytical problems, the incompressibility of the  flow along trajectories $\tau$ can also be ensured by the \emph{velocity} of the trajectory, i.e. if the fluid's path is very narrow, the velocity of the fluid increases. We model this as follows: we consider a decomposition of a flow (cf. also \cite[Theorem 8.8]{KV}) into paths. Imagine the flow as a fluid flow along a pipe through the discrete network. For every path $\gamma$ and every edge contained in $\gamma$ we consider a thickness $\rho \colon E \times \Gamma \to [0,\infty)$ (which we may think of as thickness of the pipe)  and a velocity $u \colon E \times \Gamma \to [0,\infty)$. The $\emph{flow}$ through the pipe per time unit then is the product 
\[
f(e,\gamma) = u(e,\gamma) \cdot \rho(e,\gamma).
\]
The flow is \emph{incompressible} if, for every $\gamma\in\Gamma$, for each neighbouring edges $e_i,e_{i+1}\in E(\gamma)$, (meaning $e_i \cap e_{i+1} \neq \emptyset$) in the path the flow is the same, i.e. the value of the flow \emph{does not} depend on $e$
\[
f(e,\gamma) = f(\gamma).
\]
Assuming that 
\[
f(\gamma) =1
\]
for all paths $\gamma$ we infer
\[
u(e,\gamma) = \rho(e,\gamma)^{-1} \quad \text{if } e \in E(\gamma).
\]
We obtain the following discrete problem. \\
\begin{problem}[Incompressible flow along one path] \label{problem:inc1}
     ~\\
     \emph{Input: } A graph $G=(V,E)$, a number $1\leq p \leq \infty$ and a bijective map $\sigma: V \to V$. \\
    \emph{Task: } Find paths $\gamma_v$ from $v$ to $\sigma(v)$, $\Gamma$ denoting the set of paths, and $\rho \colon E \times \Gamma \to [0,1]$ obeying
    \begin{enumerate} [label=(\roman*)]
        \item $\rho(e,\gamma) =0$ if $e \notin E(\gamma)$;
        \item time constraint: for each $\gamma \in \Gamma$ we have
            \begin{equation} \label{eq:volume}
                \sum_{e \in E(\gamma)} \rho(e,\gamma) \leq 1;
            \end{equation}           
        \item capacity constraint: for each $e \in E$ we have
            \begin{equation} \label{eq:capacity}
                \sum_{\gamma \colon e \in E(\gamma)} \rho(e,\gamma) \leq 1
            \end{equation}
        \end{enumerate}
    such that the $l_p$ cost functional
    \begin{equation} \label{def:cp}
         \tilde{c}_p(\rho):= \left(\sum_{\gamma \in \Gamma} \sum_{e \in E(\gamma)} \rho(e,\gamma)^{-p+1} \right)^{1/p}
    \end{equation}
    and, for $p=\infty$,
    \begin{equation} \label{def:cinfty}
        \tilde{c}_\infty(\rho):= \sup_{\gamma \in \Gamma} \sup_{e \in E(\gamma)} \rho(e,\gamma)^{-1}
    \end{equation}
    is minimised.
\end{problem}
~\\
This problem deals with the construction of a flow along \emph{one} path according to a permutation as the bijective function $\sigma:V\rightarrow V$, $\sigma\in\mathfrak{S}_{N,\nu}$. 
We shortly remind the reader of the background behind conditions \eqref{eq:volume},\eqref{eq:capacity} and for the definition of the $\ell^p$ cost functional \eqref{def:cp},\eqref{def:cinfty}.  Condition \eqref{eq:capacity} is related to the amount of pipes that go through an edge $e$: this amount has to be controlled to guarantee the construction in the network can be adapted to the continuous setting. Condition \eqref{eq:volume} refers to the fact that the fluid particle shall be transported from one vertex to another in a certain amount of time, which is in our case at most $t=1$. Finally, the $\ell_p$ cost functional is related to the $L^p$-norm of the velocity field, which is given by $$ \left(\sum_{\gamma\in\Gamma}\sum_{e\in E(\gamma)}u(e,\gamma)^p\rho(e,\gamma)\right)^{\frac{1}{p}}.$$

Especially in higher dimension it might be useful to consider a flow that is split up into several different paths. We then obtain the following problem: \\
\begin{problem}[Incompressible flow along multiple paths] \label{problem:inc2}
   ~\\
   \emph{Input: } A graph $G=(V,E)$, a number $1\leq p \leq \infty$ and a bijective map $\sigma: V \to V$. \\
    \emph{Task: } Find paths $\gamma_{v,i}$ from $v$ to $\sigma(v)$, $v \in V$ and $i \in I_v$ ($I_v$ being a finite index set), $\Gamma= \bigcup_{v \in V} \bigcup_{i \in I_v} \{\gamma\}$ be the set of all paths, $\rho \colon E \times \Gamma \to [0,1]$ and weights $\omega \colon \Gamma \to (0,1]$ obeying
    \begin{enumerate} [label=(\roman*)]
        \item $\rho(e,\gamma) =0$ if $e \notin E(\gamma)$;
        \item decomposition of flow: for each $v \in V$
                \begin{equation} \label{eq:decomp}       
                    \sum_{i \in I_v} \omega(\gamma_{v,i}) =1.
                \end{equation}
        \item time constraint: for each $\gamma \in \Gamma$ we have
            \begin{equation} \label{eq:volume2}
                \sum_{e \in E(\gamma)} \rho(e,\gamma) \leq 1;
            \end{equation}           
        \item capacity constraint: for each $e \in E$ we have
            \begin{equation} \label{eq:capacity2}
                \sum_{\gamma \colon e \in E(\gamma)} \omega(\gamma) \cdot \rho(e,\gamma) \leq 1
            \end{equation}
        \end{enumerate}
    such that the $l_p$ cost functional
    \begin{equation} \label{def:cp2}
        c_p(\rho,\omega) := \left(\sum_{\gamma \in \Gamma} \sum_{e \in E(\gamma)} \omega(\gamma) \rho(e,\gamma)^{-p+1} \right)^{1/p}
    \end{equation}
    and, for $p=\infty$,
    \begin{equation} \label{def:cinfty2}
        c_\infty(\rho,\omega) := \sup_{\gamma \in \Gamma} \sup_{e \in E(\gamma)}  \rho(e,\gamma)^{-1}
    \end{equation}
    is minimised.
\end{problem}
~ \\
Obviously, Problem \ref{problem:inc1} is a simpler version of Problem \ref{problem:inc2} as the index set $I_v$ only consists of one member each and $\omega(\gamma) =1$ for each $\gamma$. The weight $\omega$ corresponds to a \emph{volume portion} of the cube $Q$ that is transported via the path $\gamma$. Taking a more probabilistic standpoint, condition \eqref{eq:decomp} might be viewed as the probability that a fluid particle starting in a small cube follows the path $\gamma_{v,i}$, giving a connection to the concept of Generalised Flows, cf. \cite{Brenier1}.
\smallskip

While we might attempt to solve the discrete problem exactly, for our purposes it is sufficient to give a solution that attains the minimum up to a constant. We remark that an \emph{exact solution} of the discrete problem might be connected to obtaining minimisers of the action functional. Let us call paths $\gamma$, thicknesses $\rho$ and weights $\omega$ \emph{admissible} if they satisfy the constraints given by Problems \ref{problem:inc1} and Problem \ref{problem:inc2}. \\
\begin{problem} \label{realproblem}  ~\\
  \emph{Input: } A graph $G=(V,E)$, a number $1\leq p \leq \infty$ and a bijective map $\sigma: V \to V$. \\
  \emph{Task: } Find admissible $\gamma'$, $\rho'$ and $\omega'$ such that
  \[
  c_p(\rho',\omega') \leq C(p,G) \inf_{\gamma,\rho,\omega \text{ admissible}} c_p(\rho,\omega).
  \]
\end{problem}
Here, $C(p,G)$ is a constant that only depends on $p$ and the graph $G$ (in our application to grids, the constant only depends on $p$ and the dimension $\nu$).
While it is very hard to give an explicit value to the infimum on the right hand side, we can give a quite effective lower bound. 
\begin{lemma}
Let $G=(V,E)$, $1 \leq p \leq \infty$.
Then, for $1 \leq p <\infty$
\begin{equation} \label{lowerbound:p}
\inf_{\gamma, \rho,\omega \text{ admissible}} c_p(\rho,\omega) \geq \left( \sum_{v \in V} \dist(v,\sigma(v))^p \right)^{1/p}
\end{equation}
and
\begin{equation} \label{lowerbound:infty}
\inf_{\gamma,\rho,\omega \text{ admissible}} c_\infty(\rho,\omega) \geq \sup_{v \in V} \dist(v,\sigma(v)).
\end{equation}
\end{lemma}
The same inequality obviously also holds for $\tilde{c}_p$ and $\tilde{c}_{\infty}$ as they are larger than $c_p$ and $c_{\infty}$, respectively.
\begin{proof}
Fix $v \in V$. Any path between $\gamma$ between $v$ and $\sigma(v)$ has length at least $\dist(v,\sigma(v))$.
If $p= \infty$ we infer from the time constraint that there is an edge $e \in E(\gamma)$ such that $\rho(e,\gamma) \leq \dist(v,\sigma(v))^{-1}$, i.e. $\rho(e,\gamma)^{-1} \geq \dist(v,\sigma(v))$, i.e. \eqref{lowerbound:infty} holds.

If $p<\infty$, we recall the time constraint and further observe that $x \mapsto x^{-p+1}$ is convex. As we have $\ell(\gamma)$ summands we infer
\begin{equation}\label{eq:final:section}
\sum_{e \in E(\gamma)} \rho(e,\gamma)^{-p+1} \geq \ell(\gamma) \cdot (\ell(\gamma)^{-1})^{1-p} = \ell(\gamma)^p \geq \dist(v,\sigma(v))^p.
\end{equation}
Summing over all possible $\gamma$, and using \eqref{eq:decomp} yields the statement.
\end{proof}
So, we might reformulate Problem \eqref{realproblem} in a fashion such that  instead of the infimum we have the distance on the right hand side.
\begin{problem} \label{realproblem2} ~ \\
  \emph{Input: } A graph $G=(V,E)$, a number $1\leq p \leq \infty$ and a bijective map $\sigma: V \to V$. \\
  \emph{Task: } Find admissible $\gamma'$, $\rho'$ and $\omega'$ such that
  \[
  c_p(\rho',\omega') \leq C(p,G) \Vert \dist(v,\sigma(v)) \Vert_{l_p(V)}.  
  \]
\end{problem}
\smallskip

The rest of this section is devoted to this task in the special case where $G$ is the graph coming from a $\nu$-dimensional grid.

\subsection{The one-dimensional problem} \label{sec:discrete:1D}
Let us first discuss the one-dimensional problem as it is quite straightforward.
Suppose that we are given $N$ vertices in a line, i.e. $V= \{0,....,N-1\}$ and $E = \{ \{k,k+1\} \colon k=0,...,N-2\}$. Then $\dist(k,k') = \vert k- k' \vert$. We now consider a permutation of vertices $\sigma \colon V \to V$ . We propose to solve Problem \ref{problem:inc1} approximately as follows: for each pair of vertices $k$ and $\sigma(k)$ consider the direct path from $k$ to $\sigma(k)$ and let
\begin{equation} \label{def:flow:1D}
    \rho(e,\gamma) = \begin{cases}
        \Bigl(\max \left( F(e), \ell(\gamma) \right)\Bigr)^{-1} \text{ if } e\in E(\gamma), \\
        0\quad\text{ otherwise.}
    \end{cases}
\end{equation}
where $F(e)$ is the number of paths $\gamma$ such that $e \in E(\gamma)$.
\begin{lemma} \label{lemma:1D:1}
Consider the previous setup of $G=(V,E)$, paths $\gamma$ and $\rho$. Then 
\begin{enumerate} [label=(\roman*)]
    \item \label{lemma:1D:1:a} If $e = \{k, k+1\}$ then $F(e) = \#\{j > k \colon \sigma(j) \leq k \} + \#\{ j\leq k \colon \sigma(j)>k\}$;
    \item \label{lemma:1D:1:b} $(\gamma,\rho)$ is admissible (for Problem \ref{problem:inc1}).
\end{enumerate}
\end{lemma}
\begin{proof}
   
    Point  \ref{lemma:1D:1:a} follows immediately by the definition of $F(e)$. For point\ref{lemma:1D:1:b} we first check the time constraint: fix $\gamma$ and $e\in E(\gamma)$. Then
    \begin{equation*}
        \sum_{e\in E(\gamma)}\rho(e,\gamma)=\sum_{e\in E(\gamma)} \Bigl(\max \left( F(e), \ell(\gamma) \right)\Bigr)^{-1}\leq \sum_{e\in E(\gamma)}\ell(\gamma)^{-1}=1.
    \end{equation*}
    The capacity constraint follows by the same computations using the trivial inequality $\max \left( F(e), \ell(\gamma) \right)\geq F(e)$. 
\end{proof}
Suppose now that $X \subset V$ and let 
\[
F_X(e) = \# \{ \gamma_v \in \Gamma \colon v \in X, e \in E(\gamma_v). \}
\]
\begin{lemma} \label{lemma:1D:2}
    Let $G=(V,E)$, $\gamma$ and $\rho$ be as before. Then we have
    \begin{enumerate} [label=(\roman*)]
        \item \label{lemma:1D:2a} $\Vert F_X(\cdot) \Vert_{l^1(E)} = \Vert \dist(\cdot,\sigma(\cdot)) \Vert_{l^1(X)}$;
        \item  \label{lemma:1D:2b} $\Vert F_X(\cdot) \Vert_{l^{\infty}(E)} \leq 2 \Vert \dist(\cdot,\sigma(\cdot))\Vert_{l^{\infty}(X)} $;
        \item for any $1 \leq p \leq \infty$ we have $ \Vert F(\cdot) \Vert_{l^{p}(E)} \leq C_p \Vert \dist(\cdot,\sigma(\cdot)) \Vert_{l^{p}(V)}$.
    \end{enumerate}
\end{lemma}
\begin{proof}
    The $l^1$-equality is an argument of counting the same quantity twice. Let $\gamma_v$ be the path from $v$ to $\sigma(v)$.
    \begin{align*}
        \Vert F_X(\cdot) \Vert_{l^1(E)} =& \sum_{e \in E} \# \{ \gamma_v, v \in X \colon e \in E(\gamma) \} = \sum_{v \in V} \# \{ e \colon e \in E(\gamma_v), v \in X \} = \Vert \dist(\cdot,\sigma(\cdot)) \Vert_{l^1(X)}.
    \end{align*}
    The $l^{\infty}$-inequality builds on the following observation: Let $K= \Vert \dist(\cdot,\sigma(\cdot)) \Vert_{l^{\infty}(X)}$ and $e = \{ k, k+1\}$. Then if $e \in E(\gamma_v)$ for some $v \in X$ we infer $k-K < v < k+1+K$. Hence, $F_X(e) \leq 2K$ and therefore we get the desired inequality. \\
    Observe that setting $X=V$ in both \ref{lemma:1D:2a} and \ref{lemma:1D:2b} yields the bound for $p=1$ and $p=\infty$. 
    The $l^p$-inequality for $1<p<\infty$ then follows by an interpolation argument. To be more precise, we roughly follow ideas from the proof of Marcinkiewicz interpolation theorem (e.g. \cite[Theorem 1.3.2]{Grafakos}).
    In particular, for a positive number $s >0$ set 
    \[
   X_s^1 = \{ v \in V \colon \vert v - \sigma(v) \vert \geq s/2\}  \quad \text {and } \quad X_s^{\infty} = \{ v \in V \colon \vert v - \sigma(v) \vert < s/2\},
    \]
    such that $V= X_s^1 \cup X_s^{\infty}$.
    Accordingly, we obtain a decomposition $\Gamma = \Gamma^1_s \cup \Gamma^\infty_s$ of paths with lengths larger than $s/2$ and strictly less than $s/2$, respectively. Let $F^1_s(e)$ be the number of paths in $\Gamma^1_s$ that go through $e$ and similarly define $F^{\infty}_s(e)$.
 Observe that 
 \[
 F(e) = F^1_s(e) + F_s^{\infty}(e)
 \]
 and that 
 \[
 \Vert \sigma - \id \Vert_{l^1(X_s^1)} \leq (s/2)^{1-p} \Vert \sigma - \id \Vert_{l^p(X_s^1)}^p, \quad \Vert \sigma - \id \Vert_{l^{\infty}(X_s^{\infty})} < s/2.
 \]
Then, due to \ref{lemma:1D:2b}, $F_s^{\infty}(e) <s/2$. 
Using the discretised version of the layer-cake formula and Chebychev's inequality we obtain
\begin{align*}
    \Vert F(\cdot) \Vert_{l^p(E)}^p &= \sum_{s=1}^{\infty} (s^p -(s-1)^p) \# \{e \colon F(e) \geq s \} 
    \\
    & \leq  \sum_{s=1}^{\infty} (s^p -(s-1)^p) \# \{e \colon F^1_s(e) \geq (s/2) \} \\
    & \leq \sum_{s=1}^{\infty} (s^p -(s-1)^p) \tfrac{2}{s} \Vert F^1_s(e) \Vert_{l^1(E)} 
     \leq \sum_{s=1}^{\infty} (s^p -(s-1)^p) \tfrac{2}{s}  \Vert \sigma -\id \Vert_{l^1(X^1_s)}
    \\
    &\leq \sum_{s=1}^{\infty} (s^p-(s-1)^p) \tfrac{2}{s} \sum_{v \colon \vert v-\sigma(v) \vert \geq s/2} \vert \sigma(v) - v \vert \\
    & \leq \sum_{v \in V} \vert \sigma(v) - v\vert \sum_{s=1}^{2\vert \sigma(v) - v \vert}  (s^p-(s-1)^p) \tfrac{2}{s} 
     \\
     &\leq C_p \sum_{v \in V} \vert \sigma(v) - v \vert  \Bigl(\vert \sigma(v) - v \vert^{p-1} \Bigr) \\
    & = C_p \Vert \sigma - \id \Vert_{l^p(V)}^p.
\end{align*}

Taking the $p$-th root of the equation yields the desired interpolation.
\end{proof}
We shortly remark that one may optimise the constant in Lemma \ref{lemma:1D:2} to be $C_p\leq 2$.
We finally arrive at the following result regarding the 1D problem:
\begin{thm} \label{thm:1D}
    Let $G=(V,E)$ be the line, let $\gamma_v$ be the direct path from $v$ to $\sigma(v)$ and let $\rho$ be as in \eqref{def:flow:1D}. Then the set of paths $\gamma$ and $\rho$ is admissible to the Problem \ref{problem:inc1} and there is a constant $C>0$ such that for any $1 \leq p \leq \infty$
    \[
    \tilde{c}_p(\rho) \leq C \Vert \dist(\cdot ,\sigma(\cdot))\Vert_{l^p(V)}.
    \]
\end{thm}
\begin{proof}
    We already showed in Lemma \ref{lemma:1D:1} that $\gamma$ and $\rho$ are admissible. It remains to show the bound.
    For this observe that 
    \begin{align*}
        \rho(e,\gamma)^{-1} \leq F(e) + \ell(\gamma).
    \end{align*}
    Thus, due to Lemma \ref{lemma:1D:2}
    \[
    \tilde{c}_{\infty}(\rho) = \sup_{\gamma \in P} \sup_{e \in E(\gamma)} \rho(e,\gamma)^{-1} \leq \sup_{e \in E} F(e) + \sup_{\gamma \in P} \ell(\gamma) \leq 3 \Vert \cdot - \sigma(\cdot) \Vert_{l^{\infty}(V)}.
    \]
    Moreover, for $p<\infty$ we have
    \begin{align*}
        \tilde{c}_p(\rho)^p &= \sum_{\gamma \in \Gamma} \sum_{e \in E(\gamma)} \left(\rho(e,\gamma)^{-1}\right)^{p-1} 
        \leq  \sum_{\gamma \in \Gamma} \sum_{e \in E(\gamma)} F(e)^{p-1} + \ell(\gamma)^{p-1} \\
        &= \sum_{e \in E} \sum_{\gamma \in \Gamma \colon e \in E(\gamma)} F(e)^{p-1} + \sum_{\gamma \in \Gamma} \sum_{e \in E(\gamma)} \ell(\gamma)^{p-1}
        = \sum_{e \in E} F(e)^p + \sum_{v \in V} \ell(\gamma_v)^{p} \\
        &= \Vert F(\cdot) \Vert_{l^p(E)}^p + \Vert \ell(\gamma_\cdot) \Vert _{l^p(V)}^p \leq (C_p+1) \Vert \dist(\cdot, \sigma(\cdot)) \Vert_{l^p(V)}^p
    \end{align*}
    and therefore
    \[
    \tilde{c}_p(\rho) \leq \tilde{C}_p \Vert \dist(\cdot,\sigma(\cdot))\Vert_{l^p(V)}.
    \]
\end{proof}

\subsection{The two-dimensional problem} 
\label{sec:discrete:2D}
We now consider the two-dimensional problem, which is not as direct as the one-dimensional problem, but still much easier than higher dimensions. Suppose that we are given $N^2$ vertices, $V= \{v=(i,j) \colon i,j =0,\ldots N-1\}$ and $E$ is the set of edges along coordinate lines, i.e.
\[
E=\{ \{v,w\} \colon v,w \in V \text{ and }\vert v - w \vert =1 \}
\]
where $\vert v- w \vert = \vert v_1 - w_1 \vert + \vert v_2 -w_2 \vert$ is the usual Manhattan distance.
A permutation $\sigma$ is a map from $V$ to $V$ and we measure $\dist(\cdot,\sigma(\cdot))$ also via the Manhattan metric. We denote by $\sigma_1$ the first coordinate of $\sigma$ and by $\sigma_2$ the second.

In contrast to the one-dimensional problem, there is not only one shortest path $\gamma$ from $v$ to $\sigma(v)$ and thus we have a choice.
Indeed, for $v=(v_1,v_2)$ to $w=(w_1,w_2) = \sigma(v)$ we consider the path $\gamma= \gamma_v$ defined by the following procedure. Consider $z=(w_1,v_2)$. There is a unique shortest path from $v$ to $z$ and a unique shortest path from $z$ to $w$. Those paths are edge-disjoint and define $\gamma$ as the concatenation of those paths. In other words, we choose the paths between $v$ and $w$ in such a way, that we first adjust the first coordinate and then the second.

Define $\Gamma$ to be the set of all those paths $\gamma_v$ and for an edge $e$ let 
\[ 
F(e) = \# \{ \gamma \in \Gamma \colon e \in E(\gamma) \}.
\]
As in \eqref{def:flow:1D} we define
\begin{equation} \label{def:flow:2D}
    \rho(e,\gamma) = \begin{cases}
        \Bigl(\max \left( F(e), \ell(\gamma) \right)\Bigr)^{-1} &\text{ if } e\in E(\gamma), \\
        0 & \text{ otherwise.}\end{cases}
\end{equation}
We need to show the counterpart of Lemma \ref{lemma:1D:1} first.
\begin{lemma} \label{lemma:2D:1}
   Consider the previous $2D$ setup of $G=(V,E)$, paths $\gamma$ and $\rho$. Then
   \begin{enumerate} [label=(\roman*)]
       \item \label{lemma:2D:1:a} If $e=\{(i,k),(i+1,k)\}$ then $F(e) = \# \{ j >i \colon \sigma_1(j,k) \leq i \} + \# \{ j \leq i \colon \sigma_1(j,k) > i\}$;
       \item \label{lemma:2D:1:b} If $e=\{(i,k),(i,k+1)\}$ then $F(e) = \# \{ j>k \colon (i,j) = \sigma(v) \text{ and } v_2 \leq k \} +  \# \{ j<k+1 \colon (i,j) = \sigma(v) \text{ and } v_2>k \};$
       \item \label{lemma:2D:1:c} $(\gamma,\rho)$ is admissible (for Problem \ref{problem:inc1}).
   \end{enumerate}
\end{lemma}
\begin{proof} 
    Item \ref{lemma:2D:1:a} and \ref{lemma:2D:1:b} follow by a simple counting argument. In \ref{lemma:2D:1:a}, $e \in E(\Gamma_v)$ if and only if $v=(j,k)$ and either $j \leq i$ and $\sigma_1(v) > i$ \emph{or} $j \geq i+1$ and $\sigma_1(v) \leq i$. A similar argument works for \ref{lemma:2D:1:b}.

    Admissibility follows by the careful definition of $\rho$, cf. \eqref{def:flow:2D}. In particular,  as in the proof of Lemma \ref{lemma:1D:1} the time constraint follows by $\rho(e,\gamma) \leq \ell(\gamma)^{-1} = \# E(\gamma)^{-1}$ for any path $\gamma$ and any $e \in E(\gamma)$. The capacity constraint follows by $\rho(e,\gamma) \leq F(e)^{-1}$ for each $\gamma$ such that $\gamma \in F(e)$.
\end{proof}
The next result is the counterpart of Lemma \ref{lemma:1D:2}.
\begin{lemma}
    Let $G=(V,E)$, $\gamma$ and $\rho$ be as before. Then:
    \begin{enumerate}  [label=(\roman*)]
        \item We have $\Vert F(\cdot) \Vert_{l^1(E)} = \Vert\dist(\cdot,\sigma(\cdot)) \Vert_{l^1(V)}$;
        \item $\Vert F(\cdot) \Vert_{l^{\infty}(E)} \leq 2 \Vert \dist(\cdot,\sigma(\cdot)) \Vert_{l^{\infty}(V)}$;
        \item for any $1 \leq p \leq \infty$ we have $\Vert F(\cdot) \Vert_{l^p(E)} \leq C_p \Vert \dist(\cdot, \sigma) \Vert_{l^p(V)}$.
    \end{enumerate}
\end{lemma}
\begin{proof}
    The $l^1$-equality again is shown by counting the same number in different ways:
     \begin{align*}
        \Vert F(\cdot) \Vert_{l^1(E)} =& \sum_{e \in E} \# \{ \gamma \colon e \in E(\gamma) \} = \sum_{v \in V} \# \{ e \colon e \in E(\gamma_v) \} = \Vert \dist(\cdot,\sigma(\cdot)) \Vert_{l^1(V)}.
    \end{align*}
    The $l^{\infty}$-inequality is slightly more intricate. Let $K=\Vert \dist(\cdot, \sigma(\cdot)) \Vert_{l^{\infty}}$. If $e= \{(i,k),(i+1,k)\}$, we infer that $e \in E(\gamma_v)$ implies $i-K < v_1 < i+K+1$ and $v_2=k$ therefore $F(e) \leq 2K$. On the other hand, if $e=\{(i,k),(i,k+1)$, we infer that $e \in E(\gamma_v)$ implies that $k-K < (\sigma(v))_2 < k+K-1$ and $(\sigma(v))_1 =i$, such that $F(e) \leq 2K$. We therefore get
    \[
    \Vert F(\cdot) \Vert_{l^{\infty}(E)} = \sup_{e \in E} F(e) \leq 2K = 2 \Vert \dist(\cdot,\sigma(\cdot) \Vert_{l^{\infty}(V)}.
    \]
    The $l^p$-inequality again follows by interpolation, see the proof of Lemma \ref{lemma:1D:2}.
\end{proof}
Combining those as before we get with exactly the same proof as in the 1D case the following theorem.
\begin{thm} \label{thm:2D}
    Let $G=(V,E)$ be the $2$-dimensional grid, let $\gamma_v$ be the path adjusting the first coordinate and then the second, $v$ to $\sigma(v)$ and let $\rho$ be as in \eqref{def:flow:2D}. Then the set of paths $\gamma$ and $\rho$ is admissible for the Problem \ref{problem:inc1} and there is a constant $C>0$ such that for any $1 \leq p \leq \infty$
    \[
    \tilde{c}_p(\rho) \leq C \Vert \dist(\cdot ,\sigma(\cdot))\Vert_{l^p(V)}.
    \]
\end{thm}

\subsection{Higher dimensions} \label{sec:discrete:multiD}

The situation in higher dimensions is different from the one- and two-dimensional results. Fixing notation, let $\nu >2$ and suppose that we are given $N^{\nu}$ vertices $V=\{v=(v_1,...,v_\nu) \in \{0,\ldots,N-1\}^{\nu}\}$ and let $E$ be the set of edges along coordinate lines, i.e.
\[
E = \{\{v,w\} \colon v,w \in V \text{ and } \vert v - w \vert=1\}.
\]

If we were to follow the same strategy as for the two-dimensional problem, we directly encounter nasty examples that prevent a result \`{a} la Theorem \ref{thm:1D} \& \ref{thm:2D}.
\begin{thm} \label{thm:counterexample} 
    Suppose that $G=(V,E)$ is the $\nu$-dimensional grid, $\nu>2$. For $v = (v_1,...,v_{\nu})$ and $w= (w_1,...,w_{\nu})$ define the paths $\gamma_{vw}$ from $v$ to $w$ as follows. Set $v^0=v$, $v^{\nu}=w$, $v^j = (w_1,...,w_j,v_{j+1},...,v^{\nu})$, such that there is a unique shortest path (along coordinate lines) from $v^j$ to $v^{j+1}$. Define $\gamma_{vw}$ as the concatenation of those paths.

    Then there exists a permutation $\sigma \colon V \to V$ such that for paths $\gamma_{v\sigma(v)}$ among all admissible $\rho$  we have
    \[
    \tilde{c}_p(\rho) > C(p) N^{1-1/p} \Vert \dist(\cdot,\sigma(\cdot)) \Vert_{l_p(V)}.
    \]
\end{thm}
We give the proof for $\nu=3$ here, the extension to higher dimension is straightforward. Further observe that in the endpoint case $p=1$ such an estimate is fine, but for any other $p$ we get a scaling with the grid size, which is not desirable.
 \begin{figure}
     \centering \includegraphics[width=0.5 \textwidth]{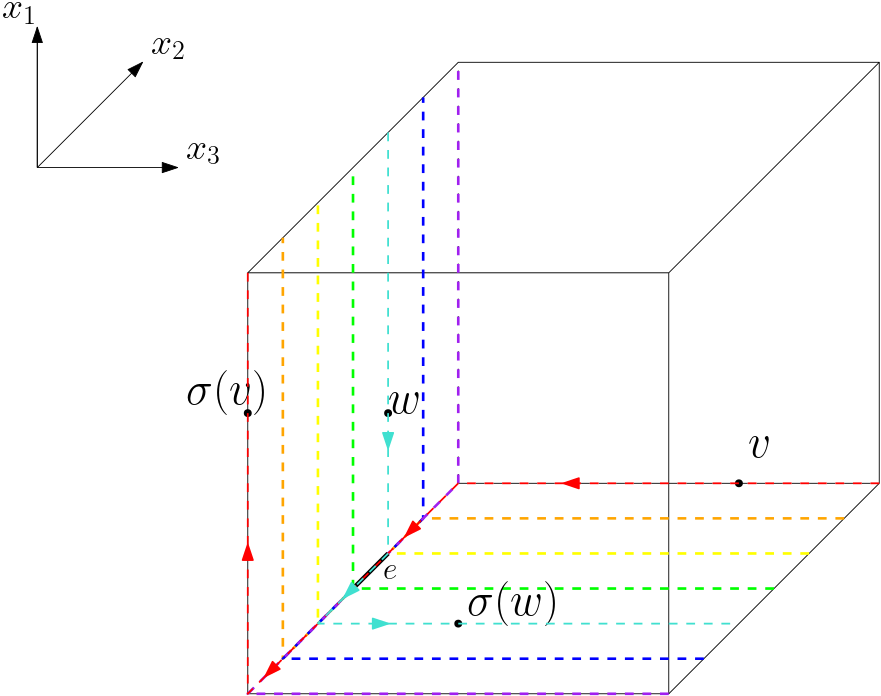}
     \caption{Visualisation of Theorem \ref{thm:counterexample}. The permutation $\sigma$ and the paths from $v$ to $\sigma(v)$ in different colors. Note that the number of paths passing through an edge like $e$ is $\sim N^2$. }
    \end{figure}
\begin{proof}
   We routinely write a vertex $v$ as $(v_1,v_2,v_3) \in \{0,...,N-1\}^3$  and define the permutation $\sigma$ as follows: 
   
    \[
    \sigma((v_1,v_2,v_3)) = (v_3,N-1-v_2,v_1).
    \]
    Then, a path from $v$ to $\sigma$ of $v$ consists of a path from $v$ to $(v_3,v_2,v_3)$, a path from $(v_3,v_2,v_3)$ to $(v_3,N-1-v_2,v_3)$ and a path from $(v_3,N-1-v_2,v_3)$ to $\sigma(v)$.
    Consider an edge $e=\{(i,j,i),(i,j+1,i)\}$. The number $F(e)$ of paths $\gamma_{v,\sigma(v)}$ with $e \in E(\gamma_{v,\sigma(v)})$ is
    \[
    2N \cdot \min \{ j,N-1-j\}.
    \]
    In particular,
    \begin{equation} \label{def:Etilde}
    \# \{e \in E(G) \colon F(e) \geq N^2 \} \geq N^2/2.
    \end{equation}
    Let us call this set of edges in \eqref{def:Etilde} $\tilde{E}$. 
    
    Now recall the capacity constraint \eqref{eq:capacity} for any admissible $\rho$. We have for any  edge in $e \in E$
    \[ 
    \sum_{\gamma \colon e \in E(\gamma)} \rho(e,\gamma) \leq 1.
    \]
    and therefore, for any edge $e \in \tilde{E}$
    \[
    \inf_{\gamma \colon e \in E{\gamma}} \rho(e,\gamma) \leq N^{-2}.
    \]
    We conclude that for any admissible $\rho$ with those chosen paths
    \[
    \tilde{c}_{\infty} (\rho) \geq  N^{2} \geq N \Vert \cdot -\sigma(\cdot) \Vert_{l^{\infty}(V)}. 
    \]
    For $1 \leq p< \infty$ observe that
    \begin{align*}
        \tilde{c}_p(\rho)^p &\geq  \sum_{e \in \tilde E} \sum_{\gamma \colon e \in E(\gamma)} (\rho(e,\gamma)^{-1})^{p-1} 
        \\ & \geq  \sum_{e \in \tilde E} F(e) \cdot (F(e))^{p-1} \geq \sum_{e \in \tilde{E}} N^{2p} \geq \tfrac{1}{2} N^{2p+2}
    \end{align*}
    and that 
    $\Vert \cdot - \sigma(\cdot) \Vert{^p}_{l^p(V)} \leq C  N^3 N^p$. Therefore,
    \[
    \tilde{c}_p(\rho) \geq N^{2+\tfrac{2}{p}} \geq N^{1-1/p} N^{1+3/p} \geq C \Vert \cdot - \sigma(\cdot) \Vert_{l^p(V)}.
    \]
\end{proof}
Instead of considering Problem \ref{problem:inc1} where we have to make the choice of a specific path, we consider Problem \ref{problem:inc2}, where we assume that each particle can choose a path of minimum length from a set of minimal paths chosen below. Denoting by $\Gamma_v$ the set of those paths, we assume that every paths in $\Gamma_v$ has the same weight, i.e.
\begin{equation}\label{assumpt:weights}
    \omega(\gamma_v)=\frac{1}{\# \Gamma_v},
\end{equation}
in particular, the decomposition condition \eqref{eq:decomp} immediately holds true. 
\begin{equation}\label{eq:F:e:multi:d}
    F(e)=\sum_{\gamma\in\Gamma_e}\omega(\gamma),
\end{equation}
where we recall that $\Gamma_e=\lbrace \gamma\in\Gamma: e\in E(\gamma)\rbrace.$ Then, similarly as done before, we define
\begin{equation} \label{def:flow:multi:D}
    \rho(e,\gamma) = \begin{cases}
        \Bigl(\max \left( F(e), \ell(\gamma) \right)\Bigr)^{-1} & \text{if } e\in E(\gamma), \\
        0 & \text{otherwise.}
    \end{cases}
\end{equation}
We state now the equivalent of Lemmas \ref{lemma:1D:1} and \ref{lemma:1D:2}:
\begin{lemma}\label{lemma:multi:D:1}
    Let $G=(V,E)$ be the $\nu$-dimensional grid. Suppose that, for each point, $\Gamma_v$ is a set of shortest paths from $v$ to $\Gamma_v$. Let $\rho$ and $\omega$ be as in \eqref{def:flow:multi:D} and \eqref{assumpt:weights}, respectively.
    Then $(\gamma,\rho,\omega)$ is admissible (for Problem \ref{problem:inc2}).
\end{lemma}
\begin{proof}
    Observe that the weights are set up in a way that they sum up to 1. Therefore it remains to check that the time and capacity constraint. The time constraint follows by
    \[
    \sum_{e \in E(\gamma)} \rho(e,\gamma) \leq \sum_{e \in E(\gamma)} \ell(\gamma) =1,
    \]
    and, likewise, the capacity constraint follows through
    \[
    \sum_{\gamma \colon e \in E(\gamma)} \omega(\gamma) \rho(e,\gamma) \leq \sum_{\gamma \colon e \in E(\gamma)} \omega(\gamma) F(e)^{-1} = 1.
    \]
\end{proof}
Without any clever choice of the set of paths we can further state the $l^1$-bound. Again, for $X \subset V$ define
\[
F_X(e) = \sum_{v \in X} \sum_{\{\gamma \in \Gamma_e \cap \Gamma_v\}} \omega(\gamma).
\]
\begin{lemma} \label{lemma:multi:D:2}
    Let $G=(V,E)$, $\gamma$, $\rho$, $\omega$ be as before. Suppose that $\Gamma_v$ only consists of shortest paths between $v$ and $\sigma(v)$. Let $X \subset V$. Then
   we have $$\Vert F_X(\cdot) \Vert_{l^1(E)} = \Vert \dist(\cdot,\sigma(\cdot)) \Vert_{l^1(X)}.$$
\end{lemma}
\begin{proof}
   This is, again, a matter of counting twice and using that $\Gamma_v$ only consists of shortest paths of lengths $\ell(\gamma) = \dist(v,\sigma(v))$. We therefore have
   \begin{align*}
       \Vert F_X(e) \Vert_{l^1(E)} &=  \sum_{e \in E} \sum_{v \in X} \sum_{\{\gamma \in \Gamma_e \cap \Gamma_v\}} \omega(\gamma)  \\
       &= \sum_{v \in X} \sum_{e \in E}  \sum_{\{\gamma \in \Gamma_e \cap \Gamma_v\}} \omega(\gamma) = \sum_{v \in X} \sum_{\gamma \in \Gamma_v} \omega(\gamma) \ell(\gamma) \\
       & = \sum_{v \in X} \dist(v,\sigma (v)) = \Vert \dist(\cdot,\sigma(\cdot))\Vert_{l^1(X)}.
   \end{align*} 
\end{proof}
This settles the $l^1$-bound. As shown in Theorem \ref{thm:counterexample}, not every choice of $\Gamma_v$ may give an $l^{\infty}$ (or even an $l^p$-bound). It can be proved that such an $l^{\infty}$-bound also \emph{does not} hold is the choice $\Gamma_v = \{ \text{all shortest paths from } v \text{ to } \sigma(v)\}$. Instead we need to come up with a more elaborate definition of $\Gamma_v$.
\begin{definition} \label{def:gammav}
    Let $v$ and $w = \sigma(v)$ be elements of $V = \{0,...,N-1\}^\nu$ . After suitable relabeling of coordinates and reflections we might assume 
    \begin{enumerate} [label=(\roman*)]
        \item $v_i \leq w_i$, $i =1,\ldots,\nu$;
        \item $w_1-v_1 \geq w_2-v_2 \geq \ldots \geq w_{\nu}-v_{\nu}$.
    \end{enumerate}
    Let $a_i = w_i-v_i$ and the index set $A = \{ \alpha \in \N^{\nu-2} \colon 0 \leq \alpha_i \leq a_i$\}.
    For each $\alpha \in A$ construct the path $\gamma^{\alpha}$ as the concatenation of the following paths
    \begin{enumerate}
        \item The shortest path from $v$ to $v^{1}_{\alpha}= v+\alpha_1 \un_1$;
        \item The shortest path from $v^{i-1}_{\alpha}$ to $v^i_{\alpha} := v + \sum_{j=1}^{i} \alpha_j \un_j$ for $i=2,...,\nu-2$;
        \item The shortest path from $v^{\nu-2}_{\alpha}$ to $v_{\mathrm{inter}} = v+ \sum_{j=1}^{\nu-2} \alpha_j \un_j + (w_{\nu-1}-v_{\nu-1}) \un_{\nu-1}$;
        \item The shortest path from $v_{\mathrm{inter}}$ to $w^{\nu-2}_{\alpha}$, where
        \[
        w^{\nu-2}_{\alpha} := v+ \sum_{j=1}^{\nu-2} \alpha_j \un_j + (w_{\nu-1}-v_{\nu-1}) \un_{\nu-1} + (w_{\nu}-v_{\nu}) \un_{\nu} = w - \sum_{j=1}^{\nu-2} (w_j-\alpha_j-v_j) \un_j;
        \]
        \item The shortest path from $w^{i}_{\alpha} = w- \sum_{j=1}^i (w_j -\alpha_j-v_j) \un_j$ to $w^{i-1}_{\alpha}$ for $i=\nu-2,\ldots,1$. Observe that $w^{0}_{\alpha}=w$.
    \end{enumerate}
\end{definition}
Let us reformulate what this definition means in dimension $3$. First, we pick the coordinate direction, where the distance between $v_i$ and $\sigma(v)_i$ is the largest. Then the path follows this direction for some distance $\alpha \in \{0,\ldots, \vert \sigma(v)_i - v_i \vert\}$ with probability $\tfrac{1}{\vert \sigma(v_i) + v_i \vert +1}$. Then the path corrects the second and the third coordinate and, finally, the path treavels the remaining $(\sigma(v_i)-v_i)-\alpha$ steps in the direction chosen at the beginning.
\begin{figure}
    \centering \includegraphics[width=0.4 \textwidth]{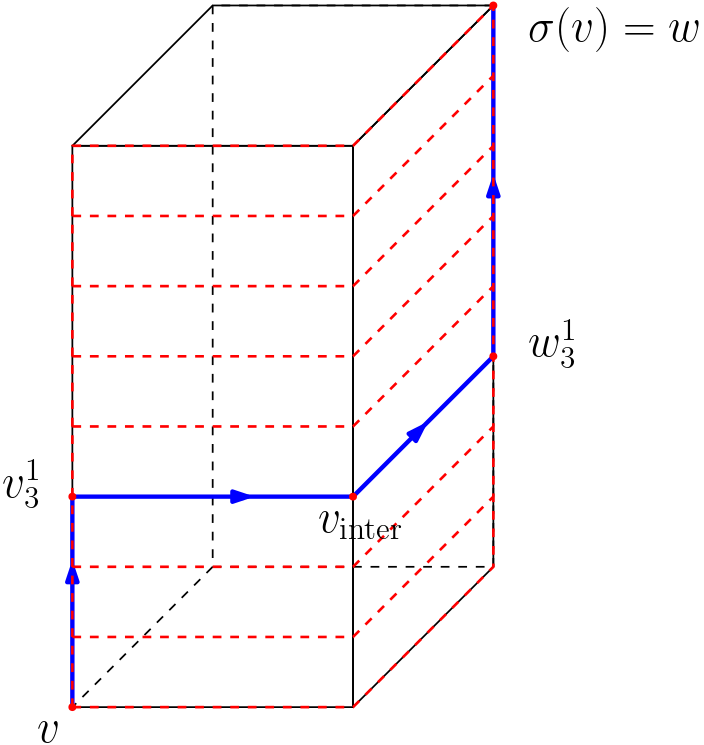}
    \caption{The set of paths from $v$ to $\sigma(v)$ in three dimensions. In blue one may see the path with $\alpha=3$.}
    \label{fig:Whitney}
\end{figure}

Recall that all those paths are unique, as these paths are along coordinate lines and that they are (apart from end-points) vertex-disjoint, such that the concatenation is well-defined.

We then have the following bounds:
\begin{lemma} \label{bounds:numberofpaths}
Suppose that $v$ and $w= \sigma(v)$ are as above such that $ w_1 - v_1 \geq \ldots \geq w_\nu - v_\nu$, $a_i = w_i - v_i$ and $v_i \leq w_i$. Then
\begin{enumerate} [label=(\roman*)]
    \item \label{bnop:1} The number of paths in $\Gamma_v$ is $\prod_{j=1}^{\nu-2} (a_j+1)$;
    \item \label{bnop:2} Suppose that $z$ is such that $v_i \leq z_i \leq w_i$ for all $i=1,\ldots,\nu$.
    Let $M(z) := \max \{ \# \{ i \colon z_i=v_i \}, \# \{ i \colon z_i=w_i\}\}$. Then
    \begin{enumerate}[label=(\alph*)]
        \item \label{bnop:2a}  If $M(z) =0$, then 
        \[
        \frac{\# \{ \gamma \in \Gamma_v \colon z \in V(\gamma)\}}{\# \Gamma_v} =0;
        \]
        \item  \label{bnop:2b} If $M(z) = \nu -1$, then
          \[
            \frac{\# \{ \gamma \in \Gamma_v \colon z \in V(\gamma)\}}{\# \Gamma_v} \leq 1;
         \]
         \item \label{bnop:2c} If $M(z) = i$, $i=1,\ldots,\nu-2$, then
         \[
             \frac{\# \{ \gamma \in \Gamma_v \colon z \in V(\gamma)\rbrace}{\# \Gamma_v} \leq \prod_{j=1}^{\nu-i-1} \frac{1}{a_{j}+1};
         \]
        \end{enumerate}
    \item \label{bnop:3} If $z \in V$ such that $z_i> w_i$ or $z_i < v_i$, then the number of paths in $\Gamma_v$ through $w$ is $0$.
\end{enumerate}
\end{lemma}
\begin{proof}
    The first assertion is clear, as this is the number of multiindices in Definition \ref{def:gammav}. By construction (and also by the fact that these paths are shortest), also \ref{bnop:3} is clear from the previous definition. It remains to show \ref{bnop:2}.

    For \ref{bnop:2a} observe that for any point on the paths from $v^{i-1}_{\alpha}$ to $v^{i}_{\alpha}$ and $w^{i}_{\alpha}$ to $w^{i-1}_{\alpha}$ at least one coordinate coincides with the same coordinate of $v$ and $w$ respectively. For the path from $v^{\nu-2}_{\alpha}$ to $v_{\mathrm{inter}}$ the last coordinate coincides with $v$, whereas on the path from $v_{\mathrm{inter}}$ to $w_{\alpha}^{\nu-2}$ the $(\nu-1)$th coordinate coincides with $w$. Therefore, if $z \in V(\gamma)$ we have $M(z) \geq 1$.

    \ref{bnop:2b} is obviously true.

    Suppose that $M(z) =i$. By a symmetric argumentation we may assume w.l.o.g. that $\#\{i \colon z_i = v_i \}=i$. Denote by $i'\leq i$ the largest number such that 
    \[
    z_j = v_j \quad \text{for all } j >\nu-i'.
    \]
    Then we can show that $z \in V(\gamma^{\alpha})$ only if $\alpha_j = z_j-v_j$ for all $j <\nu-i'$.  The number of those multiindices $\alpha$ is exactly
    \[
        \prod_{j=\nu-i'}^{\nu-2} (a_j+1),
    \]
    and therefore,
    \[
    \#\{ \gamma \in \Gamma_v \colon z \in V(\gamma) \} \leq \prod_{j=\nu-i'}^{\nu-2} (a_j+1) \leq \prod_{j=\nu-i}^{\nu-2} (a_j+1).
    \]
    Dividing this by the result of \ref{bnop:1} we get the desired
    \[
   \frac{ \#\{ \gamma \in \Gamma_v \colon z \in V(\gamma) \}}{\# \Gamma_v } \leq \prod_{j=1}^{\nu-i-1} \frac{1}{a_j+1}.
    \]
\end{proof}
We are now ready to prove the following:
\begin{lemma}
Let $X \subset V$ and let $F_X(e)$ be as before. Let $\Gamma_v$ be as in Definition \ref{def:gammav} and $\omega$ as in \eqref{assumpt:weights}. There is a dimensional constant $C(\nu)$ such that
\[
\Vert F_X (\cdot) \Vert_{l^{\infty}(E)} \leq C(\nu) \Vert \sigma-\id \Vert_{l^{\infty}(X)}.
\]
\end{lemma}
\begin{proof}
Let $e = \{z_1,z_2\}$ for some $z_1,z_2 \in V$. For $z \in V$ consider
\[
\tilde{F}(z) = \sum_{\gamma \in \Gamma \colon v \in V(\gamma)} \omega(\gamma).
\]
Observe that $F(e) \leq \min \{ \tilde{F}(z_1), \tilde{F}(z_2)\}$, hence it suffices to show an $l^{\infty}$ bound on $\tilde{F}$ instead. Let $K = \Vert \sigma - \id \Vert_{l^{\infty}(X)}$.

Fix now $z \in V$. Let $v \in X$, such that more than zero paths from $v$ to $\sigma(v)$ touch $z$. We conclude that $\vert v_i - z_i \vert \leq K$ and $\vert v_i - \sigma(v)_i \vert \leq K$ for all $i=1,\ldots,\nu-2$. Let $Q_K$ be the cube of all $v \in V$ such that $\vert v_i - z_i \vert \leq K$ for all $v \in K$.
For each $v \in V$ define $\tilde{a}_i(v,z) =\vert v_i -z_i \vert $ and let 
$a_i$ be the resorting of those numbers such that $a_1(v,z) \geq a_2(v,z) \geq \ldots \geq a_{\nu}(v,z)$. Again, let $M(v,z)$ be the largest number $i$, such that $a_{\nu+1-i} =0$. Observe that if $z \in V(\gamma^{\alpha}_v)$ for some $\gamma^{\alpha}_v \in \Gamma_v$ then $M(v,z) \geq 1$ or $M(\sigma(v),z) \geq 1$ (cf. Lemma \ref{bounds:numberofpaths} \ref{bnop:2a}).
Then we have for a single paths,
\begin{align*}
\sum_{ \gamma \in \Gamma_v \colon z \in V(\gamma) } \omega(\gamma) & \overset{\eqref{assumpt:weights}}{=} \frac{ \# \{ \gamma \in \Gamma_v \colon z \in V(\gamma)\rbrace}{\# \Gamma_v} \\
&\overset{\text{\ref{bounds:numberofpaths} \ref{bnop:2c}}}{\leq} \prod_{j=1}^{\nu-1-\max\{M(v,z),M(\sigma(v),z)} \tfrac{1}{a_j(\sigma(v),v)} \\
&\leq \prod_{j=1}^{\nu-1-M(v,z)} \tfrac{1}{a_j(v,z)+1} +  \prod_{j=1}^{\nu-1-M(v,\sigma(v))} \tfrac{1}{a_j(\sigma(v),z)+1}.
\end{align*}

We conclude 
\begin{align*}
    \tilde{F}(z) & \leq \sum_{v \in Q_K} \sum_{\gamma \in \Gamma_v \colon z \in V(\gamma)} \omega(\gamma) \\
    & \leq \sum_{v \in Q_K} \prod_{j=1}^{\nu-1} \tfrac{1}{a_j(v,z)+1} +  \prod_{j=1}^{\nu-1} \tfrac{1}{a_j(\sigma(v),z)+1} \\
    & \leq 2 \nu \sum_{v \in Q_K \colon v_\nu = z_{\nu}}  \prod_{j=1}^{\nu-1} \tfrac{1}{a_j(v,z)+1} 
     \leq 2^{\nu}\nu \sum_{v \in z + \{0,...,K\}^{\nu-1}} \prod_{j=1}^{\nu-1} \tfrac{1}{a_j(v,z)+1} \\
    &\overset{(\ast)}{=} 2^{\nu}\nu \sum_{i_1=0}^K \sum_{i_2=0}^{i_1} \ldots  \sum_{i_{\nu-1}=0}^{i_{\nu-2}}  \prod_{j=1}^{\nu-2} \tfrac{1}{i_j+1}
    \\
  & \leq 2^{\nu} \nu \sum_{i_1=0}^K \sum_{i_2=0}^{i_1} \tfrac{1}{i_1+1} \sum_{i_3=0}^{i_2} \tfrac{1}{i_2+1} \ldots \sum_{i_{\nu-1}=0}^{i_{\nu-2}} \tfrac{1}{i_{\nu-2}+1} \\
  &\leq 2^{\nu} \nu (K+1) \leq 2^{\nu+1} \nu K.
\end{align*}
We emphasise that the equality in $(\ast)$ comes from the careful definition and sorting of the coordinate directions in Definition \ref{def:gammav}: We sorted the coordinate directions such that $a_1 \geq a_2 \geq \ldots \geq a_{\nu-1}$.
\smallskip

Therefore,
\[
\Vert F_X \Vert_{l^{\infty}(E)} \leq \Vert \tilde{F} \Vert_{{l^{\infty}(V)}} \leq 2^{\nu+1} \nu K = 2^{\nu+1} \nu \Vert \sigma - \id \Vert_{l^{\infty}(X)}.
\]
\end{proof}
By the same interpolation proof as in Lemma \ref{lemma:1D:2} (which does \emph{not} use dimension or any weights) we get the following result.
\begin{thm} \label{thm:multiD}
    Suppose that $G=(V,E)$ is the $\nu$-dimensional grid, let $\sigma \colon V \to V$ be bijective. Define $\Gamma_v$ as in Definition \ref{def:gammav}, weights $\omega$ as in \eqref{assumpt:weights} and $\rho$ as in \eqref{def:flow:multi:D}. Then
    \begin{enumerate} [label=(\roman*)]
        \item \label{thm:multiD:1} for $F$ as in \eqref{eq:F:e:multi:d} we have for any $p \in [1,\infty]$ 
        \[
        \Vert F \Vert_{l^{p}(E)} \leq 2^{\nu} \nu \Vert \sigma - \id \Vert_{l^p(V)};
        \]
        \item  \label{thm:multiD:2} $(\gamma,\rho,\omega)$ is admissible for Problem \ref{problem:inc2};
        \item  \label{thm:multiD:3} we have the inequality
        \[
        c_p(\rho,\omega) \leq  (1+2^{\nu} \nu) \Vert \dist(v,\sigma(v)) \Vert_{l^p(V)}.
        \]
    \end{enumerate}
\end{thm}
\begin{proof}
    As advertised, \ref{thm:multiD:1} follows with the same interpolation proof as Lemma \ref{lemma:1D:2}. We already showed the second assertion \ref{thm:multiD:2} in Lemma \ref{lemma:multi:D:1}. It remains to show the last one.

    \noindent For $p=\infty$ observe that for any $e$ and $\gamma$ we have
    \[
    \rho(e,\gamma)^{-1} \leq F(e) + \ell(\gamma) \leq \Vert F \Vert_{l^{\infty}(E)} + \Vert \sigma - \id \Vert_{l^{\infty}(V)} \leq (2^{\nu} \nu +1) \Vert \sigma - \id \Vert_{l^{\infty}(V)}.
    \]
    For $1 \leq p <\infty$ we have
    \begin{align*}
        c_p(\rho,\omega) &= \left(\sum_{\gamma \in \Gamma} \sum_{e \in E(\gamma)} \omega(\gamma) \rho(e,\gamma)^{-p+1} \right)^{1/p} \\
        & \leq  \left(\sum_{\gamma \in \Gamma} \sum_{e \in E(\gamma)} \omega(\gamma) (F(e)^{p-1} + \ell(\gamma)^{p-1}) \right)^{1/p} \\
        &\leq \left( \sum_{v \in V} \sum_{\gamma \in \Gamma_v} 
 \sum_{e \in V(\gamma)} \omega(\gamma) \vert v- \sigma(v) \vert^{p-1} + \sum_{e \in E} \sum_{\gamma \colon e \in E(\gamma)} \omega(\gamma) F(e)^{p-1} \right)^{1/p} \\
 &\leq \left( \sum_{v \in V} \vert v -\sigma(v) \vert^p + \sum_{e \in E} F(e)^p \right)^{1/p} \\
 &\leq \Vert \sigma-\id \Vert_{l^p(V)} + \Vert F \Vert_{l^p(E)} \\
 &\leq (2^{\nu}\nu+1) \Vert \sigma-\id \Vert_{l^p(V)},
    \end{align*}
finishing the proof.
    
\end{proof}

%% file: sec4.tex
\section{Flow for discrete configurations \& Setup} \label{sec:4}

We have solved a discretised version of the flow problem in Section \ref{sec:3}. In this section we translate the solution of the discrete problem back into a flow defined on the cube 
\[M=[0,1]^{\nu}.
\]
We recall the notation. We take $N \in \N$ and consider $N^{\nu}$ small cubes $$ \q_v:= N^{-1}v + N^{-1} [0,1)^{\nu},$$ where $v \in \{{0},\ldots,N-1\}^\nu$ is a multi-index.

For the construction we use pick parameters $\kappa \in \N$ and $\ell>0$, such that $\ell$ obeys \[
10^3 \ell < N^{-1}
\]
thus if we set $K=N^{-1}/{\ell} \in \N$ we require $K>10^3$ and $\kappa$ is chosen such that $\kappa > 16 N^{-1} \ell^{-1}$.
Further define 
\[
\lambda := \kappa^{-1} \ell.
\]
We shall mention that our choice of parameters is not optimal by any means. One may improve the constant of Theorem \ref{thm:main} by choosing constants more carefully, but, in contrast to the H\"older exponent $\alpha$, we do not focus on the optimal constant. We moreover refer to Figure \ref{fig:Martina} for a visualisation of all the choices.

We identify $V= \{0,...,N-1\}^\nu$ as the vertex set and $E$ as in Section \ref{sec:discrete:multiD} as the edge set. A \emph{permutation} $\sigma$ is a bijective map $V \to V$. For given $\sigma$, we associate a map $\psi_{\sigma} \colon M \to M$ as follows:
\begin{equation} \label{def:psisigma}
    \psi_{\sigma} (x)= N^{-1} \sigma(v) + (x- N^{-1} v) \quad \text{if } x \in \q_v,~v \in V. 
\end{equation}

Observe that we have the following equivalence of norms:
\begin{equation} \label{def:equiv}
    \Vert \psi_{\sigma} - \Id_M \Vert_{L^p(M)} \sim N^{-1} N^{-\nu/p} \Vert \sigma - \Id_V \Vert_{l^p(V)}.
\end{equation}
The goal of this and the following sections is to prove the following theorem.
\begin{thm} \label{thm:goal}
    Let $\sigma$ be a permutation on $V$, $N \in \N$ and $\nu \geq 3$. There exists a divergence free vector field $u \in L^\infty([0,1];L^\infty(M;\R^\nu)) \cap L^1([0,1];BV(M))$ and a map $\psi \in W^{1,\infty}([0,1];L^\infty(M;M))$ such that 
    \begin{enumerate} [label=(\roman*)]
        \item \label{flow:1} $\partial_t \psi(t,x) = u(t,\psi(x))$;
        \item \label{flow:2} $\psi(0,\cdot) = \Id_{M}$;
        \item \label{flow:3} $\psi(1,\cdot)$ = $\psi_{\sigma}$;
        \item \label{flow:4} $\diverg_x u=0$ in the sense of distributions for a.e. time $t>0$.
    \end{enumerate}
    and such that the following bound holds for any $p,q \in [1,\infty]$:
   \begin{equation} \label{eq:Lpbound}
       \Vert u \Vert_{L^q([0,1];L^p(M))} \leq C(p,q,\nu) \Vert \psi_{\sigma} - \Id_{M} \Vert_{L^p(M)}.
   \end{equation}
\end{thm}

Here, we require the $\BV$ regularity of the vector field in order to guarantee existence and uniqueness of the flow $\psi$ \cite{Ambrosio:BV}. In order to recover Corollary \ref{coro:main}, we will perform a smoothing procedure in Section \ref{sect:from:discrete:to:continuous}. We also point out that all flows constructed can be possibly made smooth in the interior of cubes and pipes, but we keep the $\BV$ regularity in order to give explicit constructions.

Let us remark that the constant in the construction will depend on the artificial choices made for the parameters $\kappa$ and $\lambda$, which do not depend on the permutation $\sigma$.

Nevertheless, the majority of this section and the following is devoted to a slightly simplified setup of two-dimensional flows in $3D$. In particular, consider 
\[
\sigma \colon V= \{0,\ldots,N-1\}^3 \to V
\]
with $\sigma_3(v_1,v_2,v_3) =v_3$. Let us call this set of permutations $S^2$.

The key intermediate step is as follows:
\begin{lemma}[Partial construction of flows] \label{keylemma}
Let $\sigma \in S^2$. There exists a divergence-free vector field $u \in L^\infty([0,1];L^\infty(M;\R^3))$ and a map $\psi \in W^{1,\infty}([0,1];L^\infty(M;M))$ such that 
    \begin{enumerate} [label=(\roman*)]
        \item \label{flow:1a} $\partial_t \psi(t,x) = u(t,\psi(x))$;
        \item \label{flow:2a} $\psi(0,\cdot) = \Id_{M}$;
        \item \label{flow:3a} $\psi(1,\cdot)$ = $\psi_{\sigma,\kappa}$;
        \item \label{flow:4a} $\diverg_x u=0$ in the sense of distributions for a.e. time $t>0$,
    \end{enumerate}
with 
\begin{equation} \label{eq:Lpbound2}
       \Vert u \Vert_{L^1([0,1];L^p(M))} \leq C(p) \Vert \psi_{\sigma} - \Id_{M} \Vert_{L^p(M)}
   \end{equation}
and, for $\mathcal{Q}_v = N^{-1} v + [0,\ell)^3$, we have
\begin{equation} \label{eq:Qvkappa}
\psi_{\sigma,\kappa} = \begin{cases}
    \id  & \text{if } x \in  \q_v \setminus \mathcal{Q}_{v}, \\
    \psi_{\sigma} & \text{if } x \in \mathcal{Q}_{v}.
\end{cases}
\end{equation}
\end{lemma}
\begin{figure}
    \centering
    \includegraphics[width=0.5 \textwidth]{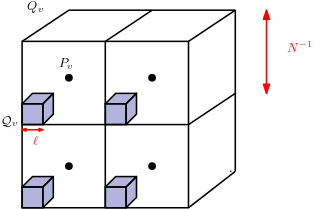}
    \caption{Each cube $Q_v$ of the tiling is subdivided into $K^3$ cubes $\mathcal{Q}_v$, each one of sidelength $\ell$. Each cube connects to the point $P_v$ of the network.}
    \label{fig:Martina}
\end{figure}

In other words, \eqref{eq:Qvkappa} entails that only parts of the cubes are moved to the right direction and (after time $t=1$) every other particle is at its original position. The main result follows by rescaling in time and applying this result $K^3$ times to all subcubes of a single cube $\q_v$.

In order to prove Lemma \ref{keylemma} we first define some basic terminology for later.

\subsection{Sources, Sinks \& Checkpoints} We aim to define the vector field $u$ in two phases (i.e. time intervals), in which $u$ will be constant-in-time. In the first phase, we connect small subcubes by pipes and construct a flow that transports cube $\q_{v}$ to the cube $\q_{\sigma(v),\kappa}$. In the second time step we then correct the flow in such a way that it is the identity on the pipes again.

In order to define $u$, we use piecewise smooth functions. Instead of giving a full direct construction of a flow from one cube to another, where it is difficult to verify that flows do not intersect, we use a step-by-step construction on smaller time- and space-frames.

\begin{definition} \label{def:sourcesink}
Suppose that $A, B \subset M$ are two open disjoint $2$-dimensional rectangles with normals $\nu_A,\nu_B \in \R^3$. Let $u_A$ and $u_B$ be two constant vectors such that $u_A = \mu_A \nu_A$ and $u_B = \mu_B \nu_B$, $\mu_A, \mu_B>0$ and 
\[
\int_A \nu_A \cdot u_A \dH^2 = \int_B \nu_B \cdot u_B \dH^2 \quad \Longleftrightarrow \quad \mu_A \Haus^{2}(A) = \mu_B \Haus^2(B).
\]
We say that a piecewise smooth $u \colon M \to \R^3$ induces a flow from $A$ to $B$ if 
\begin{enumerate} [label=(\roman*)]
    \item \label{def:sourcesink:1}$\diverg(u) = - \Haus^{2} \llcorner A \cdot u_A+ \Haus^2 \llcorner B \cdot u_B$ in the sense of distributions.
    \item There is a time $t_{A,B}$ such that we have: for any $z \in A$ the solution to the ordinary differential equation
    \begin{equation} \label{eq:ODE}
        \begin{cases}
            \partial_t \psi^z(t) = u(\psi^z(t)) & t \in [0,t_{A,B}),
            \\
            \psi^z(0) = z,
        \end{cases}
    \end{equation}
    satisfies $\psi^z(t_{A,B}) \in B$ and the map $\psi \colon z \mapsto \psi^z(t_{A,B})$ is affine.
    \item $\spt(u) = \psi([0,t_{A,B}]\times A)$.
\end{enumerate}
We call $\psi$ the \textbf{source-sink flow} from \textbf{source} A and \textbf{sink} B and, correspondingly, $u$ the \textbf{source-sink vector field.}
\end{definition}
We gather the following simple observations into the following lemma.
\begin{lemma}[Union and concatenation of source-sink-flows]
\begin{enumerate} [label=(\roman*)]
    \item Suppose that\\ $A_1,\ldots,A_k$ and $B_1,\ldots,B_k$ are disjoint 2-dimensional rectangles and that there are source-sink vector fields $u_j$ from $A_j$ to $B_j$. If the supports of $u_j$ are pairwise disjoint then we may define
    \[
    u \colon M \to \R^n, \quad u(x)=u_j(x) \quad x \in \spt(u_j)
    \]
    and we still have
    \begin{enumerate} [label=(\alph*)]
        \item $\diverg(u)  = -\sum_{j=1}^k \Haus^2 \llcorner A \cdot u_A + \sum_{j=1}^k \Haus^2 \llcorner B \cdot u_B$ in the sense of distributions,
        \item For any $z \in A_k$ the solution to the ordinary differential equation
          \[  \begin{cases}
            \partial_t \psi^z(t) = u(\psi^z(t)) & t \in [0,t_{A_k,B_k}),
            \\
            \psi^z(0) = z,
        \end{cases}
        \]
        is the same as for the source-sink flow $u_k$.
    \end{enumerate}
    \item Suppose that $A_1,\ldots,A_k$ are disjoint 2-dimensional rectangles and that there are source-sink flows $u_j$ from $A_j$ to $A_{j+1}$ with normals $\nu_{j,A_j}$ and $\nu_{j,A_{j+1}}$ and velocities $u_{j,A}$ and $u_{j,A_{j+1}}$. Suppose that
    \begin{equation*}
         \nu_{j-1,A_j} = - \nu_{j,A_j}, \quad   u_{j-1,A_j} = - u_{j,A_{j}}, \quad j=2,...,n-1
    \end{equation*}
    and that $\spt(u_j) \cap \spt(u_{j+1}) = A_{j+1}$, $\spt(u_i) \cap \spt(u_j) = \emptyset$ if $\vert i- j \vert \geq 2$.
    Then we may define
    \[
    u \colon M \to \R^n, \quad u(x)=u_j(x) \quad x \in \spt(u_j)
    \]
   and $u$ is a source-sink flow from $A_1$ to $A_k$. We then call $A_2,...,A_{k-1}$ \textbf{checkpoints} of the source-sink flow.
\end{enumerate}
\end{lemma}

In the following section, we engineer a source-sink flow that connects a cube $\Qcal_v$ to a cube $\Qcal_{\sigma(v)}$ by finding suitable checkpoints on the way and constructing the flow in a piece-wise fashion. Observe that due to the first part of the lemma, if we can ensure that all those flows have disjoint support, then we may globally define the flow.

Apart from a few other checkpoints to have a nice immediate in- and out-flow of the cube, we of course use the paths given by the discrete problem of Section 3 to get good intermediate checkpoints and construct flows along 'edges' of a graph. We recall the exact terminology when it is needed.

%------------------------------------------------
%------------------------------------------------
%------------------------------------------------
%------------------------------------------------
\section{Construction of flows} \label{sec:construction}
In this section, we give a number of source-sink-flows that achieve the goal for the first-phase: After time $t=1/2$ the flow has transported cubes $\Qcal_{v}$ to $\Qcal_{\sigma(v)}$.
The problem is that when transporting from cube to cube, we have 'en passant' also transported some fluid from one pipe to another. This is then rectified between time $t=1/2$ and $t=1$, by letting the fluid 'flow backwards' inside of the pipes.

The construction is roughly divided into two major parts: First, we need to construct a flow that connects our cube to a discrete network (the '\emph{connector flow}'). Second, we construct the flow along the discrete network (which we will call '\emph{highway flow}') to get from one vertex $v$ to $\sigma(v)$. This is connected to the cube $\Qcal_{\sigma(v)}$ via the connector flow.
\subsection{Special Source-Sink flows}
During the construction we need basically three types of  source-sink flows: Two flows that are between two sources and sinks that are parallel and flows that are between perpendicular source and sink. Before coming to the application, we construct these in a rather abstract setting.
\begin{lemma}[Parallel source-sink flow: Transport] \label{lem:parallel}
    Suppose that $A=(0,a_1) \times (a_2,a^2) \times \{0\}$ and that $B=(D_1,D_1+a_1) \times (a_2,a^2) \times \{D_3\}$
     Suppose that $(A,\nu_A,u_A)$ and $(B,\nu_B,u_B)$ obey the properties of Definition \ref{def:sourcesink}, i.e. $\nu_A=\nu_B= \un_3$ and, thus, $u_A=u_B$. Suppose further that
    \[
    \Psi \colon A \to B, \quad \Psi(x_1,x_2,0) = \left(x_1+D_1,x_2,D_3 \right).
    \]
    Then there is $u \colon [0,t_{A,B}] \times \R^3 \to \R^3$ constant-in-time inducing a source-sink flow $\psi$ from $A$ to $B$ with the following properties:
    \begin{enumerate} [label=(\roman*)]
        \item $u(t,y)=0$ if $y \notin \conv(A,B)$;
        \item $u(t,\cdot)$ is constant on $\conv(A,B)$;
        \item $\Vert u \Vert_{L^{\infty}} \leq \mu_A \cdot (1+\tfrac{D_1}{D_3}) $;
        \item the flow $\psi$ satisfies
            \[
            \psi(t_{A,B},x) = \Psi(x)
            \]
        and 
        \begin{equation*}
            t_{A,B} =   D_3 \mu_A^{-1}.
           \end{equation*}
    \end{enumerate}
\end{lemma}
\begin{proof}
    Consider the vector field defined by
    \[
    u(x_1,x_2,x_3)= \begin{cases}
    \left( \tfrac{D_1}{D_3}\mu_A,0, \mu_A \right)^T & x \in \conv(A,B), \\
    0 & x \notin \conv(A,B).
    \end{cases}
    \]
    A short calculation shows that this vector field has all the claimed properties.
\end{proof}

\begin{lemma}[Parallel source-sink flow: Change of pipe width] \label{lemma:pipewidth}
Suppose that $A=(0,a_1) \times (a_2,a^2) \times \{0\}$ and that $B=(0,b_1) \times (a_2,a^2) \times \{D_3\}$. Suppose that $(A,\nu_A,u_A)$ and $(B,\nu_B,u_B)$ obey the properties of Definition \ref{def:sourcesink}, i.e. $\nu_A=\nu_B= \un_3$ and, thus, $\mu_A a_1=\mu_B b_1$. Suppose further that
 \[
    \Psi \colon A \to B, \quad \Psi(x_1,x_2,0) = \left(\tfrac{b_1}{a_1}x_1,x_2,D_3 \right).
    \]
Then there is $u \colon [0,t_{A,B}] \times \R^3 \to \R^3$ constant-in-time inducing a source-sink flow $\psi$ from $A$ to $B$ with the following properties:
    \begin{enumerate} [label=(\roman*)]
        \item \label{par1}$\spt(u) = \conv(A,B)$;
        \item \label{par2} $\Vert u \Vert_{L^{\infty}} \leq \max\{ \mu_A,\mu_B \} \cdot (1+ D_3^{-1} \max\{a_1,b_1\}) $;
        \item \label{par25}
        We have
        $\Vert u(t, \cdot) \Vert_{L^p}^p \leq (1+ D_3^{-1} \max\{a_1,b_1\})^p
        \max\{ \mu_A^p \cdot a_1, \mu_B^p \cdot b_1 \} (a^2-a_2) D_3
        $;
         \item \label{par3} The flow $\psi$ satisfies
            \[
            \psi(t_{A,B},x) = \Psi(x)
            \]
        and the time $t_{A,B}$ satisfies
        \begin{equation*}
           \min\{ D_3/\mu_A, D_3/\mu_B\}
           \leq 
           t_{A,B}
           \leq 
           \max\{ D_3/\mu_A, D_3/\mu_B\}.
           \end{equation*}
    \end{enumerate}   
\end{lemma}
\begin{proof}
    We may assume $u_A \neq u_B$, i.e. $a_1\neq b_1$, otherwise we are in the case of the previous lemma. W.l.o.g we also assume that $a_1  > b_1$ and therefore $\mu_A < \mu_B$.

    Define $\mu_1 = \mu_A a_1 = \mu_B b_1$. Consider the following vector field initially
    defined on $\R^3$:
    \[
    \tilde{u}(x_1,x_2,x_3) = \mu_1 \left(\begin{array}{c}
    D_3 x_1 \cdot \tfrac{(b_1-a_1)}{[(D_3-x_3)a_1+x_3 b_1]^2}
    \\
    0 \\
    D_3 \cdot \tfrac{1}{[(D_3-x_3)a_1+x_3 b_1]}
     \end{array}\right)
    \]
    and further define
    \begin{equation} \label{def:urescaling}
        u(x) = \begin{cases}
            \tilde{u}(x) & x \in \conv(A,B), \\
            0 & \text{else.}
        \end{cases}
    \end{equation}
    Observe that $\diverg \tilde{u}=0$ and that $u \in BV(\R^3)$. Considering the jumps of $u$ at $\partial \conv(A,B)$ one obtains \ref{def:sourcesink:1} of Definition \ref{def:sourcesink}, i.e. $u$ becomes a proper source-sink vector field. Moreover items \ref{par1} and \ref{par2} from the present lemma are immediate.

    We proceed by showing the $L^p$ estimate. Observe that 
    \begin{equation} \label{est:u1}
    \vert u_1 \vert \leq D_3^{-1} \max 
  \{a_1,b_1\} u_3
    \end{equation}
    and we may therefore continue to estimate the norm of $u_3$. Define 
    \[
    h(x_3)= \tfrac{D_3-x_3}{D_3} a_1 + \tfrac{x_3}{D_3} b_1.
    \]
    Then
    \begin{align*}
        \Vert u_3\Vert_{L^p}^p& = \int_0^{D_3} \int_{a_2}^{a^2} \int_0^{h(x_3)} \vert u_3 \vert^p \dx_1 \dx_2 \dx_3 \\
        &= (a^2-a_2) \int_0^{D_3} \int_0^{h(x_3)} \mu_1^p h(x_3)^{-p} \dx_1 \dx_3 \\
        & \leq (a^2-a_2) D_3 \mu_1^p \max\{a_1^{-p+1}, b_1^{-p+1}\} = (a^2-a_2) D_3 \max\{ \mu_A^p a_1, \mu_B^p b_1\}.
    \end{align*}
    Combining this with \eqref{est:u1} yields the desired $L^p$ estimate.

    It remains to show \ref{par3}. To this end, consider the solution $\phi$ to the ordinary differential equation
    \[
    \begin{cases} \phi'(t) = \mu_1 \tfrac{D_3}{(D_3-\phi(t))a_1 + \phi(t) b_1}, \\
    \phi(0) = 0.
    \end{cases}
    \]
    For $(x_1,x_2) \in A$, the solution to the ODE
    \[
    \begin{cases}\partial_t \psi(t,x_1,x_2) = \tilde u(\psi(t,x_1,x_2)) & t \geq 0,\\
    \psi(0,x_1,x_2)= (x_1~ x_2 ~0)^T, & 
    \end{cases}    
    \]
    is given by
    \[
    \psi(t,x_1,x_2) = \left(\begin{array}{c}
        \tfrac{x_1}{a_1} h(\phi(t))   \\
        x_2 \\
        \phi(t)
    \end{array}
    \right).
    \]
    This can be verified by explicit calculation.
  % Indeed, observe that
  %   \begin{equation*}
  %       \psi_1'(t)=\mu_1 D_3^{-1}x_1(t)\frac{(b_1-a_1)}{h(\phi(t))^2}=-\psi_1(t)\frac{\varphi''(t)}{\dot\varphi'(t)}.
  %   \end{equation*}
    Therefore, $\psi$ is a source-sink-flow from $A$ to $B$ with time $t_{A,B}$ such that $\psi(t_{A,B}) = D_3$. By the estimate
    \[
    \mu_A \leq \partial_t \phi(t) \leq \mu_B
    \]
    we also find the upper and lower bound for $t_{A,B}$ from \ref{par3}. 
\end{proof}

\begin{remark} \label{rem:continuity}
While we do not need the exact value of $t_{A,B}$, observe that, fixing $A$ and $\mu_A$, $t_{A,B}$ depends continuously on the choice of $b_1$.
\end{remark}
\begin{figure}[!htb]
   \begin{minipage}{0.48\textwidth}
     \centering
     \includegraphics[width=.9\linewidth]{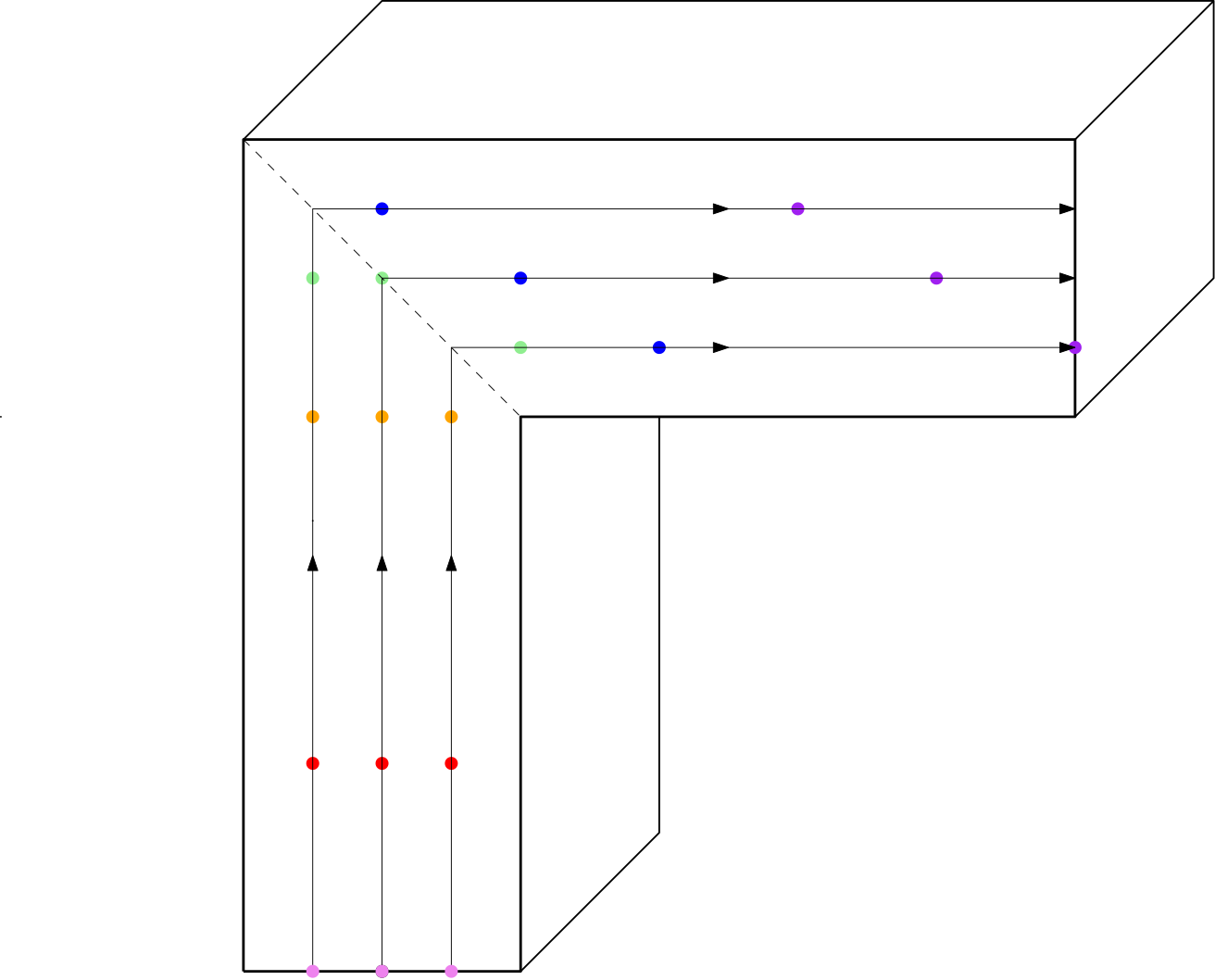}
     \caption{Corner Flow I: If we keep the absolute value of the velocity constant, then particles starting simultaneously at the source will not arrive at the sink at the same time. In both pictures, points of the same color indicate the position of particles at a specific time.}\label{corner1}
   \end{minipage}\hfill
    \begin{minipage}{0.48\textwidth}
      \centering
      \includegraphics[width=.9\linewidth]{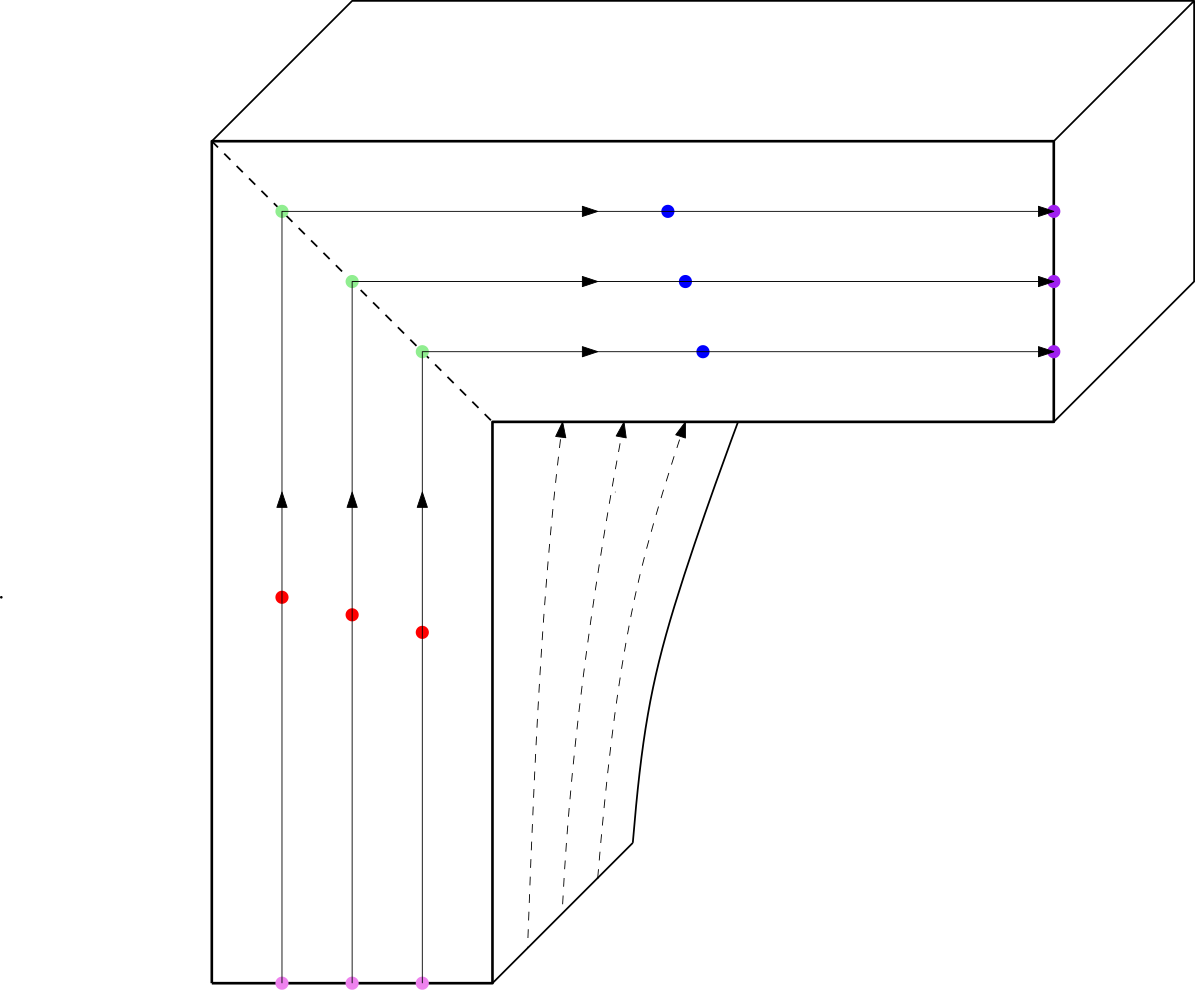}
      \caption{Corner Flow II: Instead of the flow in Figure \ref{corner1}, we decelerate the fluid particle on the 'inner lane' in the relevant direction, so that the fluid particles reach their destination at the same time. The price we have to pay to keep incompressibility is an acceleration into the third coordinate direction (see the dashed trajectories).
      }\label{corner2}
    \end{minipage}
\end{figure}

\begin{lemma} [Banked corner flow] \label{lemma:corner}
   Suppose that $A=(0,a_1) \times (0,a_2) \times \{0\}$ and that $B=\{L\} \times (0,a_2) \times (L-a_1,L)$,  $L> a_1$. Suppose that $(A,\nu_A,u_A)$ and $(B,\nu_B,u_B)$ obey the properties of Definition \ref{def:sourcesink}, i.e. $\nu_A=\un_3$, $\nu_B=\un_1$ and $u_A = \mu \un_3$, $u_B = \mu \un_1$ for some $\mu>0$. Suppose further that
   \[
   \Psi \colon A \to B, \quad \Psi(x_1,x_2,0) = (L,x_2,L-x_1).
   \]
   Then there is $u \colon [0,t_{A,B}] \times \R^3 \to \R^3$ constant-in-time inducing a source sink flow $\psi$ from $A$ to $B$ with the following properties:
   \begin{enumerate} [label=(\roman*)]
       \item \label{corner:1} $\spt(u) \subset (0,a_1) \times (0,\e^{L/(L-a_1)} a_2) \times (0,L) \cup (0,L) \times (0,\e^{L/(L-a_1)} a_2) \times (L-a_1,L)$;
       \item \label{corner:2} $u$ is piecewise smooth and $\Vert u \Vert_{L^{\infty}} \leq \mu (1 +\tfrac{a_2}{L} e^{(2L)/(L-a_1)})$;
        \item \label{corner:3} The flow $\psi$ satisfies
            \[
            \psi(t_{A,B},x) = \Psi(x)
            \]
        and 
        \begin{equation*}
            t_{A,B} = 2 \frac{L}{\mu}.
           \end{equation*}
   \end{enumerate}
\end{lemma}

\begin{proof}
    We construct the flow in two steps. First of all, consider the bijective map
    \[
    \Psi^1 \colon A \to C, \quad \Psi(x_1,x_2,0)=(x_1,G(x_1)x_2,L-x_1),
    \]
    for some $G(x_1)$ defined below, where $C= \{x_1+x_3 = L,~0<x_1<a_1,~0<x_2<a_2G(x_1)\}$.
    Set $\Psi^2 \colon C \to B$ as $\Psi^2 = \Psi \circ (\Psi^1)^{-1}$. 
    We first construct a flow from $A$ to $C$ and, after that, we may define the flow from $C$ to $B$ via a reflection along $\{x_1+x_3=L\}$.

    To this end, first define $\tilde{u}$ as follows
    \[
    \tilde{u}(x_1,x_2,x_3)=  \left(\begin{array}{c}
         0 \\
         F(x_1) x_2 \\
         \mu - F(x_1) x_3
    \end{array} \right) 
    \]
    where $F(\cdot)$ is defined via setting $F(0) =0$ and
    \[
        F(z) = \tfrac{\mu}{L} \cdot \left( W\left[\exp\left(\tfrac{L}{z-L}\right) \tfrac{L}{z-L}\right] + \tfrac{L}{L-z} \right).
    \]
    Here, for an argument $y \in (-\e^{-1},0)$, $W(y)$ denotes the Lambert-$W$-function, i.e. the solution to
    \[
    w \e^w =y
    \]
    that satisfies $0>w >-1$.

    Thus, $F$ is smooth and positive for $z \in (0,a_1)$ and $\diverg \tilde{u}=0$ pointwise almost everywhere. Further observe that $F(z) \leq \mu/(L-z)$. On the other hand, observe that if $x_1+x_3 \leq L$  then $F(x_1) x_3 < (L-x_1) \tfrac{\mu}{L} \tfrac{L}{L-x_1} \leq \mu$, i.e.
    \[
    \tilde{u}_3 (0) > 0.
    \]
    Now consider, for $z = (x_1,x_2,0)^T$, the ordinary differential equation
    \begin{equation} \label{eq:ODE2}
    \begin{cases}\partial_t \Psi(t,z) = \tilde u(\psi(t,z)) & t \geq 0,\\
    \psi(0,z)= z. & 
    \end{cases}    
    \end{equation}
    The solution to this equation is as follows
    \[
    \psi(t,x_1,x_2) = \left(\begin{array}{c}
         x_1 \\
         x_2 \exp{(F(x_1)t)}\\
         \mu (F(x_1))^{-1} (1- \exp{(-F(x_1)t)})
    \end{array} \right).
    \]
       Plugging in $t^{\ast}= L \mu^{-1}$ gives 
    \[   
    \psi(t^{\ast},x_1,x_2) = \left(\begin{array}{c}
         x_1 \\
         x_2 \exp{(F(x_1)t^{\ast})}\\
         L-x_1
    \end{array} \right),
    \]
    and setting $G(x_1) = \exp{(F(x_1)t^{\ast})}$ gives a flow from $A$ to $C$.
    Setting 
    \[
    \bar{u}(x_1,x_2,x_3) = \begin{cases}
        \tilde{u}(x_1,x_2,x_3) & x \in \psi([0,t^{\ast}] \times (0,a_1) \times (0,a_2)), \\
        0 & \text{else,}
    \end{cases}
    \]
   gives that  \begin{enumerate} [label=($\bar{u}$\arabic*)]
       \item \label{baru1} $\diverg(\bar{u})= -\mu \un_3 \mathcal{H}^{n-1} \llcorner A + (\un_1+\un_3) h(x) \mathcal{H}^{n-1} \llcorner C$, where $h(x) = \tilde{u}(x) \cdot \tfrac{1}{2} (\un_1+\un_3)$;
       \item following the trajectories we find 
       \[
        \spt(\bar{u}) \subset (0,a_1) \times (0,\e^{L/(L-a_1)} a_2) \times (0,L), 
       \]
       furthermore all trajectories are disjoint and do not intersect with itself as $u_3>0$;
       \item $\bar{u}$ is piecewise smooth and 
       \[\Vert \bar{u} \Vert_{L^{\infty}} \leq \mu + \sup_{(t,x) \in \spt(\bar{u})}(x_2 F(x_1)) \leq \mu (1+ \e^{L/(L-a)}/(L-a));
       \]
       \item \label{baru4}  The solution $\psi$ to the differential equation \eqref{eq:ODE2} obeys $\psi(L\mu^{-1}) = \Psi^1 \colon A \to C$.
   \end{enumerate}
    Now define the reflected
    \[
    \bar{u}^2(x_1,x_2,x_3) = \left( \begin{array}{r}
        \bar{u}_3(L-x_3,x_2,L-x_1) \\
        -\bar{u}_2(L-x_3,x_2,L-x_1)
        \\
        \bar{u}_1(L-x_3,x_2,L-x_1)
    \end{array} \right)
    \]
    and define the vector field
    \begin{equation}
    u = \bar{u} + \bar{u}^2.
    \end{equation}
    One may show similar properties to \ref{baru1}-\ref{baru4} for $\bar{u}^2$ (it defines a vector field inducing a flow from $C$ to $B$) and, combining those properties in the sum of the vector fields $u$ we find that
    \begin{enumerate} [label=($u$\arabic*)]
        \item $\diverg(u)= -\mu \un_3 \mathcal{H}^{n-1} \llcorner A +  \mu \un_1 \mathcal{H}^{n-1} \llcorner B$;
        \item we find \ref{corner:1} if we combine the supports of $\bar{u}$ and $\bar{u}^2$. Moreover, the trajectories are disjoint and do not intersect with itself.
        \item $u$ is piecewise smooth and 
        \[
        \Vert u \Vert_{L^{\infty}} \leq  \mu (1+ \e^{L/(L-a)}/(L-a));
        \]
        \item $u$ induces a source-sink flow from $A$ to $B$ with $\psi(2L/\mu) = \Psi$.
    \end{enumerate}
  \end{proof} 

\subsection{Overview over the different parts of the construction}

We give a short overview of what comes in the next subsections (cf. Figure \ref{fig:network:1} and also Figure \ref{figure:levels}) before constructing the flow.
\begin{itemize}
    \item In this section, we construct the flow \emph{inside} a small cube $\Qcal_v$ (the \emph{snake flow}) and in Section \ref{sec:gate} we directly connect this flow to predefined \emph{gates} (the \emph{gate flow});
    \item in Section \ref{sec:reservoir} we construct a flow that keeps track of the time constraint from the discretised problem (the \emph{reservoir flow});
    \item in Sections \ref{sec:highway} and \ref{sec:interchange} we construct the flows that correspond to the discretised problem in the network (the \emph{highway flow} and the \emph{interchange flow}, respectively);
    \item finally, we connect the cube to the network in Sections \ref{sec:entrexit} (the (highway) \emph{entrance} and \emph{exit} flows).
\end{itemize}

 \begin{figure}
      \includegraphics[width=.75\linewidth]{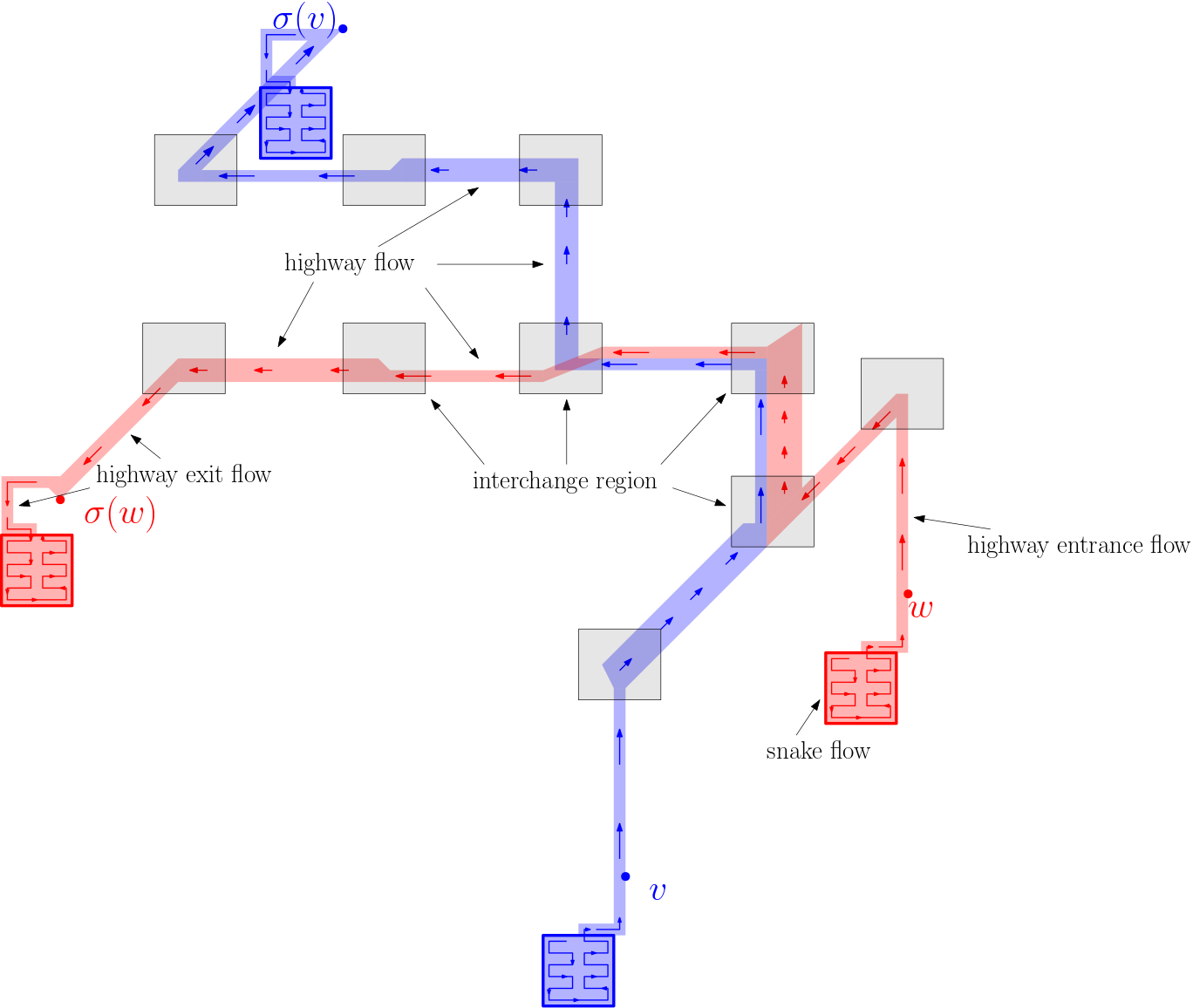}
      \caption{Overview over the different construction steps I: One may see the flow that arises through the solution of the discrete problem. We refer to Figure \ref{figure:levels} for a clearer picture of the construction of the flow locally around $\Qcal_v$, i.e. the snake flow, the gate flow, the reservoir flow and the entrance/exit flow.}
      \label{fig:network:1}
\end{figure}

\FloatBarrier
\subsection{Snake Flow} \label{sec:snake}
\FloatBarrier
 \begin{figure}[!htbp]
     \centering \includegraphics[width=0.7 \textwidth]{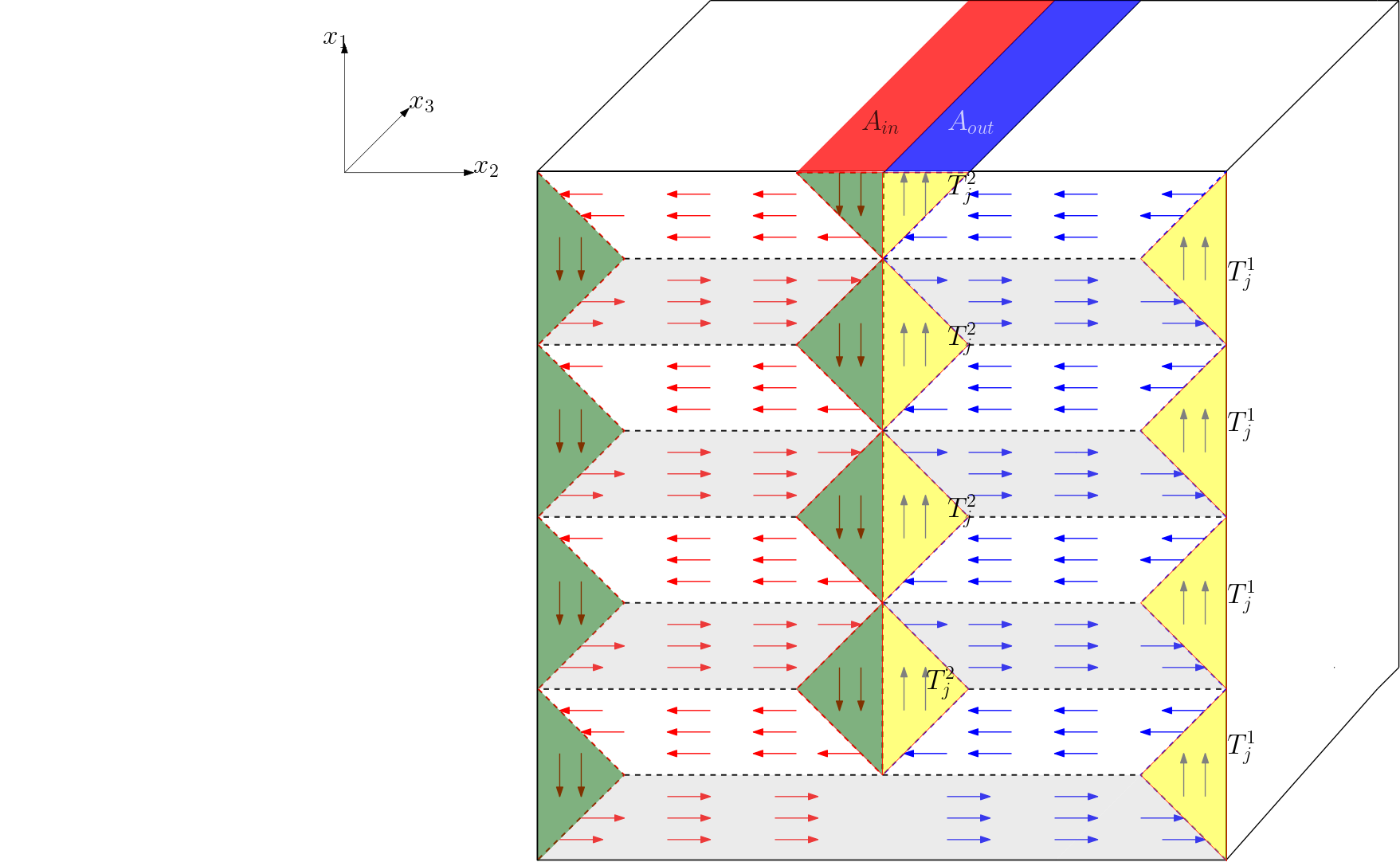}
     \caption{ Visualisation of the snake flow inside a cube $\Qcal_v$. The absolute value of the velocity is constant, only the direction changes. The yellow region is $B_1$, the green $B_2$, the gray $B_3$ and the white one is $B_4$.} \label{figure:snake}
    \end{figure}
\FloatBarrier
We first construct the vector field $u$ inside a small cube that is moved by the permutation $\sigma$. If $\sigma(v)=v$, then we keep the small cube at rest. Upon translation, we may assume that we are given the small cube
\[
\Qcal_v = [0, \ell]^3
\]
and, for simplicity write $\kappa^{-1} \ell = \lambda$.

\begin{lemma}[Snake flow] \label{lemma:snake}
Let $\Qcal= [0,\ell]^3$ and let
$$A_{in} = \{ \ell \} \times ((\ell-\lambda)/2, \ell/2) \times (0,\ell) \quad A_{out} =  \{ \ell \} \times (\ell/2,(\ell+\lambda)/2) \times (0,\ell). $$
Let ${\mu_0} \in \R_+$ be given. Then there exists a constant-in-time vector field $u =u_{\mathrm{Sn}} \in BV(\R^3)$ with the following properties:
\begin{enumerate} [label=(Sn\arabic*)]
    \item \label{sn:1} $\spt(u) = \Qcal$;
    \item \label{sn:2} $\vert u \vert = {\mu_0}$ almost everywhere in $\Qcal$;
    \item \label{sn:3} $\diverg ( u) = -{\mu_0} \un_3 \mathcal{H}^{n-1} \llcorner A_{in} + {\mu_0} \un_3 \mathcal{H}^{n-1} \llcorner A_{out}$ in the sense of distributions;
    \item \label{sn:4} The flow $\psi$ induced by $u$
        \[
        \begin{cases}
            \partial_t \psi(t,z) = u(\psi(t,z)) & t \geq 0{,} \\
            \psi(0,z) = z & z \in A_{in}{,}
        \end{cases}
        \]
    satisfies
    \begin{enumerate} [label=(\roman*)]
        \item For $z=(z_1,z_2,z_3) \in A_{in}$
        \[
        t_z := \tfrac{1}{{\mu_0}} \cdot ( 2 \ell \kappa - \lambda + 4(\ell/2-z_2)).
        \]
   
    Then $\Psi_z \colon z \mapsto \psi(t_z,z)$ is a homeomorphism from $A_{in}$ to $A_{out}$ and, furthermore, 
    \[
        \Psi_z(z_1,z_2,z_3) = (z_1, \ell -z_2, z_3).
    \]
    \item The union of trajectories
    \[
    T_\psi = \{ \psi(t,z) \colon z \in A_{in}, 0 \leq t \leq t_z \}
    \]
    obeys $\mathcal{L}^{n}(\Qcal \setminus T_\psi) =0$.
     \end{enumerate}    
    \end{enumerate}
\end{lemma}
\begin{remark}
    We choose later ${\mu_0}$ in such a way that $t_z$ is comparable to $1$, i.e. ${\mu_0} \sim \ell \kappa \sim N^{-1}$, i.e. the flow through the cube is quite slow.
    Moreover, observe that $u$ has $L^{\infty}$-norm ${\mu_0}$.
    The flow is chosen in the present \emph{snake} pattern to achieve the following goals: First of all, all of the fluid present in the cube at the beginning shall leave the cube at some time. Second, we have 'pipes' also through the cube $\Q_v$ with thickness $\lambda/2$, which is a constant we need to carefully choose during the construction. Last, the flow needs to be setup with $A_{in}$ and $A_{out}$ as in the present configuration to allow for a nice gate and reversed gate flow (cf. Sections \ref{sec:gate} and \ref{sec:phase2}).
\end{remark}

\begin{proof}
We define $u$ as follows:
\begin{equation} \label{def:usnake}
u(x) = \begin{cases}{\mu_0} \un_1 & x \in B_1, \\
-{\mu_0} \un_1 & x \in B_2, \\
{\mu_0} \un_2 & x \in B_3, \\
-{\mu_0} \un_2 & x \in B_4 ,
\end{cases}
\end{equation}
where the sets $B_1$,...,$B_4$ are given by 
\begin{align*}
   & B_1 = \{x_2 > \ell/2\} \cap \bigcup_{j=0}^{\kappa-1} T_j^1 \cap \bigcup_{j=1}^{\kappa-1} T_j^2; \\
   &T_j^1:= \{x_2 < \ell\} \cap  \{x_1+x_2 > \ell + j \lambda\} \cap \{x_2-x_1 > \ell-(j+1)\lambda\} \cap \{0 < x_3 < \ell\}; \\
   &T_j^2 :=  \{x_2 > \ell/2\} \cap  \{x_1+x_2 < \ell/2 + \lambda/2 + j \lambda\} \cap \{x_2-x_1 < \ell/2 - \lambda/2 -(j-1)\lambda\} \\
   & \hspace{1cm} \cap \{0 < x_3 < \ell \}; \\
   &B_2 := \{(x_1,x_2,x_3) \in Q \colon (x_1,\ell-x_2,x_3) \in B_1 \}; \\
   &B_3 := \bigcup_{j=0}^{\kappa-1} \{j \lambda < x_1 < (j+1/2)\lambda\} \setminus (B_1 \cup B_2); \\
   &B_4:= Q \setminus \{ B_1 \cup B_2 \cup B_3\}.
\end{align*}
Properties \ref{sn:1} \& \ref{sn:2} follow immediately by definition. For \ref{sn:3} one may check that all of the jumps (except for those at $A_{in}$ and $A_{out}$) are divergence-free. For \ref{sn:4} one may explicitly calculate the trajectories and their lengths (cf. Figure \ref{figure:snake}).
\end{proof}
\subsection{Gate Flow } \label{sec:gate}
For a visualisation of the following flow we refer to Figure \ref{figure:gate}. Define the following rectangles, that we also shall call \emph{gates}
\begin{align}
    G_{in} = (\ell,\ell+\lambda/2) \times \{ \ell/2-\lambda/2 \} \times (0,\ell),\quad G_{out} = (\ell,\ell+\lambda/2) \times \{ \ell/2+\lambda/2 \} \times (0,\ell). 
\end{align}
We define a velocity field $u_G$ in $\conv(G_{in},G_{out})$ by defining it piecewisely in different regions.
\begin{align*}
    u_G(x)= \begin{cases}
        {\mu_0} \un_1 & x \in B_5= \{x_1 > \ell, x_2 > \ell/2, x_1+x_2 < 3\ell/2+\lambda/2 \} \cap \{0 < x_3 < \ell \}, \\
        -{\mu_0} \un_1 & x \in B_6 = \{x_1 > \ell, x_2 < \ell/2, x_1-x_2 < \ell/2 +\lambda/2 \} \cap \{0 < x_3 < \ell \},  \\
        {\mu_0} \un_2 & x \in \conv(G_{in},G_{out}) \setminus (B_5 \cup B_6)
    \end{cases}
    \end{align*}
Together with the Snake flow we can show the following result:
\begin{coro} \label{lucatoni}
Let $G_{in}$ and $G_{out}$ be as above and define $u$ to be $u_{\mathrm{Sn}}$ (cf. Lemma \ref{lemma:snake}) inside $Q$ and the gate flow $u_G$ inside $\conv(G_{in},G_{out})$. Then $u$ is a source-sink-flow from $G_{in}$ to $G_{out}$ with velocity ${\mu_0}$ and
\[
t_{G_{in},G_{out}} = \tfrac{1}{{\mu_0}} \cdot (2 \ell \kappa + \lambda).
\]
\end{coro}
We now choose ${\mu_0}$ such that $t_{G_{in},G_{out}} = \tfrac{1}{4}$, i.e.
\begin{equation} \label{choice:u0}
    {\mu_0} = (8 \ell \kappa + 4 \lambda) \in (8 \ell \kappa, 9 \ell \kappa).
\end{equation}
\FloatBarrier
\begin{figure}[!htbp]
     \centering \includegraphics[width=0.7 \textwidth]{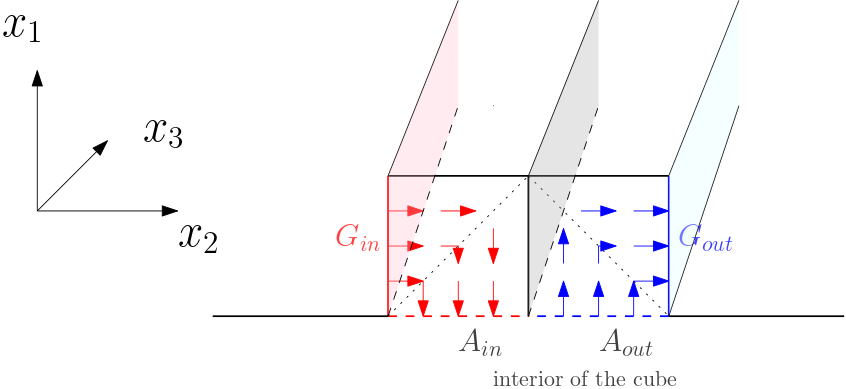}
     \caption{Flow in and out of the cube through the 'gates'.}
     \label{figure:gate}
    \end{figure}

We shortly comment on \emph{why} we need to do the construction of the gate flow: In the first part of the construction we move the fluid from cube $\Qcal_v$ to $\Qcal_{\sigma(v)}$. In the second part (cf. Section \ref{sec:phase2}) we take care of the fluid that was initially in our pipe: For this we reverse the flow and 'shut' the gates, so that the flow directly goes from $G_{out}$ to $G_{in}$ and \emph{not} through the cube $\mathcal{Q}_v$. With the present construction (the gates are quite adjacent, cf. Figure \ref{figure:gate}), the reversal of the flow can be achieved quite easily.

\FloatBarrier
\subsection{Reservoir Flow} \label{sec:reservoir}

Before connecting to the highway flow, we need to ensure, that the fluid starting in a cube $\Qcal_v$ ends up into $\Qcal_{\sigma(v)}$ at exactly the right time. The construction on the 'highway' (i.e. the counterpart to the discrete network) gives a lower bound on this time, but it might be different for each cube $v$. To compensate for this fact, we build in a reservoir, that holds the fluid for a certain amount of time, cf. Figure \ref{figure:reservoir}.
\begin{lemma} \label{lemma:reservoir}
For any time $\tau \in [\tfrac{\ell/2-\lambda/2}{{\mu_0}},1/4]$ there is a source-sink-flow from $G_{out}$ to $R_{out} := (\ell,\ell+\lambda/2) \times \{\ell\} \times (0,\ell)$ induced by a constant-in-time vector field $u =u_{\mathrm{Re}} \in BV(\R^3)$ with the following properties:
\begin{enumerate} [label=(Re\arabic*)]
    \item \label{Re:1} $u$ induces a source sink flow from $G_{out}$ to $R_{out}$ with $\nu = \un_2$ and $u_{G_{out}} = u_{R_{out}} = {\mu_0} \un_2$;
    \item \label{Re:2} $\diverg ( u) = -{\mu_0} \un_2 \mathcal{H}^{n-1} \llcorner G_{out} + {\mu_0} \un_2 \mathcal{H}^{n-1} \llcorner R_{out}$ in the sense of distributions;
    \item \label{Re:3} $\spt(u) \subset (\ell,8 \ell) \times (\ell/2+\lambda/2,\lambda) \times (0,\ell)$;
    \item \label{Re:35} $\Vert u \Vert_{L^{\infty}} \leq 42 {\mu_0} \leq 400 \kappa \ell $;
    \item \label{Re:4} The flow $\psi$ induced by $u$
        \[
        \begin{cases}
            \partial_t \psi(t,z) = u(\psi(t,z)) & t \geq 0, \\
            \psi(0,z) = z, & 
        \end{cases}
        \]
    for $z \in G_{out}$ satisfies the following. $\Psi_z \colon z \mapsto \psi(\tau,z)$ is a homeomorphism from $G_{out}$ to $R_{out}$ and, furthermore, 
    \[
        \Psi_z(z_1,\ell+\lambda/2,z_3) = (z_1, \ell, z_3).
    \]

     \end{enumerate}
\end{lemma}
\begin{proof}
    Consider \[
    R_{1}= (\ell, \ell+h) \times \{2/3 \ell + \lambda/4\} \times (0,\ell)\text{ and }R_2= (\ell, \ell+h) \times \{5/6 \ell + \lambda/4\} \times (0,\ell)
    \] for some $h \in (\lambda/2,7 \ell)$ and take the concatenation of source-sink flow from Lemma \ref{lemma:pipewidth} from $G_{out}$ to $R_1$, from $R_1$ to $R_2$  and from $R_{2}$ to $R_{out}$. Using Lemma \ref{lemma:pipewidth}, items \ref{Re:1}--\ref{Re:3} are immediate. Moreover, using \ref{par2} of Lemma \ref{lemma:pipewidth} for above choice of $h$ yields \ref{Re:35}.
    Finally, the time $t=t(h)$ that the flow needs to pass from $G_{out}$ to $T_{out}$ depends continuously on the height $h$ (cf. Remark \ref{rem:continuity}) and satisfies $t(0) = \tfrac{\ell/2-\lambda/2}{{\mu_0}}$
   and that for $h = 6 \ell$ we have  $t(6 \ell) \geq 1/4$. By intermediate value theorem we can therefore choose an $h$ in this interval such that $t(h) = \tau$.    
\end{proof}

\begin{figure}[!htbp]
     \centering \includegraphics[width=0.7 \textwidth]{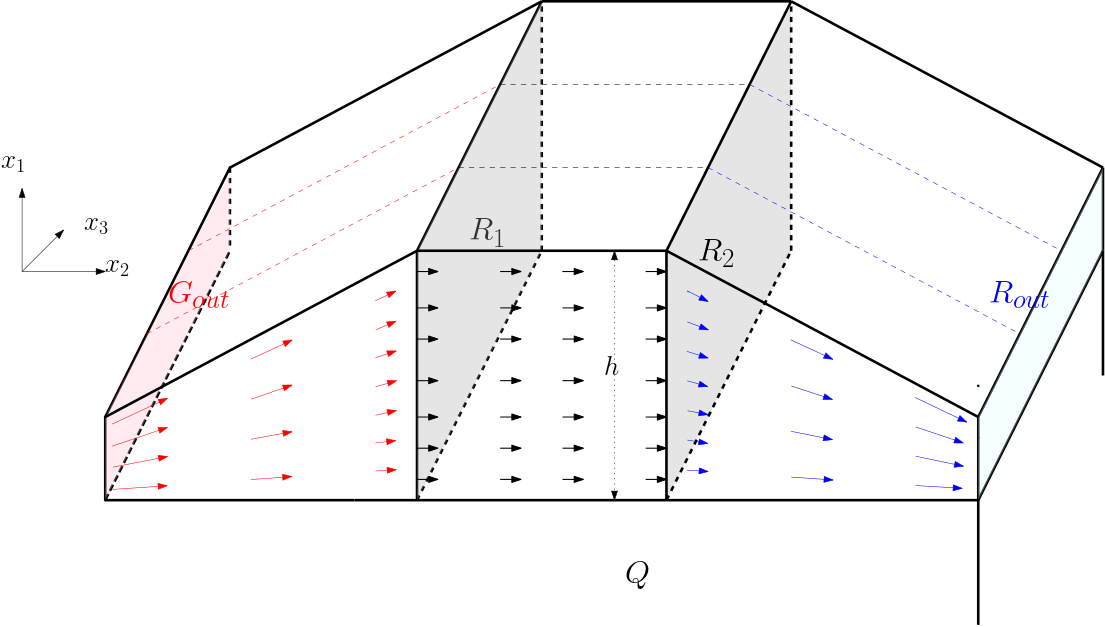}
     \caption{Flow in and out of a water reservoir. Note that the absolute value of the velocity decreases, when the fluid gets to the thicker part of the reservoir. The flow is a concatenation of three flows that change the pipe width.}
     \label{figure:reservoir}
    \end{figure}

\subsection{Highway flow} \label{sec:highway}
In this subsection we will give explicit constructions of the vector field connecting two adjacent vertices on the network. For this, we first recall some notation for the discrete problem.

As before, let $G=(V,E)$ be the $3D$ grid. As we are only dealing with two dimensional permutations we denote by $G^2=(V^2,E^2)$ the graph with $V^2=V$ and $E^2$ to be the set of edges in the horizontal direction, i.e. 
$E^2 = E \setminus \{ e= \{v,w\} \colon \vert v_3- w_3 \vert=1 \}$.
By solving the discretised problem in two dimensions we obtain paths $\gamma$ in $G^2$ from $v$ to $\sigma(v)$. For an edge $e \in E$ denote by $\Gamma(e)$ the set of paths such that $e \in E(\gamma)$.

We give an order to each path $\gamma$ as follows. Consider the \emph{lexicographic order} on $V^2$ which is as follows.
\[
v \prec w \text{ if } v_1 < w_1\text{ or } v_1 = w_1,\,v_2 <w_2 \text{ or } v_1=w_1\,v_2=w_2,\,v_3<w_3.
\] 
For each path $\gamma$ we can associate $v \in V^2$ such that $\gamma = \gamma_v$ is the path from $v$ to $\sigma(v)$. We write
\[
\gamma_v \prec \gamma_w \text{ if and only if } v \prec w.
\]
We give a direction to the edge as follows: For an edge $e = \{v,w\}$, $v \prec w$ we say that $\gamma \in \Gamma^+(e)$ if the path $\gamma$ goes through $v$ before $w$ and $\gamma \in \Gamma^-(e)$ otherwise.
We say that paths $\gamma_1,...,\gamma_k$ are in \emph{lexicographical order} if $\gamma_1 \prec \gamma_2 \prec \ldots \prec \gamma_k$ and in \emph{reverse lexicographic order} if $\gamma_k \prec \gamma_{k-1} \prec \ldots \prec \gamma_1$.

For each edge $e$ we define an order $\prec_e$ with respect to $e$ as follows.
 If $e=\{v,v+\un_2\}$ we define
 \[
 \gamma \prec_e \gamma' \quad \text{if and only if } \begin{cases}
     \gamma \in \Gamma^+(e), \gamma' \in \Gamma^-(e){,} \\
     \gamma, \gamma' \in \Gamma^+(e) \text{ and } \gamma \prec \gamma' {,} \\
     \gamma, \gamma' \in \Gamma^-(e) \text{ and } \gamma' \prec \gamma {.}
 \end{cases}
 \]
such that paths are first ordered by their direction, 'forward' paths are sorted in lexicographical order and backward paths are sorted in reverse lexicographic order. If $e= \{v,v+\un_1\}$ we likewise define
 \[
 \gamma \prec_e \gamma' \quad \text{if and only if } \begin{cases}
     \gamma \in \Gamma^-(e), \gamma' \in \Gamma^+(e){,} \\
     \gamma, \gamma' \in \Gamma^+(e) \text{ and } \gamma' \prec \gamma{,} \\
     \gamma, \gamma' \in \Gamma^-(e) \text{ and } \gamma \prec \gamma{.}
 \end{cases}
 \]
such that paths are first ordered by their direction, backward paths are sorted in lexicographic order and forward paths are sorted in reverse lexicographic order.

Finally, recall that the path $\gamma$ has a width $\rho(e,\gamma)$ along $e$.

We now fix a point $P\in[17c_{th}\ell,N^{-1}-17 c_{th}\ell]^3$. For a vertex $v \in \{0,...,N-1\}^3$ we write 
\[
    P_v = P + N^{-1} v.
\]

From the point $P_v$ we construct the following parallelepiped 
\begin{equation} \label{def:Rv}
    R_v=P_v+[0,\ell]\times[0,\ell]\times[0,14 c_{th}\ell],
\end{equation}

where the constant $c_{th}=2(1+e^2)$ is the thickness of the level; the reason for this choice is explained through Corollary \ref{coro:source:sink:interchange:corner}. Observe that if $\ell N < \tfrac{1}{10^3}$ we have $R_v \subset P + N^{-1} [1/4,3/4]$.
The choice of the constant $14$ in \eqref{def:Rv} becomes apparent in the following Definition \ref{def:level} and the subsequent discussion.

\medskip

Consider the paths $\gamma\in V(v_P)$ given by the discrete problem of Section \ref{sec:3}. 
\begin{definition}[Level]\label{def:level}
    A \emph{level} is the value of the  map $\text{Lev} \colon \{(\gamma,v) \colon v \in V(\gamma) \} \rightarrow \lbrace1,2,\dots,14\rbrace$ defined as follows
    \begin{equation*}
        \text{Lev}(\gamma,v)= \begin{cases}
        14 &\text{ if $\gamma$ ends in $v$,}\\
        13 & \text{ if $\gamma$ starts in $v$,}\\
            12&\text{ if $\gamma$ goes through $v-\un_1$,$v$,$v+\un_1$ (in this order),}\\
            \dots, \\
             1 &\text{ if $\gamma$ goes through $v-\un_2$,$v$,$v+\un_2$ (in this order).}\\
             \end{cases}
    \end{equation*}
    the definition for the whole $12$ cases is straightforward.
\end{definition}
\begin{remark}
    The first $12$ levels are labeled through the vertex that is visited by $\gamma$ before $v$ ($4$ possibilities) and the vertex that comes \emph{after} $v$ (subsequently 3 possibilities), so that we have $12= 4 \cdot 3$ configurations in total.
\end{remark}

We consider now a vertex $w$ that is adjacent to $w$ in the graph. Without loss of generality we may assume $w= v + \un_2$ and therefore $P_w = P_v + N^{-1} \un_2$. We aim at defining the system of source-sink flows starting in the parallelepiped $R_v$ and ending in the parallelepiped $R_w$. We call $e$ the edge $\lbrace v,w\rbrace$ and we fix $\gamma\in \Gamma(e)$.  
 We define
\begin{equation*}
    \tilde\rho(\gamma,e)=\rho(\gamma,e)\frac{N^{-1}}{4K^2},
\end{equation*}
where $\rho(\gamma,e)$ is given by the discrete formulation of Section \ref{sec:3}. Here, $\tilde\rho(\gamma,e)$ is an appropriate rescaling of $\rho$ such that
\[
\sum_{\gamma' \in \Gamma(e)}  \tilde{\rho}(\gamma',e) \leq \ell/4 K^{-1} \leq \ell/4.
\]
Further define 
\begin{equation} \label{def:height}
    h(\gamma,e) = 3/8 \ell + \sum_{\gamma' \in \Gamma(e) \colon \gamma' \prec_e \gamma } \tilde\rho(\gamma',e)
\end{equation}
and, for $v \in V$ and $\gamma$ with $v \in V(\gamma)$, the interval
\begin{equation} \label{def:Iv3}
  I^{\gamma,v}_{3} :=  (c_{th}\ell\left(\text{Lev}(\gamma,v)-\tfrac{1}{2}\right), c_{th}\ell\left(\text{Lev}(\gamma,v)-\tfrac{1}{2}\right)+\ell).
\end{equation}  
 \begin{definition}\label{Def:source:sink:highway}
     Define the following sets:
     \begin{align*}
         A^{e,\gamma}_{\text{in}}&=P_v +(h(\gamma,e),h(\gamma,e)+\tilde\rho(\gamma,e))\times \lbrace \ell\rbrace\times I^{\gamma,v}_{3}, \\
         A^{e,\gamma}_{\text{out}}&=P_w +(h(\gamma,e),h(\gamma,e)+\tilde\rho(\gamma,e))\times \lbrace 0\rbrace\times I^{\gamma,w}_{3}.
     \end{align*}
     They define a source-sink system in the interchange flow from the vertex $v$ to the vertex $w$. Then define
     \begin{equation*}
         \nu_{ A^{e,\gamma}_{\text{in}}}=\nu_{ A^{e,\gamma}_{\text{out}}}=\un_2,
     \end{equation*}
     and
     \begin{equation*}
         \mu_{ A^{e,\gamma}_{\text{in}}}=\mu_{ A^{e,\gamma}_{\text{out}}}=\frac{{\mu_0}\lambda}{2\tilde\rho(\gamma,e)}=\frac{2{\mu_0}K}{\kappa\rho(\gamma,e)}.
     \end{equation*}
     We now define $u(\gamma,e)$ the vector field given by Lemma \ref{lem:parallel} with $A=A^{e,\gamma}_{\text{in}}$, and $B=A^{e,\gamma}_{\text{out}}$.
 \end{definition}
 \begin{coro}\label{cor:source:sink:highway}
     Let $u(\gamma,e)$  be the vector field given in Definition \ref{Def:source:sink:highway}. Then it satisfies the following properties
     \begin{enumerate}[label=(\roman*)]
         \item $\spt(u(\gamma,e))\subset\conv (R_v,R_w)$;
         \item $\gamma\not=\gamma'\implies \spt(u(\gamma,e))\cap\spt(u(\gamma',e))=\emptyset$;
         \item $\mathscr{L}^3(\spt(u(\gamma,e)))=(N^{-1}-\ell)\ell\tilde\rho(\gamma,e)=(N^{-1}-\ell)\ell^2\frac{\rho(\gamma,e)}{4}\leq N^{-3}\rho(\gamma,e)$;
         \item $\|u(\gamma,e)\|_{L^\infty}\leq c(\kappa,K){\mu_0}\rho(\gamma,e)^{-1}$;
          \item $\|u(\gamma,e)\|_{L^p}^p\leq N^{-3} c(\kappa,K)^p ({\mu_0})^p \rho(\gamma,e)^{-p+1}$;
         \item the time $t_e=t_{A^{e,\gamma}_{\text{in}},A^{e,\gamma}_{\text{out}}}$ along the edge is
         \begin{equation*}
             t_e=\frac{(N^{-1}-\ell)\kappa \rho(\gamma,e)}{2K{\mu_0}}.
         \end{equation*}
     \end{enumerate}
 \end{coro}
\begin{proof}
    First and second property are trivially obtained by the definition. The third one is a consequence of Cavalieri's formula. The last point follows by Lemma \ref{lem:parallel}.
\end{proof}

\begin{figure}[!htbp]
     \centering \includegraphics[width=0.8 \textwidth]{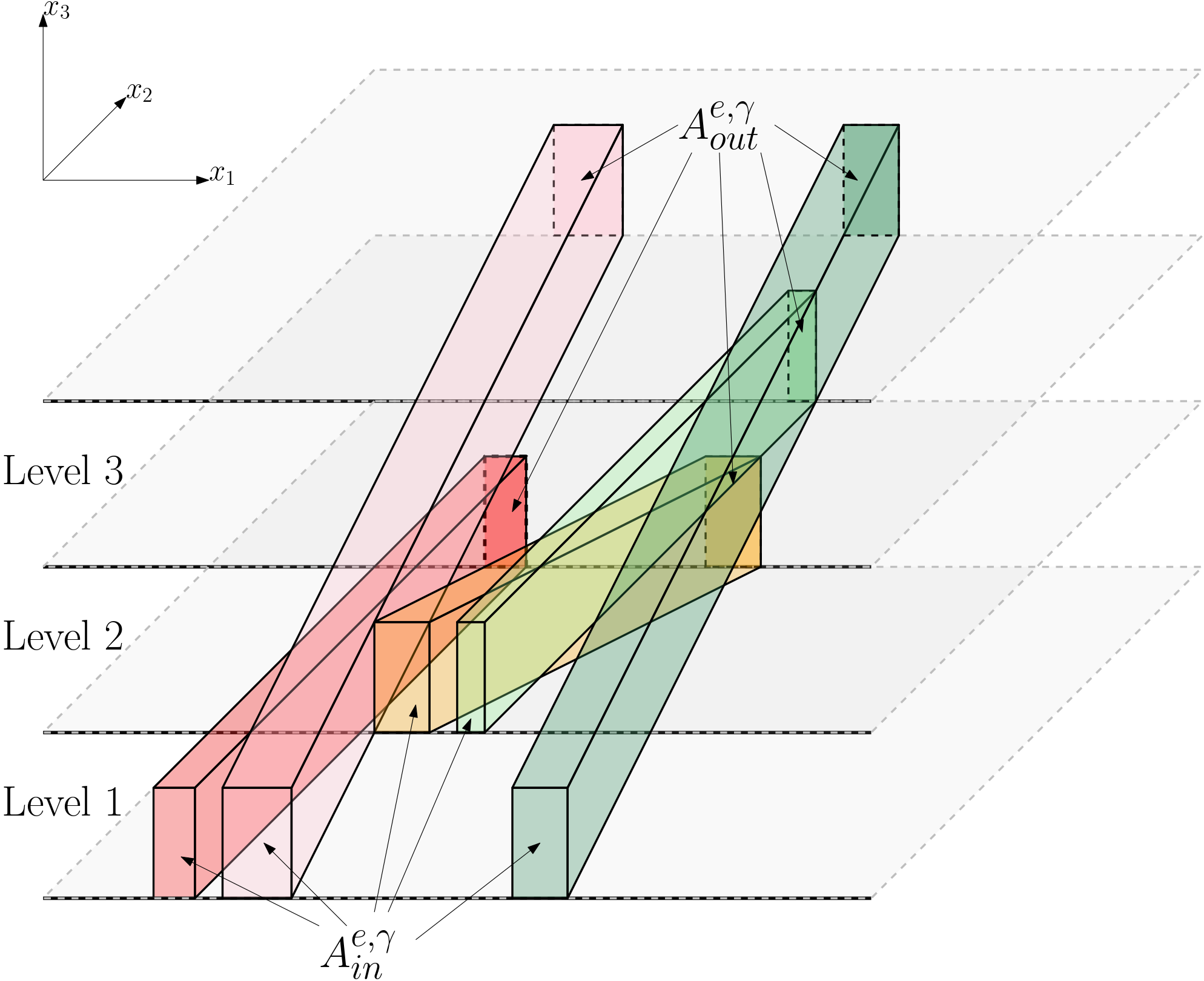}
     \caption{Highway flow along an edge $e=\{v, v+ \un_2\}$ corresponding to a discrete network. The flows are all parallel. The velocity of the vector field, however, \emph{does} depend on the thickness of the pipe: The thinner the pipe is, the faster the velocity of the vector field, so that the total amount of fluid that passes through $A^{e,\gamma}_{in}$ in a fixed time span is constant. $A^{e,\gamma}_{in}$ and $A^{e,\gamma}_{out}$ might be in different levels, which will play a role in the interchange construction.}
    \end{figure}

\subsection{Highway Interchanges} \label{sec:interchange}
We are ready to define the interchange inside a level of the rectangle $R_{v}$. There are two possible situations occurring: Either the path $\gamma$ goes straight ahead or makes a turn. We are dealing with those two cases separately.

\medskip In the first one fix a vertex $v$ and the two following edges $e=\lbrace v-\un_2,v\rbrace$, $e'=\lbrace v,v+\un_2\rbrace$ and assume $\gamma$ visits $e$ and $e'$ in this order. Consider $$A^{e,\gamma}_{\text{out}}= P_v + (h(\gamma,e),h(\gamma,e)+\tilde\rho(\gamma,e))\times \lbrace 0 \rbrace\times I^{\gamma,v}_{3}$$ and
similarily $A^{e',\gamma}_{\text{in}}$, cf. Figure \ref{fig:interchange} for a schematic display.
\begin{definition}\label{def:source:sink:interchange}
Let $v$, $e$, $e'$ and $\gamma$ be as previously specified. Define the following 2D surfaces/checkpoints as follows.
\begin{enumerate}[label=(\roman*)]
\item $A_0=A^{e,\gamma}_{\text{out}}$;
    \item $A_1=P_v + \left(h(\gamma,e)+h(\gamma,e')-\tfrac{\ell}{2},(P_v)_1+h(\gamma,e)+h(\gamma,e')-\tfrac{\ell}{2}+\tilde\rho(\gamma,e)\right)\times \lbrace\frac{\ell}{3}\rbrace\times I^{\gamma,v}_{3}$;
    \item $A_2=(P_v) +\left(h(\gamma,e)+h(\gamma,e')-\tfrac{\ell}{2},(P_v)_1+h(\gamma,e)+h(\gamma,e')-\tfrac{\ell}{2}+\tilde\rho(\gamma,e'),\right)\times \lbrace\tfrac{2\ell}{3}\rbrace\times I^{\gamma,v}_{3}$;
    \item $A_3=A^{e',\gamma}_{\text{in}}$.
\end{enumerate}

    Consider the sets $A_i$ with $i=0,1,2,3$ defined above. Then define $\nu_{A_i}=\un_2$ for all $i$ and 
    \begin{itemize}
        \item $\mu_{A_0}=\mu_{A_1}=\frac{2K{\mu_0}}{\kappa\rho(\gamma,e)}$;
        \item $\mu_{A_2}=\mu_{A_3}=\frac{2K{\mu_0}}{\kappa\rho(\gamma,e')}$.
    \end{itemize}
    We define $u(\gamma,v)$ to be the concatenation of source-sink vector fields from $A_i$ to $A_{i+1}$ with $i=0,1,2$ (see Lemma \ref{lem:parallel} and \ref{lemma:pipewidth}). 
\end{definition}
\begin{coro}\label{coro:source:sink:interchange}
 Let $u(\gamma,v)$ be the vector field given in Definition \ref{def:source:sink:interchange}. Then it satisfies the following properties
     \begin{enumerate} [label=(\text{I1.}\alph*)]
         \item $\spt(u(\gamma,v))=\bigcup_{i=0}^2\conv(A_i,A_{i+1})$;
         \medskip
         
         \item $\gamma\not=\gamma'\implies \spt(u(\gamma,v))\cap\spt(u(\gamma',v))=\emptyset$;

         \medskip
         
         \item $\frac{\ell^3}{4}\min\lbrace\rho(\gamma,e),\rho(\gamma,e')\rbrace\mathscr\leq \mathcal{L}^3(\spt(u(\gamma,v)))\leq\frac{\ell^3}{4}\max\lbrace\rho(\gamma,e),\rho(\gamma,e')\rbrace$;

         \medskip
         \item
         $\|u(\gamma,v)\|_{L^\infty}\leq 4\max_{i}\lbrace \mu_{A_i},\mu_{A_{i+1}}\rbrace\leq \max\left\lbrace\frac{8K{\mu_0}}{\kappa\rho(\gamma,e)},\frac{8K{\mu_0}}{\kappa\rho(\gamma,e')}\right\rbrace K$;
         
         \medskip
         
         \item $\|u(\gamma,v)\|^p_{L^p}\leq C^p \left(\rho(\gamma,e)^{1-p}+ \rho(\gamma,e')^{1-p}\right)
         \ell^3\frac{(K{\mu_0})^p}{\kappa^p}$;
         
         \medskip
         
         \item the time $t$  such that a particle in $A_0$ flows to $A_1$ is bounded from above and below and 
         \begin{equation*}
             t\in\left(\frac{\ell\kappa}{2K{\mu_0}}\min\lbrace \rho(\gamma,e),\rho(\gamma,e')\rbrace,\frac{\ell\kappa}{2K{\mu_0}}\max\lbrace \rho(\gamma,e),\rho(\gamma,e')\rbrace\right).
         \end{equation*}
     \end{enumerate}
    
\end{coro}
\begin{proof}
    Apply Lemma \ref{lem:parallel}, \ref{lemma:pipewidth} using Definition \ref{def:source:sink:interchange}. Observe that the flows do not intersect due to the choice of the lexicographic order, i.e $\gamma \prec_e \gamma' \Leftrightarrow \gamma \prec_{e'} \gamma'$ and therefore $h(\gamma,e) < h(\gamma',e) \Leftrightarrow h(\gamma,e') < h (\gamma',e')$.
\end{proof}

\medskip

In the second case fix a vertex $v$ and consider without loss of generality the two following edges $e=\lbrace v-\un_2,v\rbrace$, $e'=\lbrace v,v+\un_1\rbrace$ and assume $\gamma$ visits $e$ and $e'$ in this order. Consider $A^{e,\gamma}_{\text{out}}=P_v +(h(\gamma,e),h(\gamma,e)+\tilde\rho(\gamma,e))\times \lbrace 0\rbrace\times I^{\gamma,v}_{3}$ and $A^{e',\gamma}_{\text{in}}$ (cf. Figure \ref{fig:interchange}).
We consider the following 
\begin{definition}\label{def:source:sink:interchange:corner}
Let $v$, $e$, $e'$ and $\gamma$ be as previously specified. Define the following 2D surfaces/checkpoints as follows.
\begin{enumerate} [label=(\roman*)]
\item $A_0=A^{e,\gamma}_{\text{out}}$;
    \item $A_1=P_v + (h(\gamma,e),h(\gamma,e)+\tilde\rho(\gamma,e))\times \left\lbrace \tfrac{\ell}{4}\right\rbrace\times I^{\gamma,v}_{3}$;
    \item $A_2=P_v + \left\lbrace \ell \right\rbrace\times(\ell-h(\gamma,e)-\tilde\rho(\gamma,e),\ell-h(\gamma,e))\times I^{\gamma,v}_{3}$;
    \item $A_3=P_v + \left\lbrace\tfrac{5\ell}{6}\right\rbrace \times(\tfrac{3\ell}{4}-h(\gamma,e)-\tilde\rho(\gamma,e)+h(\gamma,e'),\tfrac{3\ell}{4}-h(\gamma,e)+h(\gamma,e'))\times I^{\gamma,v}_{3}$;
    \item $A_4=P_v + \left\lbrace\tfrac{11\ell}{12}\right\rbrace \times(\tfrac{3\ell}{4}-h(\gamma,e)-\tilde\rho(\gamma,e' )+h(\gamma,e'),\tfrac{3\ell}{4}-h(\gamma,e)+h(\gamma,e'))\times I^{\gamma,v}_{3}$;
    \item $A_5=A^{e',\gamma}_\text{in}=P_v + \left\lbrace\ell\right\rbrace \times(h(\gamma,e'),h(\gamma,e')+\rho(\gamma,e'))\times I^{\gamma,v}_{3}$.
\end{enumerate}
We further define $\nu_{A_0}=\nu_{A_1}= \un_2$, $\nu_{A_2}=\ldots= \nu_{A_5}=\un_1$ and
\[
\mu_{A_1}= \ldots = \mu_{A_3} =\frac{2K {\mu_0}}{\kappa \rho(\gamma,e)}, \quad \mu_{A_4} = \mu_{A_5} = \frac{2K {\mu_0}}{\kappa \rho(\gamma,e')}.
\]
We define $u(\gamma,v)$ to be the concatenation of source-sink vector fields from $A_i$ to $A_{i+1}$ with $i=0,\ldots,5$ (see Lemma \ref{lem:parallel}, \ref{lemma:pipewidth} and \ref{lemma:corner}).
\end{definition}

\begin{coro}\label{coro:source:sink:interchange:corner}
 Let $u_{v,\gamma}$ be the vector field given in Definition \ref{def:source:sink:interchange:corner}. Then it satisfies the following properties
     \begin{enumerate}  [label=(\text{I2.}\alph*)]
         \item $\spt(u(v,\gamma))=\bigcup_{i=0,i\not=1}^4\conv(A_i,A_{i+1})\cup \spt(u_{\mathrm{corner}})$, where we indicated by $\spt(u_{\mathrm{corner}})$ the support of the vector field of Lemma \ref{lemma:corner}.  In particular let $\mathbb{P} \colon \R^3 \to \R^2$ be the projection onto the first two coordinates and let $\tilde{A}_i = \mathbb P A_i$, $i=1,2$. Define $B \subset \R^2$ by  \begin{equation*}
             B=\left\lbrace (x,y): x-(P_v)_1+ y-(P_v)_2=\ell, x\in ((P_v)_1+h(\gamma,e),(P_v)_1+h(\gamma,e)+\tilde\rho(\gamma,e))\right\rbrace.
         \end{equation*}
         Then 
           \begin{equation*}
            \spt(u_{\mathrm{corner}})\subset \left[\conv(\tilde A_1,B)\cup \conv(B,\tilde A_2) \right] \times \tilde I^{\gamma,v}_{3}, 
         \end{equation*}
         where we have denoted by 
         \begin{equation*}
         \tilde I^{\gamma,v}_{3}=(P_v)_3 \lbrace c_{th}\ell(\text{Lev}(\gamma,v)-1),c_{th}\ell\text{Lev}(\gamma,v)\rbrace.
         \end{equation*}
         % \textcolor{blue}{if we call 
         % \begin{equation*}
         %     \tilde A_1=((P_v)_1+h(\gamma,e),(P_v)_1+h(\gamma,e)+\tilde\rho(\gamma,e))\times \left\lbrace \left((P_v)_2+\frac{\ell}{4}\right)\right\rbrace,
         % \end{equation*}
         %  \begin{equation*}
         %     \tilde A_2 = \left\lbrace(P_v)_1+\frac{3\ell}{4}\right\rbrace\times((P_v)_2+\ell-h(\gamma,e)-\tilde\rho(\gamma,e),(P_v)_2+\ell-h(\gamma,e))
         % \end{equation*}}
         % and the set 
         % \begin{equation*}
         %     B=\left\lbrace (x,y): x-(P_v)_1+ y-(P_v)_2=\ell, x\in ((P_v)_1+h(\gamma,e),(P_v)_1+h(\gamma,e)+\tilde\rho(\gamma,e)\right\rbrace,
         % \end{equation*}

         \item $\gamma\not=\gamma'\implies \spt(u(v,\gamma))\cap\spt(u(v,\gamma'))=\emptyset$;

         \medskip
         
         \item $c_{th}\frac{\ell^3}{4}\min\lbrace\rho(\gamma,e),\rho(\gamma,e')\rbrace\lesssim \mathcal{L}^3(\spt(u_{v,\gamma}))\lesssim c_{th}\frac{\ell^3}{4}\max\lbrace\rho(\gamma,e),\rho(\gamma,e')\rbrace$;

         \medskip
         \item
         $\|u(v,\gamma)\|_{L^\infty}\leq \max\lbrace \mu_{A_0},\mu_{A_5}\rbrace (1+8c_{th}^2)\leq 2 \cdot C \max \left\lbrace\left(\frac{2K{\mu_0}}{\kappa\rho(\gamma,e)}\right),\left(\frac{2K{\mu_0}}{\kappa\rho(\gamma,e')}\right)\right\rbrace$;
         
         \medskip
         
         \item $\|u(v,\gamma)\|^p_{L^p}\leq C_p \ell^3\max \lbrace \rho (\gamma,e)^{1-p}, \rho (\gamma,e')^{1-p} \rbrace \frac{(2{K\mu_0})^p}{\kappa^p}$;
         
         \medskip
         
         \item the time $t=t_{A_0A_5}$ (in the sense of the time of a source-sink flow)  is
         \begin{equation*}
             t\in\left(\frac{3\ell\kappa}{8K{\mu_0}}\min\lbrace \rho(\gamma,e),\rho(\gamma,e')\rbrace,\frac{5\ell\kappa}{8K{\mu_0}}\max\lbrace \rho(\gamma,e),\rho(\gamma,e')\rbrace\right).
         \end{equation*}
     \end{enumerate}
    
\end{coro}

\begin{proof}
    Apply Lemma \ref{lemma:corner} for the source-sink flow between $A_1$ and $A_2$. Furthermore,  bounds for the flow from $A_0$ and $A_1$ follow from Lemma \ref{lem:parallel}, while the flow from $A_2$ to $A_5$ via $A_3$ and $A_4$ is a rescaled version of the flow of Corollary \ref{coro:source:sink:interchange}. 
    The $L^{\infty}$ and $L^p$ bound then follow from the bounds given by \ref{lem:parallel}, \ref{lemma:corner} and \ref{lemma:pipewidth}; for the $L^p$ bound of the vector field connecting $A_1$ and $A_2$ we use the trivial bound
    \[
    \Vert u_{\mathrm{corner}} \Vert_{L^p}^p \leq \mathcal{L}^3(\spt(u_{\mathrm{corner}})) \cdot \Vert u_{\mathrm{corner}} \Vert_{L^{\infty}}^p \leq (c_{\mathrm{th}} + 8 c_{\mathrm{th}}^3)^p \ell^3 \max\{ \mu_A,\mu_B\}.
    \]
        Further observe that due to the choice of lexicographic order for edges flows of different $\gamma$ do not intersect. In particular, $\gamma \prec_e \gamma'$ if and only if $\gamma' \prec_{e'} \gamma$ and therefore $h(\gamma,e) < h(\gamma',e)$ if and only if $h(\gamma',e') < h(\gamma,e')$.
\end{proof}

\begin{figure}
    \centering
    \includegraphics[width=0.7 \textwidth]{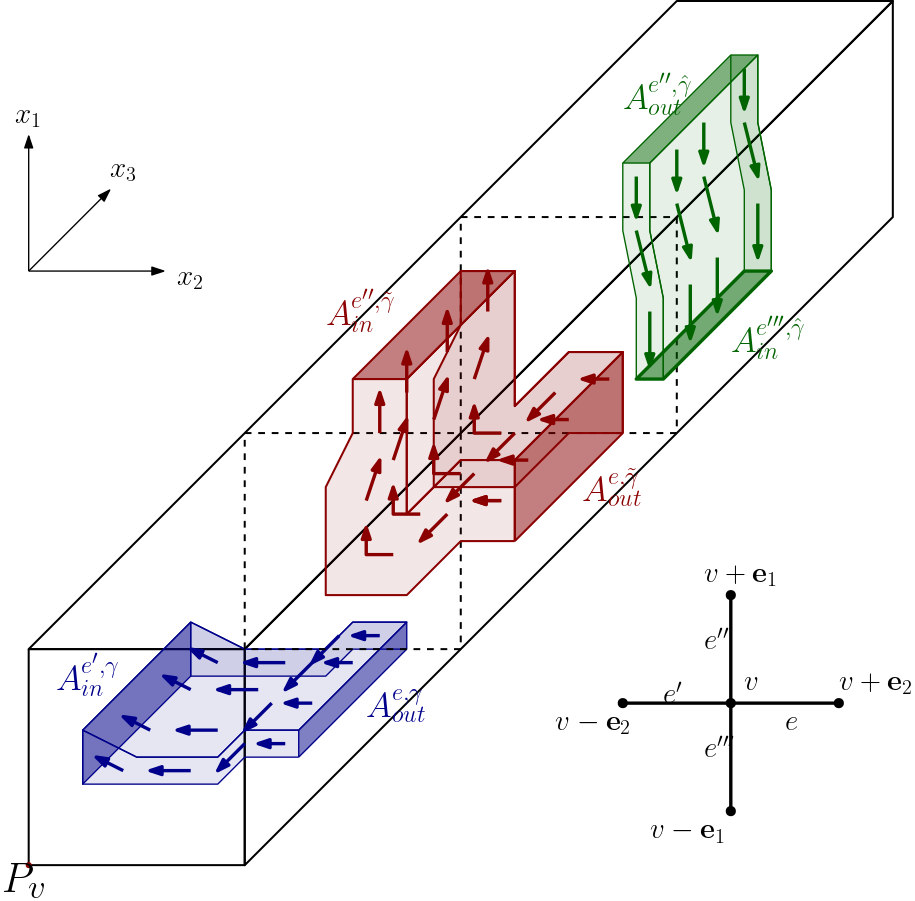}
    \caption{A schematic (not up to scale) display of the construction undertaken for the interchange in different levels. Flows in different levels of the interchange do not intersect by construction. The fact that 'parallel' flows inside one level do not touch is taken care of by the lexicographic order of paths.}
    \label{fig:interchange}
\end{figure}
\subsection{Highway entrances and exits} \label{sec:entrexit}
We have now constructed a flow inside the cubes with some in-flow (at $G_{in}$) and out-flow (at $R_{out}$). Moreover, we constructed the vector field $u$ inside the discrete network. It remains to connect the flow in and out of the cube to the discrete network.

For this we again construct some source- sink flows by concatenation of the special flows from Lemma \ref{lem:parallel}--\ref{lemma:corner}. In particular, if $v \in V$ and $\gamma$ is either starting or ending in $V$, define vector fiels as follows.
\begin{itemize}
\item If $\gamma$ is starting in $v$ we construct a vector field that gives a source-sink flow from $R_{out}$ to some surface $A^v_{entr} \subset \partial R_v$ and, after this, construct a flow from $A^v_{entr}$ to $A^{e,\gamma}_{in}$ (if $e \in E(\gamma)$ is the first edge in the path);
\item If $\gamma$ is ending in $v$ we construct a vector field that gives a source-sink flow from $A^{e,\gamma}_{out}$ to some surface $A^v_{exit} \subset \partial R_v$ and, after this, construct a flow from $A^v_{exit}$ to $G_{out}$.
\end{itemize}
We start with the definition of $A^v_{entr}$ and the flow from $G_{out}$ to $A^v_{entr}$.

\begin{definition} \label{def:entrance:1}
Let $v \in V$, $\gamma$ be a path that starts in $v$ such that $\mathrm{Lev}(\gamma,e) =13$. Define the surface $A^v_{entr}$
\begin{equation}\label{def:entrance:eq1}
A^v_{entr} = P_v+ ( \ell/4,\ell/4 +\lambda/2) \times \{0\} \times I^{\gamma,v}_3.
\end{equation}
Let now
\begin{align*}
    A_0 &= R_{out}, \\
    A_1&= (N^{-1}v_1 +\ell, N^{-1} v_1+\ell+ \lambda/2) \times \{(P_v)_2-2\ell\} \times (N^{-1} v_3, N^{-1} v_3+\ell), \\
    A_2&= (N^{-1}v_1 +\ell, N^{-1} v_1+\ell+ \lambda/2) \times \{(P_v)_2-\ell\} \times I^{\gamma,v}_3, \\
    A_3&= A^v_{entr}.
\end{align*}
Define $u_{entr}$ as the concatenation of parallel source-sink flows from $A_{i}$ to $A_{i+1}$, $i=0,1,2$ from Lemma \ref{lem:parallel} with
\[
\nu_{A_i} = \un_2 \quad \text{and} \quad \mu_{A_i} = {\mu_0} \quad \text{for } i=0,1,2,3.
\]  
\end{definition}
\begin{coro} \label{coro:entrance:0}
Let $u_{entr}$ be the vector field given by Definition \ref{def:entrance:1}. Then it satisifies the following properties:
\begin{enumerate} [label=(E1.\alph*)]
    \item $\spt(u_{entr}) = \bigcup_{i=0}^2 \conv(A_i,A_{i+1})$ and, in particular, $\spt(u_{entr})$ is disjoint from the support of any previously constructed flow;
    \item $\mathscr{L}^3(\spt(u_{entr})) \leq N^{-1} \ell^2 /\kappa$;
    \item $\Vert u_{entr} \Vert_{L^{\infty}} \leq (1+K) {\mu_0}$;
    \item $\Vert u_{entr} \Vert_{L^p}^p \leq N^{-1} \ell^2 \kappa^{-1} (1+K)^p {\mu_0}^p$;
    \item The time $t$ that the vector flow takes from $A_0$ to $A_3$ is
    \[
    t_{entr} = \tfrac{(P_2 -\ell)}{{\mu_0}} \leq \frac{K}{8 \kappa}.
    \]
\end{enumerate}  
\end{coro}

\begin{proof}
All these properties are a direct consequence of the results of Lemma \ref{lem:parallel}.
\end{proof}

We now have to define the flow from $A_{entr}$ to $A^{e,\gamma}_{in}$. We have four different possibilities for the edge $e$, also see Figure \ref{figure:cases}.
\begin{enumerate} [label=(\roman*)]
    \item If $e= \{v,v+\un_2\}$, we can use the construction of the straight interchange (cf. Definition \ref{def:source:sink:interchange} and Corollary \ref{coro:source:sink:interchange});
    \item if $e= \{v, v+ \un_1\}$ or $e = \{v,v-\un_1\}$ we can use the construction of the corner interchange (cf. Definition \ref{def:source:sink:interchange:corner} and Corollary \ref{coro:source:sink:interchange:corner});
    \item the only case not covered by previous construction is the case $e= \{v, v-\un_2\}$, where we more or less construct two corners.
\end{enumerate}
We formulate the first two cases in a corollary before we continue with the last one.
\begin{coro} \label{coro:entrance:1}
Let $v \in V$, $\gamma$ be a path starting in $v$ and let $e \neq \{v, v-\un_2\}$ be its first edge. Then there exists a source-sink flow from $A_{entr}$ to $A^{e,\gamma}_{in}$ induced by a vector field $u_{entr}^2$ with the following properties:
\begin{enumerate} [label=(E2.\alph*)]
    \item $\spt(u_{entr}^2) \subset P_v+  (0,\ell) \times (0,\ell)  \times I^{v,\gamma}_3$ and, in particular $\spt(u_{entr}^2)$ does not intersect with the support of any previously defined vector field;
    \item $\Vert u_{entr}^2 \Vert_{L^{\infty}} \leq (1+8c_{th}^2) \max \{ \mu_{A_{entr}}, \mu_{A^{e,\gamma}_{in}} \} \leq 2 \cdot CK {\mu_0} \cdot (1 + \tfrac{2}{\kappa \rho(\gamma,e)})$;
    \item $\Vert u_{entr}^2 \Vert_{L^{p}}^p \leq C_p \ell^3 K^p{\mu_0}^p \cdot (1 + 2^p\kappa^{-p})$;
    \item The time $t= t_{A_{entr},A^{e,\gamma}_{in}}$ for the flow is bounded from above by
    \[
    t= \ell \kappa \max\{ \rho(\gamma,e)K^{-1},2/\kappa\}.
    \]
\end{enumerate}
\end{coro}
\begin{proof}
    With few adjustments in the construction of $A_0,\ldots A_3$ and $A_0,...,A_5$, respectively, one may take the constructions undertaken in Definitions \ref{def:source:sink:interchange} \& \ref{def:source:sink:interchange:corner}.
\end{proof}

\FloatBarrier
\begin{figure}[!htbp]
     \centering \includegraphics[width=0.6\textwidth]{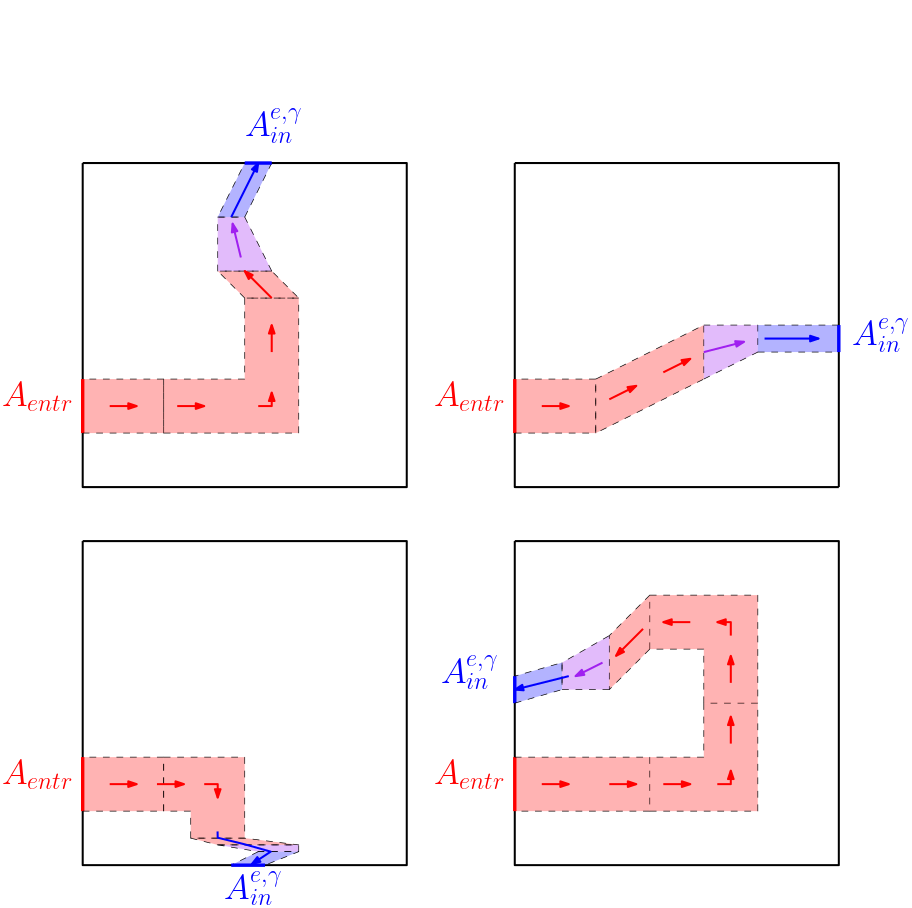}
     \caption{Top-down view (the $x_3$-coordinate is missing) of the four different cases of the entrance flow according to the position of the first edge $e$ in the path $\gamma$.}
     \label{figure:cases}
    \end{figure}
\FloatBarrier

Let now $e = \{v,v-e_2\}$. Recall the definition of $A_{entr}$ and $A^{\gamma,e}_{in}$ define the following flow:
\begin{definition} \label{def:entrance:2}
    Let $v \in V$, $\gamma$ be a path that starts in $v$. Define
    \begin{align*}
        A_0 &= A_{entr}, \\
        A_1 &= P_v + (\ell/4, \ell/4 +\lambda/2) \times \{\ell/2\}  \times I^{\gamma,v}_3, \\
        A_2 &= P_v +\{(\ell/2\} \times (3\ell/4-\lambda/2,3\ell/4) \times I^{\gamma,v}_3, \\
        A_3 &= P_v+ (3\ell/4 - \lambda/2), 3\ell/4 ) \times \{ \ell/2\}  \times I^{\gamma,v}_3, \\
        A_4 &= P_v + (h(\gamma,e),\lambda/2) \times \{\ell/2\}  \times I^{\gamma,v}_3,\\
        A_5 &= A^{\gamma,e}_{in}.
    \end{align*}
    Further set
    \[
    \nu_{A_0} = \nu_{A_1} = - \nu_{A_3} = - \nu_{A_4} = - \nu_{A_5} = \un_2, \quad \nu_{A_2} = \un_1
    \]
    and 
    \[
    \mu_{A_0} = \ldots \mu_{A_4} = {\mu_0}, \quad \mu_{A_5} = \frac{2K {\mu_0}}{\kappa \rho(\gamma,e)}.
    \]
    Define $u_{entr}^2$ to be the vector field to the concatenation of source-sink vector fields from $A_i$ to $A_{i+1}$, $i=0,\ldots,4$ that are constructed by Lemma \ref{lem:parallel} (i.e from $A_0$ to $A_1$, $A_3$ to $A_4$), Lemma \ref{lemma:pipewidth} (i.e from $A_4$ to $A_5$) and two corners, Lemma \ref{lemma:corner} (i.e. $A_1$ to $A_2$ and $A_2$ to $A_3$).
\end{definition}

\begin{coro} \label{coro:entrance2}
    Let $v \in V$, $\gamma$ be a path starting in $v$ and let $e = \{v, v-\un_2\}$ be its first edge. Then there exists a source-sink-flow from $A_{entr}$ to $A^{e,\gamma}_{in}$ induced by a vector field $u_{entr}^2$ with the following properties:
\begin{enumerate} [label=(E2.\alph*)]
    \item $\spt(u_{entr}^2) \subset P_v + (0,\ell) \times (0,\ell) \times  I^{v,\gamma}_3$ and, in particular $\spt(u_{entr}^2)$ does not intersect with the support of any previously defined vector field;
    \item $\Vert u_{entr}^2 \Vert_{L^{\infty}} \leq C \max \{ \mu_{A_{entr}}, \mu_{A^{e,\gamma}_{in}} \} \leq 2 \cdot C K{\mu_0} \cdot (1 + \tfrac{2}{\kappa \rho(\gamma,e)})$;
    \item $\Vert u_{entr}^2 \Vert_{L^{p}}^p \leq C_pK^p \ell^3 {\mu_0}^p \cdot (1 + 2^p\kappa^{-p})$;
    \item The time $t= t_{A_{entr},A^{e,\gamma}_{in}}$ for the flow is bounded from above by
    \[
    t= 2\ell \kappa \max\{ \rho(\gamma,e)K^{-1},2/\kappa\}.
    \]
\end{enumerate}
\end{coro}

We now come to the definition of the flow that connects a path ending in $v$ to the cube $\Qcal_v$. To this end, by a symmetry argument, we may construct a flow $A^{e,\gamma}_{out}$ to $A^v_{exit}$ with the same properties as the flows in Corollaries \ref{coro:entrance:1} and \ref{coro:entrance2} (except that now the direction is reversed, i.e. from $A^{e,\gamma}_{out}$ to $A^v_{exit}$ instead of from $A^v_{entr}$ to $A^{e,\gamma}_{in}$) . It remains to construct a vector field that connects $A^{v}_{exit}$ to $G_{in}$. We again refer to Figure \ref{figure:levels}.
\begin{definition} \label{def:exit:1}
Let $v \in V$, $\gamma$ be a path that ends in $v$ such that $\mathrm{Lev}(\gamma,e) =14$. Define the surface $A^v_{exit}$
\begin{equation}\label{def:exit:eq1}
A^v_{exit} = P_v +(\ell/4, \ell/4 +\lambda/2) \times \{0\} \times I^{\gamma,v}_3.
\end{equation}
Let now
\begin{align*}
    A_0 &= A^v_{exit}, \\
    A_1&= (N^{-1}v_1 +8\ell, N^{-1} v_1+8\ell+ \lambda/2) \times \{(P_v)_2-\ell\} \times I^{\gamma,v}_3, \\
    A_2&= (N^{-1}v_1 +8\ell, N^{-1} v_1+\ell+ \lambda/2) \times \{(P_v)_2-2\ell\} \times (N^{-1}v_3,N^{-1}v_3+\ell), \\
    A_3&= (N^{-1}v_1 +8\ell, N^{-1} v_1+\ell+ \lambda/2)\times \lbrace N^{-1}v_2+\frac{\ell}{2}-\frac{\lambda}{2}\rbrace \times (N^{-1}v_3,N^{-1}v_3+\ell), \\
    A_4&=\lbrace N^{-1}v_1+7.5\ell+\lambda\rbrace\times\left(N^{-1}v_2,N^{-1}v_2+\frac{\lambda}{2}\right)\times(N^{-1}v_3,N^{-1}v_3+\ell), \\
    A_5&=\lbrace N^{-1}v_1+0.5\ell-\frac{\lambda}{2}\rbrace\times\left(N^{-1}v_2,N^{-1}v_2+\frac{\lambda}{2}\right)\times (N^{-1}v_3,N^{-1}v_3+\ell), \\
    A_6&= G_{\text{in}}.
\end{align*}
Define $u_{\text{exit}}$ as the concatenation of parallel source-sink flows from $A_{i}$ to $A_{i+1}$, $i=0,1,2,\dots,5$ from Lemma \ref{lem:parallel} with
\[
\nu_{A_0} =\dots=\nu_{a_4}=-\nu_{A_6}=- \un_2 \quad,\nu_{a_5}=-\un_1\quad \text{and} \quad \mu_{A_i} = {\mu_0} \quad \text{for } i=0,1,2,\dots,6.
\]  
\end{definition}

\FloatBarrier
\begin{figure}[!htbp]
     \centering \includegraphics[width=0.8 \textwidth]{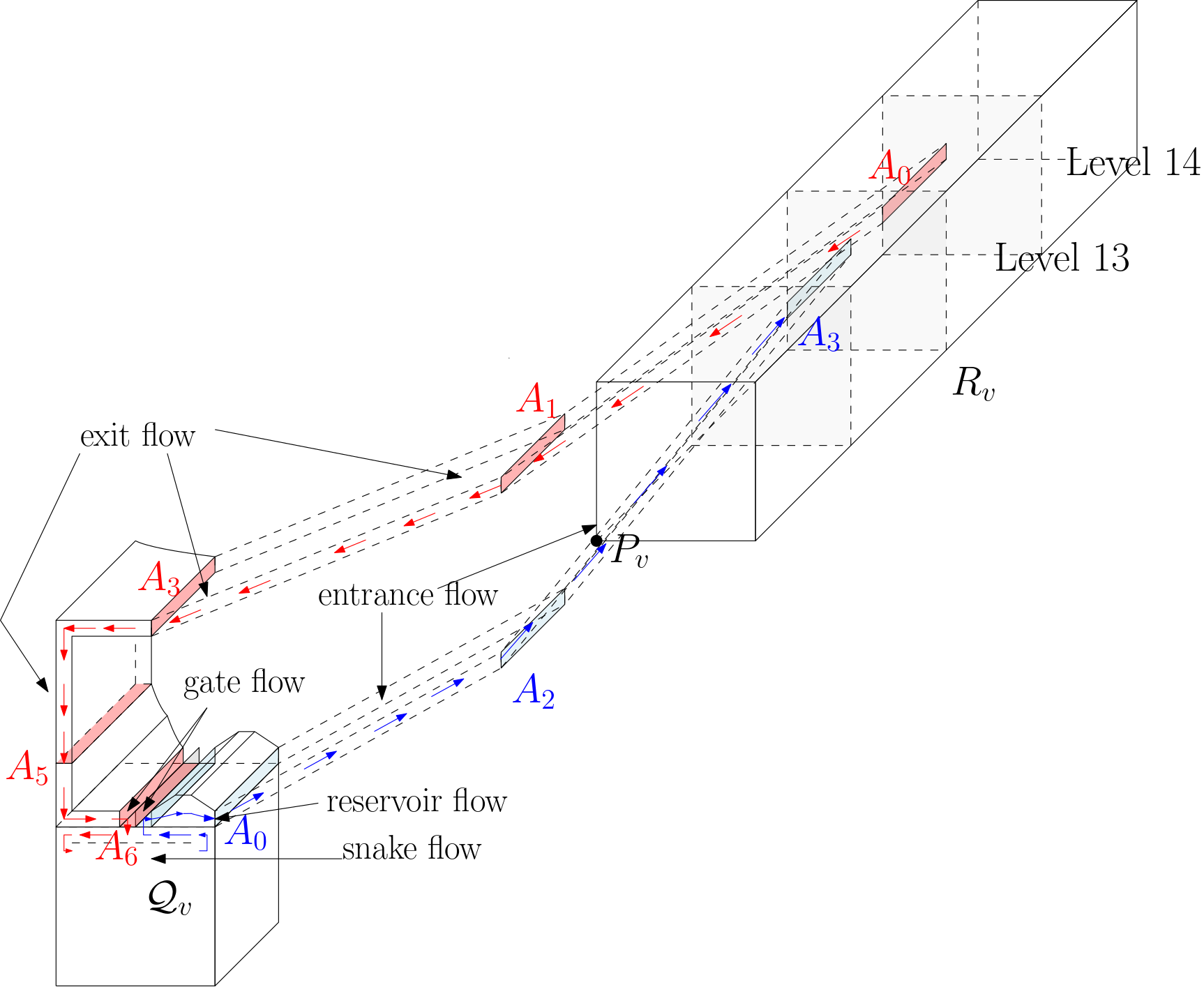}
     \caption{Slightly simplified (not all supplementary checkpoints are displayed) picture of the construction that connects the cube $\mathcal{Q}_v$ to the network.}
     \label{figure:levels}
    \end{figure}
\FloatBarrier
\begin{coro}
Let $u_{exit}$ be the vector field given by Definition \ref{def:exit:1}. Then it satisifies the following properties:
\begin{enumerate} [label=(E1.\alph*)]
    \item $\spt(u_{exit}) \subset \bigcup_{i=0}^2 \conv(A_i,A_{i+1})\cup \left( N^{-1}v_1+\ell,N^{-1}v_1+8\ell\right)\times\left(N^{-1}v_2,N^{-1}v_2+\frac{\ell}{2}-\frac{\lambda}{2}\right)\times\left(N^{-1}v_3,N^{-1}v_3+2\ell\right)$ and, in particular, $\spt(u_{entr})$ is disjoint from the support of any previously constructed flow;
    \item $\mathscr{L}^3(\spt(u_{exit})) \leq 2N^{-1} \ell^2 /\kappa$;
    \item $\Vert u_{exit} \Vert_{L^{\infty}} \leq (1+K) {\mu_0}$;
    \item $\Vert u_{exit} \Vert_{L^p}^p \leq 2 N^{-1} \ell^2 \kappa^{-1} (1+K)^p {\mu_0}^p$;
    \item The time $t$ that the vector flow takes from $A_0$ to $A_3$ is
    \[
    t_{exit} =\tfrac{P_2+7.5 \ell} {{\mu_0}}\leq  \tfrac{K}{8 \kappa}.
    \]
\end{enumerate}  
\end{coro}

\FloatBarrier
\subsection{Summary of the previous flow \& Reversing in time} \label{sec:phase2}
Let us summarise the findings of the previous subsection. We define the vector field $u_1 \colon [0,1/2] \times [0,1]^{\nu}$ piecewise on the respective supports and constant-in-time according to the constructions undertaken previously.

\begin{thm}[The properties of the vector field $u$] \label{thm:prop:u}
The vector field $$u\in L^\infty\left(\left(0,\tfrac{1}{2}\right),\BV([0,1]^\nu)\right)$$ together with the induced flow $\Psi$ satisfies the following properties
\begin{enumerate} [label=(u\arabic*)]
    \item \label{u:1} $\diverg u=0$;
    \item \label{u:2} let $t\in\left(0,\frac{1}{4}\right)$, then if $z\in (G_v)_{in}$, $\Psi(t,y)=z$, then $\Psi(t+\frac{1}{4},y)=z+\lambda \un_2\in(G_v)_{out}$;
    \item  \label{u:3} if  
 $t\in\left(0,\frac{1}{4}\right)$, then for $z\in (G_v)_{out}$, $\Psi(t,y)=z$, we have $\Psi(t+\frac{1}{4},y)=z+N^{-1}(\sigma(v)-v)-\lambda \un_2\in(G_{\sigma(v)})_{in}$;
    \item \label{u:4} in particular, if $y\in \Qcal_v\cup \conv((G_v)_{in},(G_v)_{out})$, then $\Psi(\frac{1}{2},y)=y+N^{-1}(\sigma(v)-v)$ is the translation from $\Qcal_v$ to $\Qcal_{\sigma(v)}$.
    \item \label{u:5} $u$ is constant in time;
    \item \label{u:6} $\|u(t)\|_{L^\infty}\leq C\left(K,\kappa\right)N^{-1}\sup_{\gamma,e}\lbrace \rho(\gamma,e)\rbrace^{-1}\leq C\|\Id_M-\psi_\sigma\|_{L^\infty}$ for all $t \in (0,\tfrac{1}{2})$;
     \item \label{u:7} $\|u(t)\|_{L^p}\leq C\left(K,\kappa\right)N^{-\frac{3}{p}-1}\left(\sum_{\gamma,e}\rho(\gamma,e)^{-p+1}\right)^\frac{1}{p}\leq C\|\Id_M-\psi_\sigma\|_{L^p}$ for all $t \in (0,\tfrac{1}{2})$.
\end{enumerate}
\end{thm}

\begin{proof}
The property that $\diverg u=0$ follows by construction and summing up all distributional divergences of the source-sink-flow. Also \ref{u:5} is directly given due to the construction.

We already established the second assertion in Corollary \ref{lucatoni}. For the third assertion \ref{u:3}, we summarise the time values occupied by every single step/source-sink flow: 
\begin{itemize}
    \item For the reservoir flow the time $t_{res}$ may be chosen freely (at the end of this construction), cf. Lemma \ref{lemma:reservoir};
    \item The time from the cube to the network and from the network to the cube is (cf. Section \ref{sec:entrexit}
    \begin{equation*}
        t\leq 2\frac{K}{8\kappa} +4\ell\kappa\max\lbrace \rho(\gamma,e)K^{-1},\frac{2}{\kappa}\rbrace;
    \end{equation*}
    \item By the careful choice of the constants $\kappa$, $K$ and $\mu_0$, the time on the 'highway' of an edge $e$ is (cf. Corollary \ref{cor:source:sink:highway})
    \begin{equation*}
        t\leq \frac{1}{16}\rho(\gamma,e);
        \end{equation*}
    \item The time in the interchange at a vertex $v$ and edge $e=\{v,w\}, e'=\{v,w'\} \in E(\gamma)$ is
    \begin{equation*}
        t\leq \frac{1}{8K}\max\lbrace\rho(\gamma,e),\rho(\gamma,e')\rbrace.
    \end{equation*}
\end{itemize}
Thus, the time that a particle needs from $\Q_v$ to $\Qcal_v$ is bounded from above by
\begin{equation*}
    t_{res}+\frac{\kappa^{-1} N^{-1}}{4\ell}+8\ell +\left(\frac{4\ell^2\kappa}{N^{-1}}+\frac{1}{16}+\frac{\ell}{4N^{-1}}\right)\sum_{e\in E(\gamma)}\rho(\gamma,e),
\end{equation*}
we recall that $\sum_{e\in E(\gamma)}\rho(\gamma,e)\leq 1$, so with together with the choice of $\kappa$ and $K=N^{-1}/\ell$ we have
\begin{equation*}
    4 K \leq \kappa, \quad 4 \tfrac{\ell \kappa}{K} \leq \frac{1}{32}.
\end{equation*}
It follows that with the correct choice of $t_{res}>0$ the sum of all times equals $\tfrac{1}{4}$.

It remains to show that the map $z \mapsto \Psi(\tfrac14,z)$ is the translation on the gates. To this end observe that $\Psi$ is the concatenation of multiple source-sink maps $\Psi_j \colon A_{j} \to A_{j+1}$, that are affine, orientation preserving and orientation preserving regarding \emph{only} in $x_3$ direction (i.e. if $x,y \in A_j$ and $x_3<y_3$ then also $\Psi_j(x)_3 < \Psi_j(y)_3$). Consequently, also $\Psi$ has these three properties and therefore must be the translation.

The fourth assertion follows by the two previous ones. In particular, a particle $y \in \Qcal_v$ gets transported to some $z \in (G_v)_{out}$ in some time $t^{\ast}$, then to $(G_w)_{in}$ in time $t^{\ast}+1/4$ and then finally to the translated $\Psi(1/2,z)$ in the remaining time (the snake flows for $\Qcal_v$ and $\Qcal_w$ are the same up to translation).

We will prove the seventh assertion \ref{u:7}, the sixth follows likewise. So we summarise the $L^p$ bounds found in the previous sections, namely
\begin{itemize}
    \item The $L^p$ bound for the snake flow, gate flow and reservoir flow is
    \begin{equation*}
        \|u\|^p_{L^p}\lesssim \ell^{3+p}\kappa^p;
    \end{equation*}
    \item The $L^p$ bound for the entrance/exit flow is
    \begin{equation*}
        \|u\|^p_{L^p}\lesssim c(\kappa)\ell^{3+p}\kappa^p;
    \end{equation*}
    \item The $L^p$ bound for the highway/interchange flow is
    \begin{equation*}
        \|u\|^p_{L^p}\lesssim N^{-3}c(\kappa)^p(\ell \kappa)^p\rho(\gamma,e)^{-p+1}.
    \end{equation*}
\end{itemize}
Summing up all the contributions, one gets
\begin{equation*}
    \|u\|^p_{L^p}\lesssim \sum_{v:\sigma(v)\not=v}\ell^{3+p}\kappa^p + N^{-3}\sum_{e,\gamma}\rho(\gamma,e)^{-p+1}\ell^p\lesssim N^{-3-p}\sum_{e,\gamma}\rho(\gamma,e)^{-p+1},
\end{equation*}
where we used the rough bound $\ell\leq N^{-1}$ and
$\sum_{v:\sigma(v)\not=v}1\leq \sum_{e,\gamma}\rho(\gamma,e)^{-p+1}$ (cf. equation \eqref{eq:final:section}).
\end{proof}

Theorem \ref{thm:goal} now follows by reversing the flow inside of the tubes and by the considerations of the previous lemma: 
In the previous subsections we wrote down the flow that moved each subcube $\Qcal_v$ into $\sigma(\Qcal_v)$ in time $t=1/2$, but since we need to guarantee that the flow map $\Psi$ is the identity inside outside $\cup_v \Qcal_v$, in particular inside the tubes where the flows are performed, we need to reverse the flow. In particular, we reverse the following flows:
\begin{itemize}
    \item By concatenation of the reservoir, the entrance, the exit and the highway \& interchange flows/vector fields we have a source-sink flow from $(G_v)_{out}$ to $ (G_{\sigma(v)})_{in}$. We reverse it by changing the sign and denote by $u_{rev}$ the reversed flow.
    \item We need find a reverse flow from $(G_v)_{out}$ to $(G_{v})_{in}$. For this define on for $x \in \conv((G_v)_{in},(G_v)_{out})$
    \[
        u_{rev} (x):= - \un_2 {\mu_0}.
    \]
\end{itemize}
The reversed flow $\Psi_{rev}$ induced by $u_{rev}$ has very similar properties to the forward flow $\psi$ of Theorem \ref{thm:prop:u}. In particular,
\begin{enumerate} [label=(\roman*)]
    \item It is divergence-free, constant-in-time and also satisfies the bounds \ref{u:6} and \ref{u:7};
    \item If $z \in (G_v)_{out}$, then $\psi_{rev}(t^{\ast},z) = z - \lambda \un_2 \in (G_v)_{in}$, $t^{\ast} = \tfrac{\lambda}{{\mu_0}}$;
    \item If $z \in (G_{\sigma(v)})_{in}$, then we have $\Psi_{rev}(\tfrac{1}{4},z) = z + N^{-1}(v-\sigma(v)) + \lambda \un_2 \in (G_v)_{out}$;
    \item If $z \notin \Qcal_v$ for some cube $\Qcal_v$, then $\Psi_{rev}(t^{\ast}+\tfrac{1}{4}) \circ \Psi(\tfrac{1}{2})(z) = z$.
\end{enumerate}

After rescaling in time we obtain the result of Lemma \ref{keylemma}. Theorem \ref{thm:goal} is then achieved by rescaling time by factor $(N^{-1} \ell^{-1})^3$ and 
applying Lemma \ref{keylemma} appropriately $(N^{-1} \ell^{-1})^3$-times one obtains Theorem \ref{thm:goal}. Observe that the corresponding vector field $u$ is constant-in time on $2 (N^{-1} \ell^{-1})^3$ small intervals, hence obtaining an $L^1_t L^p_x$ norm is equivalent to any $L^q_t L^p_x$ norm.

Thus, we have proven Lemma \ref{keylemma} \& Theorem \ref{thm:goal}.

%% file: sec5.tex
\section{Construction of the flow in higher dimensions} \label{sec:higherD}
In the previous subsection we constructed a vector field connecting the identity to a discrete configuration that was the identity in the third variable, i.e.
\[
\sigma(v_1,v_2,v_3) = (\tilde{\sigma}(v_1,v_2),v_3).
\]
Hence, we were able to use the solution to the 2-dimensional discrete problem, cf. Section \ref{sec:discrete:2D}.

The construction broadly remains unchanged for $3D$ and also higher dimensions. Therefore, we only outline the key changes to the construction in Section \ref{sec:4} that occur in the $3D$ case. Dimensions higher than $3$ may then be tackled with a very similar strategy.
\subsection{Main result}
We aim to prove Theorem \ref{thm:goal} in three dimensions. For this, we again subdivide any small cube $\tilde{\q}_v$ into $K^3$ small cubes and show a version of Lemma \ref{keylemma} \emph{without} the additional assumption that the permutation leaves the third coordinate unchanged.

\begin{lemma}[Partial construction of flows II] \label{keylemma:V2}
Let $\sigma \colon \{0,\ldots,N-1\}^{3} \to  \{0,\ldots,N-1\}^{3}$ be bijective. There exists a divergence-free vector field $u \in L^\infty([0,1];L^\infty(\Omega;\R^3))$ and a map $\psi \in W^{1,\infty}([0,1];L^\infty(\Omega;\Omega))$ such that 
    \begin{enumerate} [label=(\roman*)]
        \item \label{flow:1a:V2} $\partial_t \psi(t,x) = u(t,\psi(x))$;
        \item \label{flow:2a:V2} $\psi(0,\cdot) = \Id_{\Omega}$;
        \item \label{flow:3a:V2} $\psi(1,\cdot)$ = $\psi_{\sigma,\kappa}$;
        \item \label{flow:4a:V2} $\diverg_x u=0$ in the sense of distributions for a.e. time $t>0$;
    \end{enumerate}
with 
\begin{equation} \label{eq:Lpbound2:V2}
       \Vert u \Vert_{L^1([0,1];L^p(\Omega))} \leq C(p) \Vert \psi_{\sigma} - \Id_{\Omega} \Vert_{L^p(\Omega)}
   \end{equation}
and, for $\Qcal_v = N^{-1} v + [0,\ell)^3$, we have
\begin{equation} \label{eq:Qvkappa:V2}
\psi_{\sigma,\kappa} = \begin{cases}
    \id  & \text{if } x \in  Q_v \setminus \Qcal_{v}, \\
    \psi_{\sigma} & \text{if } x \in \Qcal_{v}.
\end{cases}
\end{equation}
\end{lemma}
The two main differences between above Lemma \ref{keylemma:V2} and Lemma \ref{keylemma} can be summarised as follows: \begin{itemize}
    \item Observe that in $3D$ the solution to the discretised problem in Section \ref{sec:3} is different and we need to split up the flow coming from the cube $\q_v$ into several parts that follow different paths $\gamma$.
    \item The snake flow, the gate flow, the entrance and exit flow and the flow along an edge may be left (almost) unchanged. However, observe that the construction of the interchanges needs to be approached with more care: The construction of Section \ref{sec:interchange} needs one dimension that plays the role of the level.
\end{itemize}
Without following the details of Section \ref{sec:4} gains, we shortly present the adaptations that are necessary to solve the problems outline above.
\subsection{Adaptation of the flow from the cube to the network}
There are not a lot of adaptations needed in this case: The flow from $G_{in}$ to $G_{out}$ (cf. Section \ref{sec:snake} \& \ref{sec:gate}) stays untouched.

As mentioned previously, however, there are multiple paths that start in $v$ and end in $\sigma(v)$ (i.e. are element of $\Gamma_v$ that have a weights $\omega(\gamma)$ that add up to one.
Equip those paths with an order to restore a lexicographic order as in Section \ref{sec:highway}. Let 
\[
h_{res}(\gamma) = \sum_{\gamma' \prec \gamma} \omega(\gamma').
\]
Let now $\gamma \in \Gamma_v$ be given. Then we define the new reservoir flow to the cubes $(0,\ell)^3$ as follows (also cf. Figure \ref{fig:reservoir:3D}).
\begin{definition}[Adapted reservoir flow
] Let $h \in (\lambda/2 \cdot \omega(\gamma),7 \ell \omega(\gamma))$. Define hypersurfaces $R_i = R_i[\gamma]$ as follows
\begin{itemize}
    \item $R_0 = \left(\ell + \lambda/2 h_{res}(\gamma), \lambda/2 (h_{res}(\gamma)+\omega(\gamma))\right) \times \{\ell/2+\lambda/2\} \times (0,\ell) \subset G_{out}$;
    \item $R_1 = (\ell +7 \ell h_{res}(\gamma),\ell +7 \ell h_{res}(\gamma) + \lambda/2 \omega(\gamma)) \times \{7\ell/12\} \times (0,\ell) $;
    \item $R_2 = (\ell +7 \ell h_{res}(\gamma) +7 \ell \omega(\gamma)) \times \{2\ell/3\} \times (0,\ell) $;
    \item $R_3 =  (\ell +7 \ell h_{res}(\gamma) +7 \ell \omega(\gamma)) \times \{5\ell/6\} \times (0,\ell) $;
   \item $R_4= (\ell +7 \ell h_{res}(\gamma),\ell +7 \ell h_{res}(\gamma) + \lambda/2 \omega(\gamma)) \times \{11\ell/12\} \times (0,\ell) $;
   \item $R_5=  (\ell + \lambda/2 h_{res}(\gamma), \lambda/2 (h_{res}(\gamma)+\omega(\gamma)) \times \{\ell\} \times (0,\ell)  \subset R_{out}$.
\end{itemize}
Define the vector fields $u_{Res,\gamma}$ to be the concatenation of parallel source-sink vector fields from Lemma \ref{lem:parallel} and \ref{lemma:pipewidth}.
\end{definition}
Following the lines of Lemma \ref{lemma:reservoir} we get appropriate counterparts of \ref{Re:1} -- \ref{Re:4}. In particular, up to constants we are able to get the same $L^{\infty}$ and $L^p$ bounds.

\medskip
The flow from the cube to the network remains relatively unchanged. The only issue difference is that the interchange region $R_v$ contains more levels (also cf. the following Section \ref{sec:adaptation:interchange}).

\FloatBarrier
\begin{figure}[!htbp]
     \centering \includegraphics[width=0.6
     \textwidth]{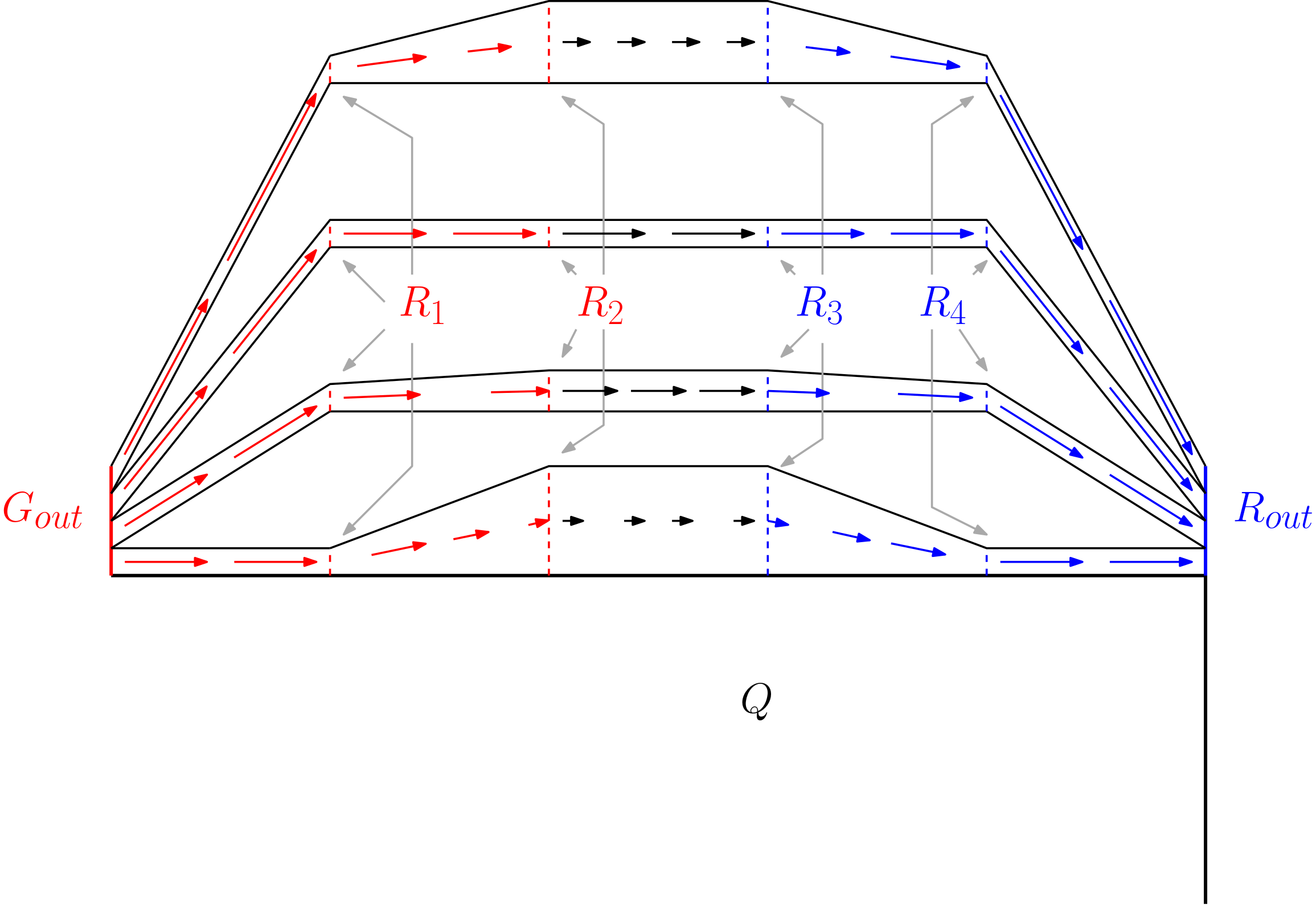}
     \caption{Top-down view of the construction of the reservoir flow, if particles move along different paths $\gamma_1$,..., $\gamma_4$: Each path is allocated enough space for a proper reservoir first (flow from $G_{out}$ to $R_1^{\gamma}$), and then we repeat the construction of the reservoir flow in the simple case (flow from $R_1^{\gamma}$ to $R_4^{\gamma}$ via $R_2^{\gamma}$ \& $R_3^{\gamma}$).}
     \label{fig:reservoir:3D}
    \end{figure} 
\FloatBarrier
\subsection{Adaptation of the grid and interchanges} \label{sec:adaptation:interchange}
\begin{figure}[!htbp]
     \centering \includegraphics[width=0.6
     \textwidth]{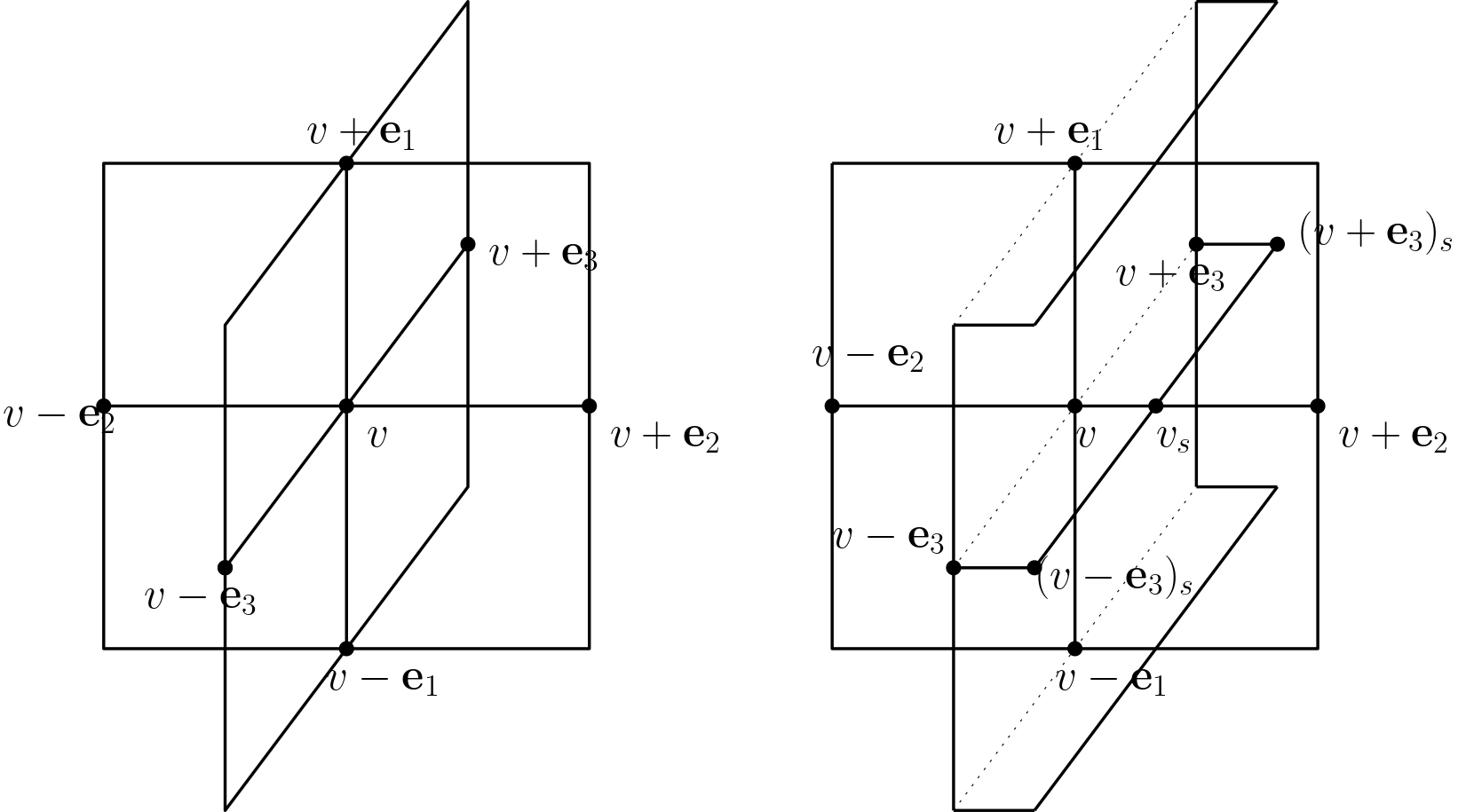}
     \caption{Visualisation of the standard grid vs a grid where the edges in $x_3$ direction are a little bit shifted. Note that at a vertex on the standard grid there are 6 directions in all 3 coordinate directions to choose from, which makes the construction of the interchange flow more complicated. In contrast, in slightly shifted grid both in $v$ and $v_s$ there are 4 possible directions in only two coordinate direction, which allows for a repeat of the construction of interchanges.} \label{network3D}
    \end{figure} 
After adjusting the flow from the cube to the network we see how we adjust the flow inside the network. As indicated by Figure \ref{network3D}, we adjust the constructed of interchanges (at grid points $P_v$) as follows: Instead of considering one interchange region $R_v$, we take the interchange region $R_v$ and an additional interchange region $\tilde{R}_v$, which is a little bit shifted by $s \un_2$.

Given a point $P \in [42c_{th} \ell, N^{-1} -42c_{th} \ell]^3$ and $v \in \{0,\ldots,N-1\}^3$, we therefore define
\begin{equation} \label{def:Pv:Pvtilde}
P_v = N^{-1} v + P, \quad \tilde{P}_v = N^{-1} v+ P + 4 \ell  \un_2, \quad \tilde{P}^2_v = N^{-1} v+ P + 8\ell \un_2
\end{equation}
and
\begin{equation} \label{def:Rv:Rvtilde}
    R_v = P_v + [0,\ell]^2 \times [0,42c_{th}\ell], \quad \tilde{R}_v =  \tilde{P}_v +  [0,42c_{th}\ell]  \times [0,\ell]^2, \quad   S_v = \tilde{P_v^2} + \{0\} \times [0,\ell]\times [0,42c_{th}\ell].
\end{equation}
Furthermore, for $\gamma \in \Gamma$ and $v \in V$ we can define a level map as in Definition \ref{def:level}
\[
Lev \colon \{(\gamma,v) \colon v \in V(\gamma) \} \to \{1,2,\ldots,42\}.
\]
Here we remark that similarly to the $2D$ case, the first $30=6 \cdot 5$ levels are labeled through the vertex that is visited by $\gamma$ before $v$ ($6$ possibilities) and after $v$ ($5$ possibilities). The other $12= 6+6$ levels are reserved for paths that either start or end in $v$ and visit some edge $e= \{v, w\}$, $w \in \{\{v \pm \un_1, v \pm \un_2, v \pm \un_3\}$ first or last, respectively.

\medskip

The flow in the network is then constructed keeping the following in mind (also cf. Figure \ref{3D:new}):
\begin{enumerate} [label=(Ad\arabic*)]
    \item \label{item1} If $e = \{v, v \pm \un_1\}$ we take a construction as in the previous section from $R_{v}$ to $R_{v \pm \un_1}$;
    \item \label{item2} if $e = \{ v, v \pm \un_3\}$ we take the construction from the previous subsection between $\tilde{R}_v$ and $\tilde{R}_{v \pm \un_3}$ with the roles of the coordinate directions $\un_1$ and $\un_3$ reversed.
    \item \label{item3} if $e = \{v, v +\un_2\}$ we slightly change the construction of the flow to get  from interchange $R_v$ to $\tilde R_v$ to $S_v$ and then, from $S_v$ onwards use the flow to $R_{v + \un_2}$
    \item \label{item4} if $\{v,v \pm \un_1\}, \{v, \pm \un_2\} \in E(\gamma)$, for the interchange we take the construction as in Section \ref{sec:interchange} in $R_v$.
    \item \label{item5} if $\{v,v \pm \un_3\}, \{v, \pm \un_2\} \in E(\gamma)$, for the interchange we take the construction as in Section \ref{sec:interchange} in $\tilde{R}_v$.
    \item \label{item6} if $\{v,v \pm \un_1\}, \{v, \pm \un_3\} \in E(\gamma)$ the interchange construction is as follows: We take an interchange in $R_v$ that changes onto $x_2$ direction, consider a flow from $R_v$ to $\tilde{R}_v$ as in \ref{item3}, and than change to $x_3$ direction (cf. Figure \ref{network3D} for visualisation on the grid).
\end{enumerate}

\begin{figure}[!htbp]
     \centering \includegraphics[width=0.6
     \textwidth]{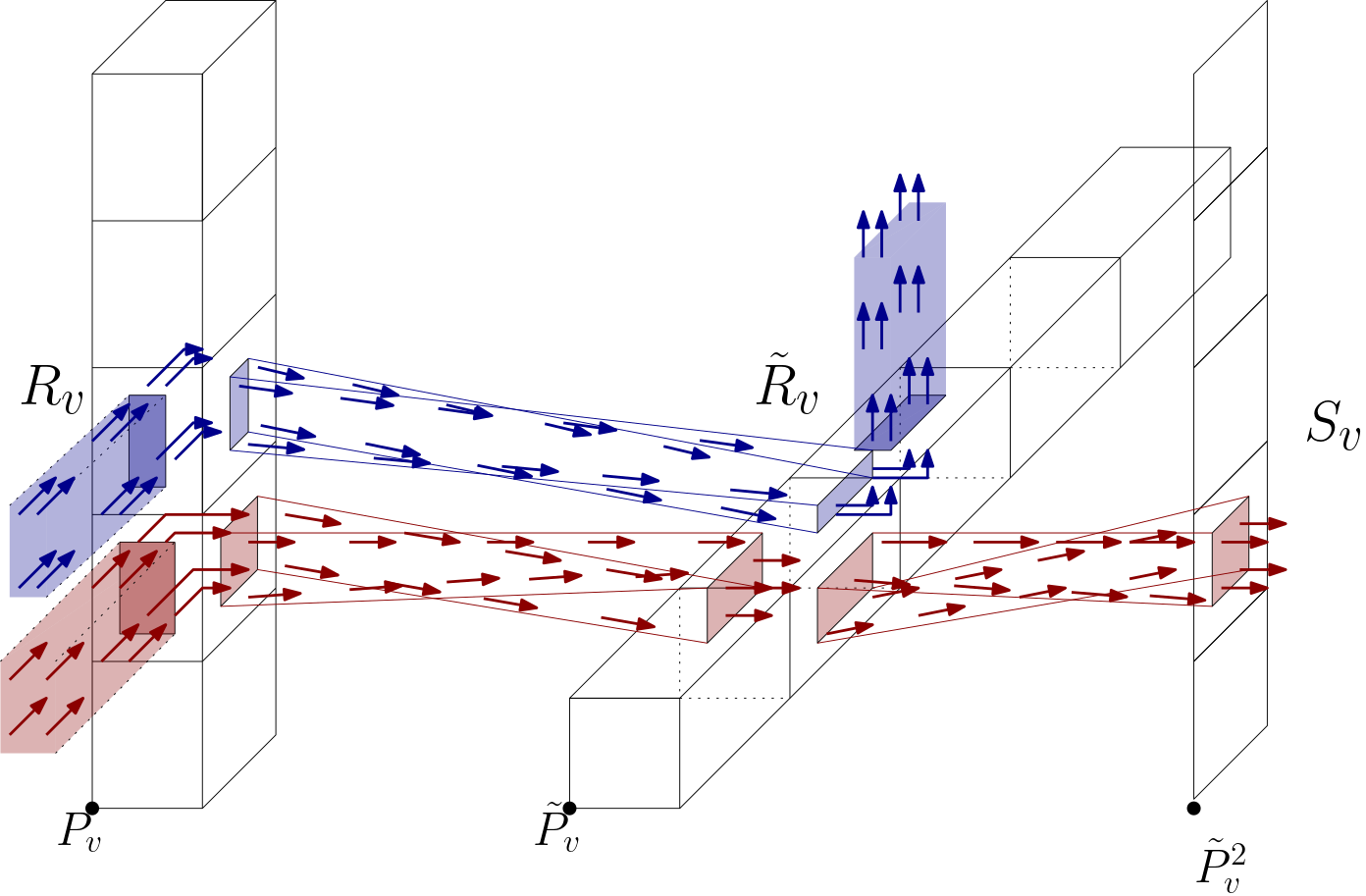}
     \caption{Schematic display of some flows in 3D. The red flow corresponds to \ref{item3}: We first flow into $R_v$, then into $\tilde{R}_v$ with additional rotation and rotate back to $S_v$. The blue flow corresponds to the flow of \ref{item6}: We first take the interchange in $R_v$ and then the interchange in $\tilde{R}_v$.}
     \label{3D:new}
    \end{figure} 

In particular, the construction of flows insider $R_v$ and $\tilde{R}_v$, i.e. \ref{item4}-\ref{item6},  might be copied from Section \ref{sec:interchange}. For \ref{item1} \& \ref{item2} take the flow/vector field from Section \ref{sec:highway}. 

We would like to point out that in all constructions, while the velocity keeps unchanged, the pipewidth is multiplied with the further weight of $\omega(\gamma,e)$ (the reader may also see this in the last construction undertaken below).
 It therefore remains to construct a suitable vector field in between $R_v$ and $\tilde{R}_v$ (for \ref{item3} and \ref{item6}) and $\tilde{R}_v$ and $S_v$ (for \ref{item3}).
\medskip

To this end, we deal with the case of \ref{item3} first (the latter is commented on in Remark \ref{rem:item6}). Suppose that $e = \{v, w\}$, $w= v + \un_2$ and define, as before, the interval

\[
I^{v,\gamma} = \left(c_{th}\ell(\mathrm{Lev}(\gamma,v)-1/2),c_{th}\ell(\mathrm{Lev}(\gamma,v)-1/2) + \ell \right).
\]
Let $\rho(\gamma,e)$ and the rescaled $\tilde{\rho}(\gamma,e) = \tfrac{\ell^2}{4N^{-1}} \rho(\gamma,e)$ be as before and let \[
h(\gamma,e) = \frac{3\ell}{8} + \sum_{\gamma' \prec_e \gamma} \omega(\gamma,e) \tilde{\rho}(\gamma,e).
\]
As before, with adjustment of the thickness let then
\[
A^{\gamma,e}_{in} = P_v + (h(\gamma,e),h(\gamma,e)+ \omega(\gamma,e) \tilde{\rho}(\gamma,e)) \times \{\ell \} \times I^{v,\gamma}
\]
and 
\[
A^{\gamma,e}_{out}= P_w + (h(\gamma,e),h(\gamma,e)+ \omega(\gamma,e) \tilde{\rho}(\gamma,e)) \times \{0 \}\times I^{w,\gamma}.
\]
We construct a flow via source-sink-flows and certain checkpoints from $A^{\gamma,e}_{in}$ to $A^{\gamma,e}_{in}$. In particular, those checkpoints are
\begin{align*}
\tilde{A}^{\gamma,e}_{out} &= P_v + I^{v,\gamma} \times \{4 \ell\} \times (h(\gamma,e),h(\gamma,e)+ \omega(\gamma,e) \tilde{\rho}(\gamma,e)) \subset \partial \tilde{R}_v, \\
\tilde{A}^{\gamma,e}_{in} & = P_v + (h(\gamma,e),h(\gamma,e)+ \omega(\gamma,e) \tilde{\rho}(\gamma,e)) \times \{8 \ell\} \times I^{v,\gamma}  \subset S_v.
\end{align*}

The definition of $\tilde{A}^{\gamma,e}_{out}$ reflects that in $\tilde{R}_v$ the roles of the $x_1$ and $x_3$ direction are reversed. To achieve nice flows between those checkpoints with disjoint trajectories we need to define another 'building block' source-sink flow.

\begin{lemma}[Rotation flow] \label{lemma:transrot}
Let $A \subset \R^3$, $A= A' \times \{0\}$ where $A' \subset B(0,r)$ is a rectangle. For an angle $\phi \in [0,2\pi)$ let $A'_{\phi}$ be the counter-clockwise rotation of $A'$ and define $B= A'_{\pi/2} \times \{1\}$. Let $\mu_0 \in \R_+$ such that $\nu_A = \nu_B = \un_3$ and $u_A=u_B= \un_3 \mu_0$. Then there exists a source-sink flow $\psi$ from $A$ to $B$ induced by a vector field $u$ such that 
\begin{enumerate} [label=(\roman*)]
\item $\Vert u \Vert_{L^{\infty}} \leq C(1+r) \mu_0$;
\item The flow transports particles in $A$ to $B$ in time $t^{\ast}= \tfrac{1}{u_0}$;
\item \label{transrot:3} $\spt(u) = \{(x,t) \in B(0,r) \times [0,1] \colon x \in A'_{t*(\pi/2)} \}$;
\item If $z=(z_1,z_2,0) \in A$ then $\psi(t^{\ast},z) = (-z_2,z_1,1)$.  
\end{enumerate}
\end{lemma}

\begin{proof}
Consider the vector field $ \tilde{u}$ that is given by
\[
\tilde{u}(x) = \mu_0 \left( \begin{array}{c} -x_2 \\ x_1 \\ 1 \end{array} \right)
\]
and define $u = \chi_X \tilde{u}$, where $X$ is the set given as support in \ref{transrot:3}. All properties (including the fact that $u$ defines a source-sink flow, cf. Definition \ref{def:sourcesink}) can be now easily verified by explicit calculation.
\end{proof}
In particular observe that if $A^1$ and $A^2$ are two disjoint surfaces supported on $B(0,r)$, then (independently of the choice of the velocities $\mu_A$ and $\mu_B$), the supports of the trajectories are disjoint.

We are now ready to define the flows/vector fields for \ref{item3}.

\begin{definition} \label{def:highwaye2}
    Let $\gamma$, $v$, $v + \un_2$, $A^{\gamma,e}_{in}$ and $A^{\gamma,e}_{out}$ be as above. Define
    \begin{align*}
        A_0& = A^{\gamma,e}_{in} = P_v + (h(\gamma,e),h(\gamma,e)+ \omega(\gamma,e) \tilde{\rho}(\gamma,e)) \times \{\ell \} \times I^{v,\gamma}, \\
        % A_1&= A^{\gamma,e}_{in} + c_{th}(\mathrm{Lev}(\gamma,e)-1/2) \ell \un_1 + \ell \un_2 \\
        % A_2 &= \tilde{A}^{\gamma,e}_{out} +c_{th} (\mathrm{Lev}(\gamma,e)-1/2) \ell \un_3 - \ell \un_2  \\  
        A_2 &=  \tilde{A}^{\gamma,e}_{out} = P_v + I^{v,\gamma} \times \{4 \ell\} \times (h(\gamma,e),h(\gamma,e)+ \omega(\gamma,e) \tilde{\rho}(\gamma,e)) \subset \partial \tilde{R}_v, \\
        A_3 &= \tilde{A}^{\gamma,e}_{out} + \ell \un_2 \subset \partial \tilde{R}_v,\\
        % A_5 &= A_2 + 3 \ell \un_2 \\
        % A_6 &= A_1 + 5 \ell \un_2 \\
        A_4 &= A_0 + 7 \ell \un_2 = \tilde{A}^{\gamma,e}_{in}, \\
        A_5 &= A^{\gamma,e}_{out}.
    \end{align*}

    \FloatBarrier
% \begin{figure}[!htbp]
%      \centering \includegraphics[width=0.6
%      \textwidth]{Rotation_interchange.png}
%      \caption{Schematic display of the flow from $R_v$ to $\tilde{R}_v$, i.e. the surfaces $A_0$-$A_3$ in Definition \ref{def:highwaye2} are visible (the rest is just a mirrored version of this). Note that the first and third flow is just a translation, while in the second flow we need to use the rotation defined in Lemma \ref{lemma:transrot}}
%      \label{fig:RvRtildev:3D}
%     \end{figure} 
\FloatBarrier
    Further set $\mu_{A_i} = \frac{2 u_0}{\kappa \rho(e)}$ (as in the definitions of \ref{sec:highway} \& \ref{sec:interchange}) and $\nu_{A_i} = \un_2$ for all $i=0,\ldots 4$. Then define $u$ to be the source-sink vector field that arises through the concatenation of source-sink vector fields from $A_i$ to $A_{i+1}$, $i=0,\ldots 3$, where we take
    \begin{itemize}
        \item the (appropriately rescaled and translated) vector fields from Lemma \ref{lemma:transrot} for the flows from $A_0$ to $A_1$ and $A_2$ to $A_3$, once with clockwise and once with counterclockwise rotation;
        \item the vector fields from Lemma \ref{lem:parallel} for the remaining source sink pairs $(A_i,A_{i+1})$.
    \end{itemize}
\end{definition}
With the same considerations as in the proof of Corollary \ref{cor:source:sink:highway} we then obtain.
\begin{coro}\label{coro:highwaye2}
Let $\tilde{u}(\gamma,e)$  be the vector field given in Definition \ref{def:highwaye2}. Then it satisfies the following properties
     \begin{enumerate} [label=(\roman*)]
         \item $\gamma\not=\gamma'\implies \spt(u(\gamma,e))\cap\spt(u(\gamma',e))=\emptyset$;
         \item $\mathscr{L}^3(\spt(u(\gamma,e)))=(N^{-1}-\ell)\ell\tilde\rho(\gamma,e)=(N^{-1}-\ell)\ell^2\frac{\rho(\gamma,e)}{4}\leq N^{-3}\rho(\gamma,e)$;
         \item $\|u(\gamma,e)\|_{L^\infty}\leq c(\kappa)u^0\rho(\gamma,e)^{-1}$;
          \item $\|u(\gamma,e)\|_{L^p}^p\leq N^{-3} c(\kappa)^p (u^0)^p \rho(\gamma,e)^{-p+1}$;
         \item the time $t_e=t_{A^{e,\gamma}_{\text{in}},A^{e,\gamma}_{\text{out}}}$ along the edge is
         \begin{equation*}
             t_e=\frac{(N^{-1}-\ell)\kappa \rho(\gamma,e)}{2u_0}.
         \end{equation*}
     \end{enumerate}
 \end{coro}
\begin{proof}
The first assertion follows from the construction (as before for Corollary \ref{cor:source:sink:highway}) \emph{and} the observation for the rotation of Lemma \ref{lemma:transrot}. The $\mathcal{L}^3$ bound on the support can be calculated using Cavalieri's principle. The $L^{\infty}$ bound follows from the the $L^{\infty}$ bounds on the flows of Lemmas \ref{lem:parallel} \& \ref{lemma:transrot} and the $L^p$ bound by the $L^{\infty}$ bound and the measure of the support. Finally, the estimate for the time follows, as the velocity along $\un_2$-direction is always $\tfrac{u_0}{\kappa \rho(\gamma,e)}$.
\end{proof}

\begin{remark} \label{rem:item6}
We shortly describe adaptions to the case where $v \in V(\gamma)$ and $e',e'' \in E(\gamma)$ where $e' = \{ v, v \pm \un_1 \}, e'' = \{v , v \pm \un_3\}$. Suppose w.l.o.g that $e'$ is visited before $e''$. In the slightly shifted discrete network, our flow will then arrive at $R_v$, then we need to go a short distance along the edge $e = \{v, v \pm \un_2\}$ to arrive at $\tilde{R}_v$ and then use the interchange to flow along $e''$. Concretely, this means that at the interchange $R_v$ we need to connect $A^{e',\gamma}_{out}$ to 
\[
	A^{e,\gamma}_{in} =   P_v + (h(\gamma,e''),h(\gamma,e'')+ \omega(\gamma,e'') \tilde{\rho}(\gamma,e)) \times \{\ell \} \times I^{v,\gamma}
\]
then connect this to
\[
	\tilde{A}^{e,\gamma}_{out} =  P_v + I^{v,\gamma} \times \{\ell \} \times (h(\gamma,e''),h(\gamma,e'')+ \omega(\gamma,e'') \tilde{\rho}(\gamma,e))
\]
in the same fashion as before. Observe that this construction generates a vector field that is disjoint to  to the construction undertaken before, as we operate on different levels. Finally, then connect $\tilde{A}^{e,\gamma}_{out}$ via the interchange at $\tilde{R}_v$ to $A^{e'',\gamma}_{in}$.
\end{remark}

%% file: sec6.tex
\section{Back from  discrete to the continuous configurations }\label{sect:from:discrete:to:continuous}
In this section we give the proof of Corollary \ref{coro:main}, that we again state below as a reminder.
\begin{coro}
    [Sharp Shnirelman's Inequality] 
    Let $\nu\geq 3$, then there exists $C>0$ such that, for every $f,g\in\mathcal{D}(M)$ it holds
\begin{equation}\label{eq:ineq:shnirelman:sharp:3}
        \text{dist}_{\mathcal{D}(M)}(f,g)\leq C\|f-g\|_{L^2(M)}.
    \end{equation}
\end{coro}
Before proving Corollary \ref{coro:main}, we recall the two following approximation results. For the first one we refer to \cite{Shnirelman,BrenierGangbo} and for the second to \cite{Lax}. 
\begin{lemma}\label{lem:approx:1}
   Let $M'=(0,1)^\nu$ with $\nu\geq 3$ and let $f:M'\rightarrow M'$ such that $f_\sharp\mathcal{L}^\nu=\mathcal{L}^\nu$, and let $p\in[1,+\infty)$. Then for every $\epsilon>0$ there exists $\bar f\in\mathcal{D}(M)$ orientation-preserving such that
        \[
        \Vert f - \bar f \Vert_{L^p} \leq \epsilon.
        \]
        Moreover, $\bar f(x)=x$ in a neighborhood of $\partial M$. 
\end{lemma}

\begin{lemma}[Approximation Lemma]\label{lem:approx:2}
   Let $f \in \Dcal(M)$, $M=[0,1]^{\nu}$ and $N \in \N$. Then there exists a permutation $\sigma$ such that for the flow map $\psi_{\sigma}$ on  the tiling $\Rcal_N$ we have for any $1 \leq p \leq \infty$
        \[
        \Vert f - \psi_{\sigma} \Vert_{L^p} \leq \sqrt{\nu}(1+\Vert f \Vert_{C^1}) N^{-1}.
        \]
    We can additionally choose $\sigma(v) = v$ for $v$ on the boundary if $1 \leq p < \infty$.     
\end{lemma}

\begin{proof}[Proof of Corollary \ref{coro:main}]
    Without loss of generality, we can assume $g=\Id_M$, as $\text{dist}_{\mathcal{D}(M)}(f,g)=\text{dist}_{\mathcal{D}(M)}( f\circ g^{-1},\Id_M)$ for $f,g\in\mathcal{D}(M)$. By Lemma \ref{lem:approx:2}, for every $\eta>0$, there exists $N\in\N$ and a permutation $\psi_\sigma\in\Dcal_N$ such that $\|\psi_\sigma-f\|_{L^2}\leq\eta$. By Theorem \ref{thm:main}, there exists a divergence-free vector field $u\in L^1_tL^2_x\cap L^1_t\BV_x$ satisfying
    \begin{equation*}
        \| u\|_{L^1_tL^2_x}\lesssim \|\psi_\sigma-\Id_M\|_{L^2_x}.
    \end{equation*}
   Consider the mollified vector field $u_\epsilon= u\ast \rho_\epsilon$, where $\rho_\epsilon$ is a mollification kernel (in order to guarantee that $u_\epsilon$ is divergence-free we can use the approximation lemma \ref{lem:approx:1}).
  Since $u_\epsilon\rightarrow u$ strongly in $L^1_tL^2_x$, then, by the Stability Theorem for RLFs (see Theorem $6.3$ in \cite{Ambrosio:Luminy}), if we call $\Psi_{\epsilon}$ the induced flow of $u_\epsilon$, then $\int_M \sup_{t\in[0,1]} \left| \Psi_\epsilon(t,x)-\Psi(t,x)\right|^2dx\to 0$ as $\epsilon\to 0$. In particular, if we denote by $\Psi_\epsilon$ the time-1 map of $u_\epsilon$, it holds 
   \begin{equation*}
       \|\Psi_\epsilon-\psi_\sigma\|_{L^2}\to 0\quad\text{as }\epsilon\to 0. 
   \end{equation*}
   Observe that $\Psi_\epsilon\in\mathcal{D}(M)$. By the triangular inequality,
   \begin{align*}
       \text{dist}_{\mathcal{D}(M)}(f,Id_M)&\leq \text{dist}_{\mathcal{D}(M)}(f,\Psi_\epsilon)+\text{dist}_{\mathcal{D}(M)}(\Psi_\epsilon,Id_M) \overset{\text{Thm \ref{thm:shnirelman}}}{\lesssim} \|f-\Psi_\epsilon\|_{L^2}^\alpha +\|u_\epsilon\|_{L^1_tL^2_x}\\& \lesssim \left(\|f-\psi_\sigma\|_{L^2}+\|\Psi_\epsilon-\psi_\sigma\|_{L^2}\right)^\alpha +\|u\|_{L^1_tL^2_x} +\| u-u_\epsilon\|_{L^1_tL^2_x}\\&\lesssim \|f-\Id_M\|_{L^2}+\|\psi_\sigma-\Id_M\|_{L^2}\lesssim \|f-\Id_M\|_{L^2},
   \end{align*}
   where $\epsilon, \eta$ have been chosen accordingly. 
\end{proof}
\begin{remark}
First of all, we mention that in the proof of Corollary \ref{coro:main} we actually construct a flow from $f$ to $g$ that is smooth in space and piecewise smooth in time (for \emph{finitely} many intervals).

If we only ask for $\phi_t \in \Dcal(M)$ and $\dot{\psi} \in L^q_t L^p_x$ then this result can also be achieved without Shnirelman's smoothening procedure present in the original inequality (cf. \cite[Lemma IV.7.19.]{Arnold:khesin}). In particular, for $g= \Id$, using Lemma \ref{lem:approx:2} we can approximate $f$ nicely by an approximate smooth diffeomorphism $\psi_{\sigma}$ and get a smooth flow from $\id$ to $\psi_{\sigma}$ in time $t=1/2$. One then repeats the procedure with $\tilde{f} = f \circ \psi_{\sigma}^{-1}$ and time rescaled to get a smooth flow from $\id$ to some better approximation $\psi_{\sigma_2}$ in time $t=3/4$. After an inductive procedure one then arrives at $f$ at time $t=1$; time derivatives of the flow are, however, quite irregular close to $t=1$.
\end{remark}

%% file: main.bbl
\begin{thebibliography}{10}

\bibitem{AMO}
R.~K. Ahuja, T.~L. Magnanti, and J.~B. Orlin.
\newblock {\em Network flows}.
\newblock Prentice Hall, Inc., Englewood Cliffs, NJ, 1993.
\newblock Theory, algorithms, and applications.

\bibitem{Ambrosio:BV}
L.~Ambrosio.
\newblock Transport equation and {C}auchy problem for {$BV$} vector fields.
\newblock {\em Invent. Math.}, 158(2):227--260, 2004.

\bibitem{AF}
L.~Ambrosio and A.~Figalli.
\newblock Geodesics in the space of measure-preserving maps and plans.
\newblock {\em Arch. Ration. Mech. Anal.}, 194(2):421--462, 2009.

\bibitem{Ambrosio:Luminy}
{Ambrosio L.}
\newblock {Lecture notes on transport equation and Cauchy problem for BV vector fields and applications.}
\newblock 2004.

\bibitem{Arnold:khesin}
V.~I. Arnold and B.~A. Khesin.
\newblock {\em Topological methods in hydrodynamics}, volume 125 of {\em Applied Mathematical Sciences}.
\newblock Springer, Cham, second edition, 2021.

\bibitem{Brenier1}
Y.~Brenier.
\newblock The least action principle and the related concept of generalized flows for incompressible perfect fluids.
\newblock {\em J. Amer. Math. Soc.}, 2(2):225--255, 1989.

\bibitem{BrenierGangbo}
Y.~Brenier and W.~Gangbo.
\newblock {$L^p$} approximation of maps by diffeomorphisms.
\newblock {\em Calc. Var. Partial Differential Equations}, 16(2):147--164, 2003.

\bibitem{Bressan}
A.~Bressan, S.~\v{C}ani\'{c}, M.~Garavello, M.~Herty, and B.~Piccoli.
\newblock Flows on networks: recent results and perspectives.
\newblock {\em EMS Surv. Math. Sci.}, 1(1):47--111, 2014.

\bibitem{CR7}
P.~Christiano, J.~A. Kelner, A.~Mpolhkadry, D.~A. Spielman, and S.-H. Teng.
\newblock Electrical flows, {L}aplacian systems, and faster approximation of maximum flow in undirected graphs.
\newblock In {\em S{TOC}'11---{P}roceedings of the 43rd {ACM} {S}ymposium on {T}heory of {C}omputing}, pages 273--281. ACM, New York, 2011.

\bibitem{collins}
M.~Collins, L.~Cooper, R.~Helgason, J.~Kennington, and L.~LeBlanc.
\newblock Solving the pipe network analysis problem using optimization techniques.
\newblock {\em Management Sci.}, 24(7):747--760, 1977/78.

\bibitem{Cross}
H.~Cross.
\newblock Analysis of flow in networks of conduits or conductors.
\newblock 1936.

\bibitem{DrivasElgindi}
T.~D. Drivas and T.~M. Elgindi.
\newblock Singularity formation in the incompressible {E}uler equation in finite and infinite time.
\newblock {\em EMS Surv. Math. Sci.}, 10(1):1--100, 2023.

\bibitem{FF}
L.~R. Ford, Jr. and D.~R. Fulkerson.
\newblock Maximal flow through a network.
\newblock {\em Canadian J. Math.}, 8:399--404, 1956.

\bibitem{FF2}
L.~R. Ford, Jr. and D.~R. Fulkerson.
\newblock {\em Flows in networks}.
\newblock Princeton University Press, Princeton, NJ, 1962.

\bibitem{Grafakos}
L.~Grafakos.
\newblock {\em Classical {F}ourier analysis}, volume 249 of {\em Graduate Texts in Mathematics}.
\newblock Springer, New York, second edition, 2008.

\bibitem{GKMN}
S.~Gupta, A.~Khodabakhsh, H.~Mortagy, and E.~Nikolova.
\newblock Electrical flows over spanning trees.
\newblock {\em Math. Program.}, 196(1-2):479--519, 2022.

\bibitem{KV}
B.~Korte and J.~Vygen.
\newblock {\em Combinatorial optimization}, volume~21 of {\em Algorithms and Combinatorics}.
\newblock Springer, Heidelberg, fifth edition, 2012.
\newblock Theory and algorithms.

\bibitem{Lax}
P.~D. Lax.
\newblock Approximation of meausre preserving transformations.
\newblock {\em Comm. Pure Appl. Math.}, 24:133--135, 1971.

\bibitem{Moser}
J.~Moser.
\newblock On the volume elements on a manifold.
\newblock {\em Trans. Amer. Math. Soc.}, 120:286--294, 1965.

\bibitem{Shnirelman}
A.~I. Shnirelman.
\newblock The geometry of the group of diffeomorphisms and the dynamics of an ideal incompressible fluid.
\newblock {\em Mat. Sb. (N.S.)}, 128(170)(1):82--109, 144, 1985.

\bibitem{Shnirelman2}
A.~I. Shnirelman.
\newblock Generalized fluid flows, their approximation and applications.
\newblock {\em Geom. Funct. Anal.}, 4(5):586--620, 1994.

\bibitem{DistPermutations}
M.~Zaefferer, J.~Stork, and T.~Bartz-Beielstein.
\newblock Distance measures for permutations in combinatorial efficient global optimization.
\newblock In T.~Bartz-Beielstein, J.~Branke, B.~Filipi{\v{c}}, and J.~Smith, editors, {\em Parallel Problem Solving from Nature -- PPSN XIII}, pages 373--383, Cham, 2014. Springer International Publishing.

\bibitem{Zizza24}
M.~Zizza.
\newblock An alternative approach to the discrete {S}hnirelman's inequality.
\newblock {\em arXiv preprint: https://arxiv.org/abs/2407.09377}, 2024.

\end{thebibliography}
